\documentclass[aop]{imsart}

\usepackage{euscript,amsfonts,amssymb,amsmath,amscd,amsthm,enumerate,hyperref}

\usepackage{comment}

\usepackage{graphicx}

\usepackage{color}

\usepackage{bbm}

\newtheorem{theorem}{Theorem}[section]
\newtheorem*{theorem*}{Theorem}

\newtheorem{proposition}[theorem]{Proposition}

\newtheorem{claim}[theorem]{Claim}
\newtheorem{corollary}[theorem]{Corollary}

\theoremstyle{remark}
\newtheorem{remark}[theorem]{Remark}

\newcommand{\eps}{\varepsilon}

\newcommand{\E}{\mathbb E}

\newcommand{\eq}{\begin{equation}}
\newcommand{\en}{\end{equation}}

\newcommand{\ii}{\mathbf i}

\newcommand{\q}{\mathfrak q}

\renewcommand{\O}{\mathcal O}
\newcommand{\s}{\mathfrak s}

\newcommand{\h}{\mathbf h}

\newcommand{\R}{\mathcal R}
\newcommand{\Rd}{\mathcal R^d}

\renewcommand{\S}{\mathbb S}

\begin{document}

\begin{frontmatter}

\title{A stochastic telegraph equation from the six-vertex model}
\runtitle{Telegraph from the six-vertex model}

\begin{aug}

\author{Alexei Borodin\thanksref{t1,m1,m2}}\and
\author{Vadim Gorin\thanksref{t2,m1,m2}}

\thankstext{t1}{MIT, Department of Mathematics,  77 Massachusetts Avenue,
Cambridge, MA 02139. \url{borodin@math.mit.edu}}
\thankstext{t2}{MIT, Department of Mathematics,  77 Massachusetts Avenue,
Cambridge, MA 02139. \url{vadicgor@gmail.com}}

\affiliation{Massachusetts Institute of Technology\thanksmark{m1} and Institute for Information Transmission Problems\thanksmark{m2}}

\runauthor{Alexei Borodin and Vadim Gorin}

\end{aug}

\begin{abstract}

A stochastic telegraph equation is defined by adding a random inhomogeneity to the classical (second
order linear hyperbolic) telegraph differential equation. The inhomogeneities we consider are
proportional to the two-dimensional white noise, and solutions to our equation are
two-dimensional random Gaussian fields. We show that such fields arise naturally as asymptotic
fluctuations of the height function in a certain limit regime of the stochastic six vertex model
in a quadrant. The corresponding law of large numbers -- the limit shape of the height function
-- is described by the (deterministic) homogeneous telegraph equation.

\end{abstract}

%\textcolor{green}{[Earlier we wanted to mention connection to Schur processes and GFF-type global
%asymptotics through the invariance of the observable. See e-mails March 15, 2017]}


\end{frontmatter}

\section{Introduction}

\subsection{Preface}

The central object of this work is a second order inhomogeneous linear differential equation
\begin{equation}
\label{eq_intro_telegraph}
 f_{XY}(X,Y)+\beta_1 f_Y(X,Y) + \beta_2 f_X(X,Y)=u(X,Y), \qquad x,y\ge 0,
\end{equation}
on an unknown function $f(X,Y)$ with given right-hand side $u(X,Y)$ and constants
$\beta_1,\beta_2\in\mathbb R$. The equation \eqref{eq_intro_telegraph} is known (in equivalent
forms obtained by multiplying the unknown function $f$ with $\exp(aX+bY)$) as the telegraph equation or the Klein-Gordon equation.

We will be particularly interested in the case when the inhomogeneity $u(X,Y)$ is proportional to
the two-dimensional white noise $\eta$,
\begin{equation}
\label{eq_intro_stoch}
 u(X,Y)=v(X,Y)\, \eta,
\end{equation}
where the prefactor $v(X,Y)$ will be made explicit later. We call \eqref{eq_intro_telegraph},
\eqref{eq_intro_stoch} the \emph{stochastic telegraph equation}.

The deterministic equation \eqref{eq_intro_telegraph} is a classical object, see e.g.\
\cite[Chapter V]{Courant}, and its stochastic versions were intensively studied in the last 50
years. Random terms were first added to hyperbolic PDEs in \cite{Cabana}, \cite{Cairoli}, and there
have been numerous developments since then. We will not try to survey those, but let us still
mention a few. The maximum of the solution was analyzed in \cite{Orsingher_tele}. The existence,
uniqueness, and regularity of the solutions in non-linear situations are discussed in
\cite{Funaki}, \cite{Carmona_Nualart_1}, \cite{Carmona_Nualart_2}, \cite{Nualart_Tindel},
\cite{Rovira_Sole}, \cite{Mueller}. The higher-dimensional setting is considered in several
articles including \cite{Dalang_Frangos}, \cite{Dalang_Leveque}, \cite{Conus_Dalang},
\cite{Ondrejat}, \cite{M-Sole}. Significant amount of work was devoted to the design of discrete
approximation schemes and numeric algorithms, e.g., in \cite{MPW}, \cite{QS}, \cite{Walsh_approx},
\cite{KLS}.
 Further, \cite{DMT}
develops Feynman--Kac type formulas, \cite{Dalang_Muller} and \cite{CJK} study intermittency of the
solutions, and \cite{KN} deals with (non-Gaussian) L\'{e}vy noises. Stochastic hyperbolic partial
differential equations were also surveyed in \cite{Dalang_lectures}, and mentioned in textbooks
\cite{Walsh}, \cite{DPZ}.

The direction we take in the present paper appears different from any of the prior works, however.
Our interest in the stochastic telegraph equation stems from the fact that it governs the asymptotics
of the macroscopic fluctuations for a particular case of a celebrated lattice model of Statistical
Mechanics called the six-vertex model; we refer to \cite{Baxter} for general information about this
and related models.

More concretely, we deal with the \emph{stochastic six-vertex model} (as well as its deformation
-- the \emph{dynamic} six--vertex model), that was first introduced in \cite{GS} and whose asymptotic behavior
has been recently studied in \cite{BCG}, \cite{Ag}, \cite{Ag_ASEP}, \cite{CT}, \cite{RS}, \cite{BBCW}, \cite{CGST}. The model is defined in the positive quadrant
via a sequential stochastic procedure. We postpone the exact definition till the next subsection, and
for now let us just say that the configurations of the model can be viewed as collections of
lattice paths on the square grid that may touch each other but can never cross, see Figure
\ref{Fig_state_space}. These paths are further interpreted as level lines of a function $H(X,Y)$ called the \emph{the height function}.

\begin{figure}[t]
\begin{center}
%{\scalebox{0.45}{\includegraphics{Configuration_height_transpose.pdf}}} \qquad \qquad
{\scalebox{0.45}{\includegraphics{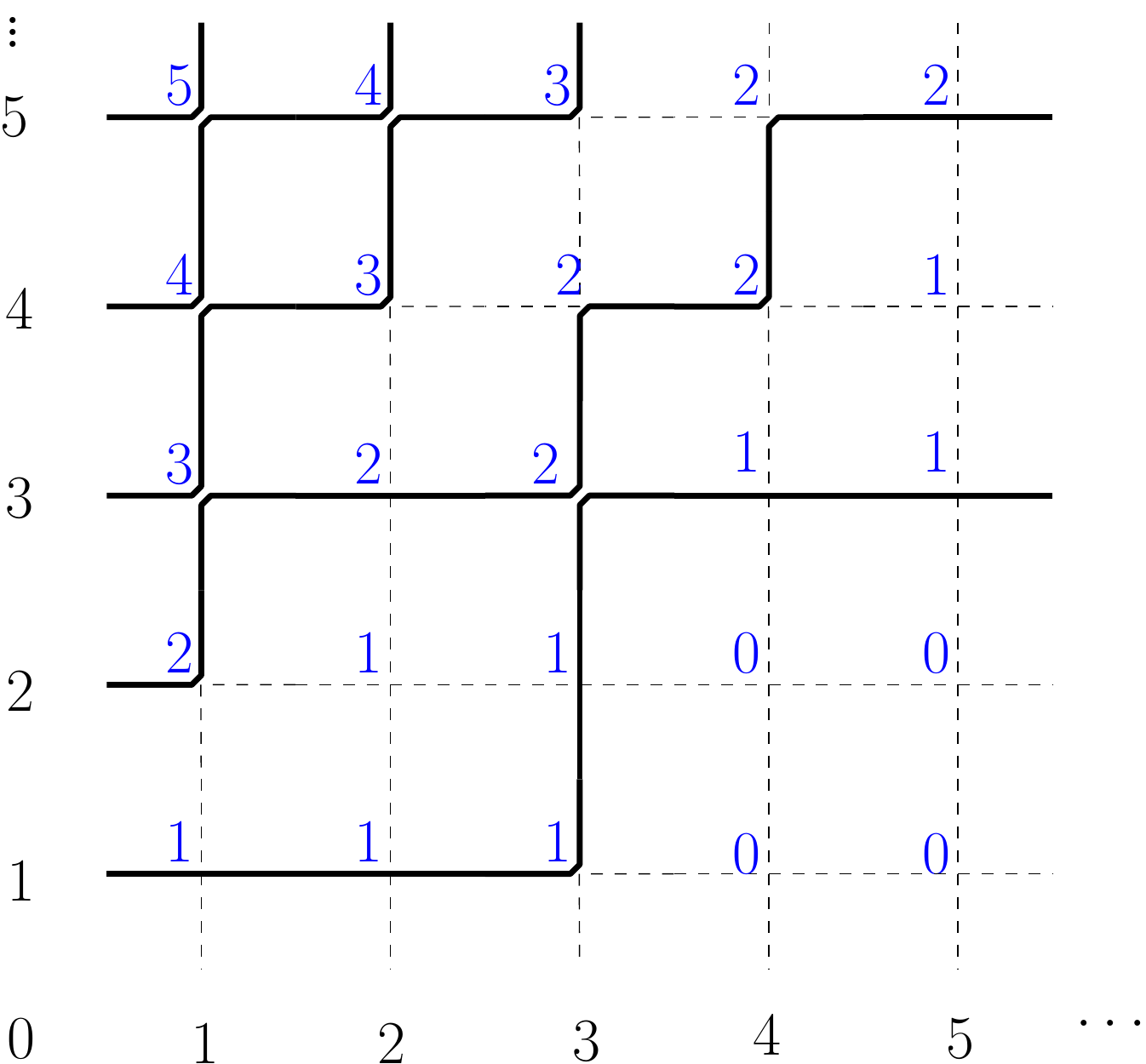}}}

 \caption{%Left panel: stochastic six--vertex model in the notations of \cite{BCG}.
 Configuration of the six-vertex model in the positive quadrant with the domain wall boundary conditions
 %he same configuration in the notations of the present article
 and the corresponding height function $H(x,y)$.
 %The panels differ by the flip around $x=y$ line.
 \label{Fig_state_space} }
\end{center}

\end{figure}

We investigate the limit regime in which the mesh size of the grid goes to $0$, and simultaneously
the turns of the paths become rare -- the weights of two of the six possible local edge configurations around a vertex converge to
zero. We find that the exponential $q^{H(X,Y)}$, where $q$ is a quantization
parameter involved in the definition of the model (that tends to 1 in our limit regime), converges to a non-random \emph{limit shape},
which solves \eqref{eq_intro_telegraph} with zero right-hand side $u(X,Y)\equiv 0$. Simultaneously,
centered and scaled fluctuations of $q^{H(X,Y)}$ converge to solutions of the stochastic
telegraph equation \eqref{eq_intro_telegraph}, \eqref{eq_intro_stoch}.

The stochastic six-vertex model and our results can be put in several contexts. The asymptotic
results of \cite{GS}, \cite{BCG}, \cite{CT}, \cite{CGST} treat the model as an interacting particle
system in the Kardar--Parisi--Zhang (KPZ) universality class \cite{KPZ}, \cite{Corwin_KPZ}. In fact,
there is a limit transition \cite{BCG}, \cite{Ag_ASEP} from the stochastic six-vertex model to a
ubiquitous member of this class -- the Asymmetric Simple Exclusion Process (ASEP). There are two
further limits from the ASEP to stochastic partial differential equations: the first one leads to a
certain Gaussian field of fluctuations \cite{DPS},\cite{DG}, while the second one leads to the KPZ equation
itself \cite{BeGi}, \cite{ACQ}, \cite{SS}, \cite{BO}. However, in both cases the resulting SPDEs are
stochastic versions of a \emph{parabolic} PDE -- the heat equation, while in our limit regime we observe a \emph{hyperbolic} PDE with a stochastic term.

While the heat equation is closely related to Markov processes (indeed, the transition probabilities of the Brownian motion are given by the heat kernel), the telegraph equation \eqref{eq_intro_telegraph} is not. It provides the simplest instance of a non-Markovian evolution, and we refer to \cite{DH} for a review of its relevance in
physics. From the point of view of the approximation by the six-vertex model, the lack of
Markov property is a corollary of the fact that for a rarely turning path, it is important to know
not only its position, but also the direction in which it currently moves. Thus, in order to create
a Markov process, one would need to extend the state space so that the direction is also recorded; see \cite{P} for nice lectures about such \emph{random evolutions}.

For the six-vertex model with fixed (i.e., not changing with the mesh size) weights, there is a general
belief that the model should develop deterministic limit shapes as the mesh size goes to zero, see
\cite{PR}, \cite{Resh_lectures}. However, mathematical understanding or description of them remains a
major open problem. For special points in the space of parameters the model is equivalent
to dimer models, where the limit shape phenomenon is well understood, see \cite{CKP}, \cite{KO}. The approach that one uses in these cases is to develop \emph{variational principles}, identifying limit shapes with maximizers of a certain integral functional of the slope of the shape. As a corollary, the limit shape solves Euler--Lagrange equations for the variational problem, and these equations ordinarily are \emph{elliptic}. From this perspective, our hyperbolic PDE \eqref{eq_intro_telegraph} seems difficult to predict.

 In the
stochastic case of the six-vertex model with fixed weights \cite{BCG} computes the limit shape for the domain wall boundary conditions, and
\cite{GS}, \cite{RS} explain that, more generally, the limit shape has to satisfy a version of the inviscid Burgers
equation. The telegraph equation can be treated as a regularization of this equation (cf.\ inviscid
vs. viscous Burgers equation); in Remark \ref{Remark_back_to_six} below, we explain how the PDE of
\cite{RS} can be recovered as a limit of \eqref{eq_intro_telegraph}. One might be surprised that
while the six-vertex hydrodynamic equation of \cite{GS}, \cite{RS} does not look linear, \eqref{eq_intro_telegraph} is. The
explanation lies in the change of the unknown function $H(X,Y)\mapsto q^{H(X,Y)}$, which linearizes the
equation. A vague analogy would be with the Hopf-Cole transform, which identifies the exponentials of
solutions of the (non-linear) KPZ equation with solutions of the (linear, with multiplicative noise) stochastic heat equation.

The same observable $q^{H(X,Y)}$ plays an important role in \cite{CGST}, where a convergence of
the stochastic six-vertex model to the KPZ equation is proven via SPDE techniques (a one-point
distributional convergence in a similar limit regime was proved in \cite[Theorem 12.3]{BO} via a free fermionic reduction of \cite{Bor16}, and an SPDE convergence in a low-density regime for higher spin stochastic
vertex models was previously proved in \cite{CT}; see the introduction to \cite{CGST} for a more
complete bibliography of related works). The limit regime of \cite{CGST} is similar
to ours in the part that both address the case of \emph{weak asymmetry} in the stochastic six-vertex model, yet the two regimes yield very different limiting SPDE's. It would be interesting to try to find an interpolation between our results and those of \cite{CGST}.

\medskip

In the rest of the introduction, we give a precise definition of the stochastic six-vertex model, describe our limit regime, and list the asymptotic results. We further outline our results on the telegraph equation and its discrete version that, to our best knowledge, appear to be new.

\subsection{The dynamic stochastic six-vertex model}

\begin{figure}[t]
\begin{center}
{\scalebox{0.5}{\includegraphics{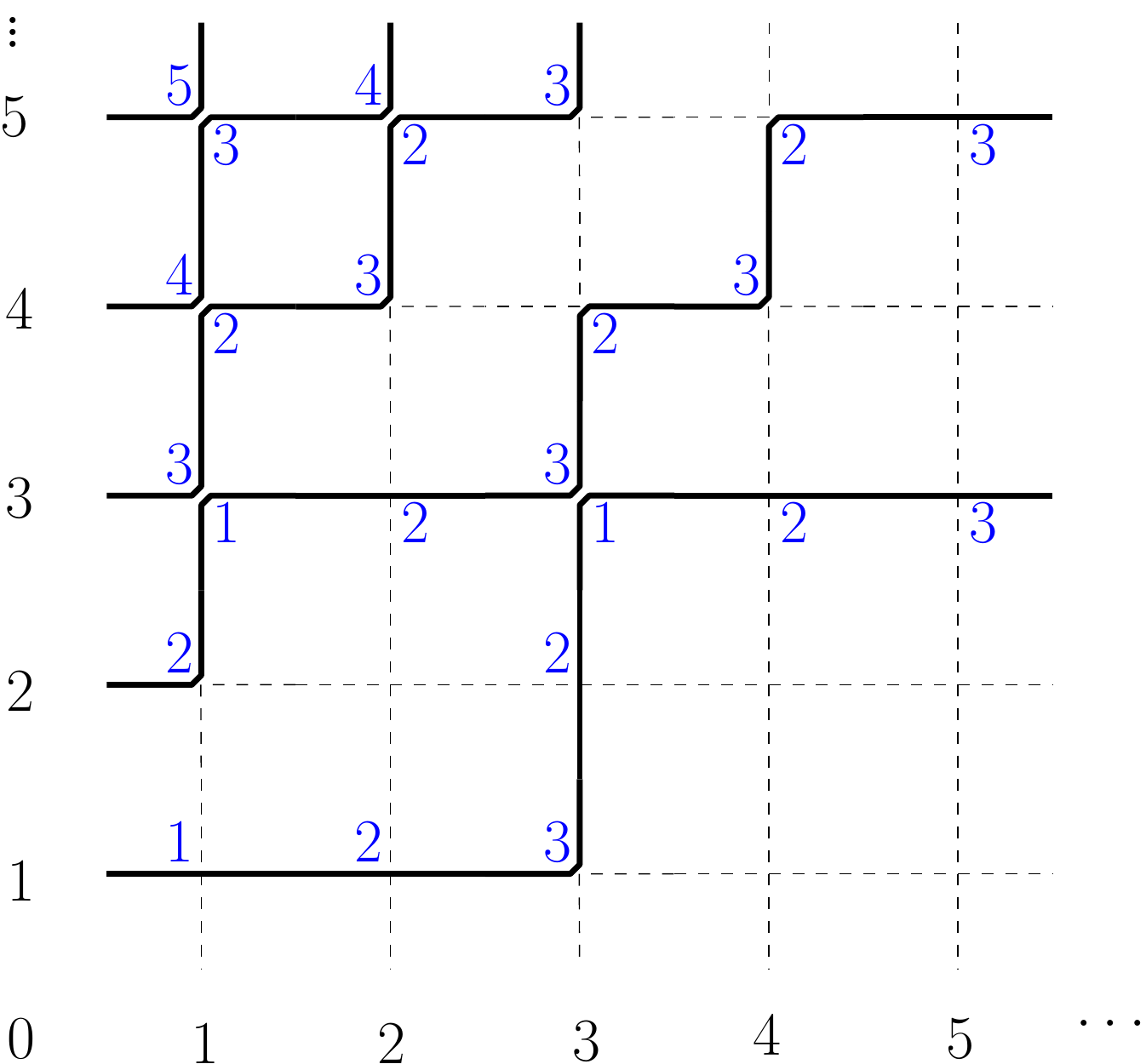}}}
 \caption{The function $d(x,y)$ defined along the paths in the dynamic stochastic six--vertex model. \label{Fig_path_height} }
\end{center}

\end{figure}

Our main object of study is the homogeneous stochastic six-vertex model of \cite{GS},\cite{BCG}
and its one-parameter deformation introduced as the \emph{dynamic stochastic six-vertex model} in
\cite{Bor_dyn}. Consider the configurations of the six-vertex model in positive quadrant. These are non-intersecting paths that are allowed to touch (see Figure
\ref{Fig_state_space}) or, equivalently, assignments of six types of vertices (see Figure \ref{Fig_weights}) to the integer points of the quadrant.

For some of our results, we focus on the domain wall boundary conditions, when the paths enter
 the quadrant through every point of its left boundary, see  Figure \ref{Fig_state_space}. For other results, we allow arbitrary deterministic boundary conditions (configurations of incoming paths) along the $x$ and $y$ axes.

A key tool of our approach is the height function $H(x,y)$. It has a local definition: We set
$H(1,0)=0$, declare that the height function is increased by $1$, $H(x,y+1)-H(x,y)=1$, whenever we
move up and the segment $[ (x-\frac12,y+\frac12),(x-\frac12,y+\frac32)]$ crosses a path, and it is decreased by
$1$, $H(x+1,y)-H(x,y)=-1$, whenever we move to the right and the segment
$[(x-\frac12,y+\frac12),(x+\frac12,y+\frac12)]$ crosses a path. The height function is constant in regions with no paths. One
way to think about the height function is that it is defined not at the integer points, but at the
half-integers -- centers of the faces of the square grid; then $H(x,y)$ corresponds to the point
$(x-\frac12,y+\frac12)$.\footnote{There is a slight asymmetry between $x$ and $y$ coordinates which we keep
to match the notations to those of previous works.} Figure \ref{Fig_state_space} shows an example.
For the domain wall boundary conditions, $H(x,y)$ counts the number of paths that pass through or
below  $(x,y)$. Formally, for $(x,y)\in\mathbb Z^2_{\ge 1}$, $H(x,y)$ is the total number of
vertices of types $II$, $III$ and $V$ at positions $(x,y')$ with $y'\le y$. We further extend
$H(x,y)$ to real $(x,y)$ first linearly in the $x$-direction, and then linearly in the $y$--direction. The
resulting function is monotone and $1$-Lipschitz in $x$ and $y$ directions.

\begin{figure}[t]
\begin{center}
{\scalebox{0.6}{\includegraphics{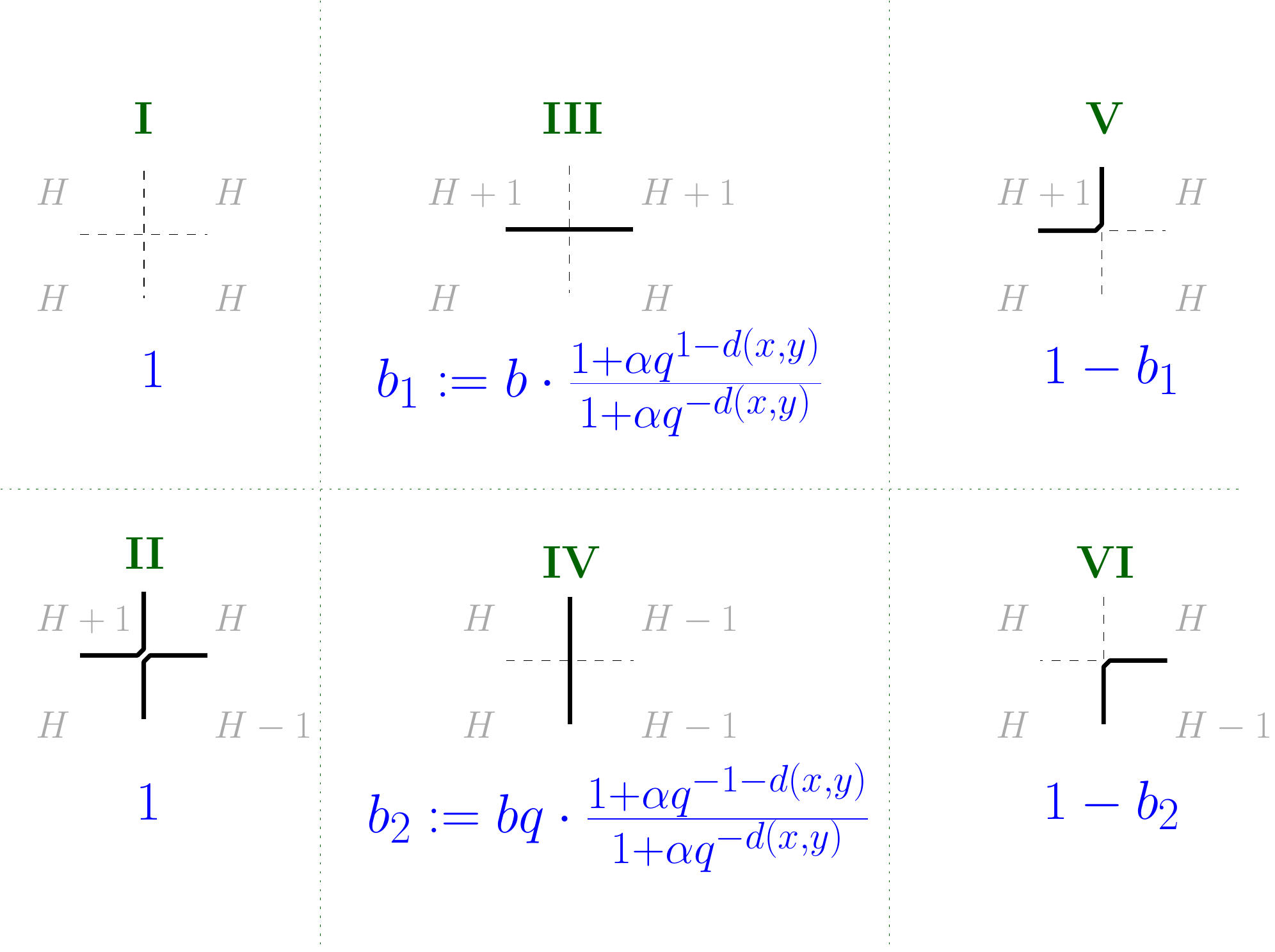}}}
 \caption{Weights of six types of vertices. Local changes of the height function $H(x,y)$ are shown in gray.
 \label{Fig_weights}}
\end{center}
\end{figure}

We also need a modified version of the height function defined through
\begin{equation}
d(x,y)=x-y-1+2H(x,y).
\end{equation}
When we move one step to the right, $d(x,y)$ increases by $1$ if we follow a path. When we move one
step up, it \emph{decreases} by $1$ if we follow a path. Therefore, along each path the height
changes piecewise-linearly, growing along the horizontal segments and decaying along the vertical ones. Note that
this rule is contradictory at points where two paths touch, as we will have two values of
$d(x,y)$ with difference $2$, cf.\  Figure \ref{Fig_path_height}. However, this is not important,
as we will never need the value of the function $d(x,y)$ at such points.

We now define the probability distribution on our path configurations. The random configuration is
obtained by a sequential construction from the bottom-left corner in the up-right direction, and
the vertices are sampled according to the probabilities in Figure \ref{Fig_weights}. The
probabilities depend on three fixed real parameters: $q>0$, $\alpha\ge 0$, $0<b<1$. The parameter $\alpha$ is sometimes referred to as the \emph{dynamic parameter}, according to the fact that for $\alpha\ne 0$ the weights of the model satisfy the \emph{dynamic}, or \emph{face} variant of the Yang-Baxter equation rather than the simpler vertex one. Following the conventional terminology of statistical physics, our probability distribution can be viewed as a stochastic (or Markovian) version of a two-dimensional exactly solvable IRF (Interaction-Round-a-Face) or SOS (Solid-On-Solid) model, cf. \cite{Bor_dyn}.
 At
$\alpha=0$, we return to the setting of the stochastic six-vertex model of \cite{BCG} with $b_1=b$,
$b_2=bq$.

\subsection{Limit regime and main asymptotic results}

In what follows, we take $L$ as a large parameter and set
\begin{equation}
\label{eq_limit_regime}
 b=\exp\left(-\frac{\beta_1}{L}\right), \qquad qb =\exp\left(-\frac{\beta_2}{L}\right),
 \quad \beta_1,\beta_2>0, \quad \beta_1\ne \beta_2.
\end{equation}
The parameter $\alpha\ge 0$ will remain fixed. In particular, if $\alpha=0$, then
$$
b_1=\exp\left(-\frac{\beta_1}{L}\right), \qquad b_2=\exp\left(-\frac{\beta_2}{L}\right).
$$
 Further, we
consider the limit $L\to\infty$, and it is sometimes convenient to use alternative parameters $\q$
and $\s$ defined by
\begin{equation}
\label{eq_limit_regime_2}
  q=\q^{1/L}, \quad  \ln(\q)=\beta_1-\beta_2, \quad
  \s=\lim_{L\to\infty} \frac{1-b}{1-bq}=\frac{\beta_1}{\beta_2}.
\end{equation}
We will sometimes switch between $\beta_1,\beta_2$ notations and $\q$, $\s$ notations to make
formulas more aesthetically pleasing. We will always assume $\beta_1\ne \beta_2$, which implies
$\q,\s\ne 1$.

\medskip

We prove the following results.

\begin{enumerate}

\item For the domain wall boundary conditions and any $\alpha\ge 0$, we develop in Theorems
    \ref{Theorem_LLN}, \ref{Theorem_CLT} the Law of Large Numbers for the height function
    $H(x,y)$ and the Central Limit Theorem for its centered and rescaled fluctuations. The
    relevant limit quantities are given as contour integrals, and the proofs are based on exact expressions for the
    expectation of shifted $q$-moments of the height function $H(X,Y)$. We rely on several
    ingredients -- contour integral expressions of \cite{Bor_dyn}, a Gaussianity lemma for random
    variables with moments given by contour integrals of \cite{BG_GFF}, and a novel
    combinatorial argument of Theorem \ref{theorem_shifted_cumulants} linking cumulants with
    their shifted versions.

\item For \emph{arbitrary} (deterministic) boundary conditions in the case $\alpha=0$, we prove in Theorem
    \ref{Theorem_LLN_general} the Law of Large Numbers by showing that $q^{H(x,y)}$ converges in
    probability to the solution of the telegraph equation \eqref{eq_intro_telegraph} with
    $u(x,y)\equiv 0$ and prescribed boundary values along the lines $x=0$ and $y=0$. The proof is based
    on a novel stochastic four point relation of Theorem \ref{Theorem_4_point} for
    $q^{H(x,y)}$. This relation does not seem to be present in the existing literature but, once written, its proof is immediate from the definition of the model. It can also be derived from the duality relations of \cite[(2.6)]{CP}, \cite[Proposition 2.6]{CT}, \cite[Corollary 3.4]{CGST}.
    We were led to this relation by \cite{W} that provided different derivations of its averaged version.

 \item For \emph{arbitrary} (deterministic) boundary conditions in the $\alpha=0$ case, we present
     the  Central Limit Theorem for $q^{H(x,y)}$ in Theorem
     \ref{Conjecture_general_CLT}. The answer is given by the stochastic telegraph equation
    \eqref{eq_intro_telegraph}, \eqref{eq_intro_stoch} with the variance of the white noise
    $v(x,y)$ being a \emph{non-linear} function of the limiting profile for $q^{H(x,y)}$ afforded by the Law
    of Large Numbers. The proof again exploits the four point relation of Theorem
    \ref{Theorem_4_point}.

\item We investigate the \emph{low density} boundary conditions (which means that there are few paths
    entering through the boundary; their locations are still deterministic, but they are changing
    as $L\to\infty$; the distinction with previous results is that in points 2 and 3 the average density of incoming paths was positive, while here it tends to 0), in the case $\alpha=0$, which has an interpretation through evolution of a family of independent persistent random walks.
     We prove in Theorem
    \ref{Theorem_LLN_CLT_low} the Law of Large Numbers and Central Limit Theorem for the properly centered and scaled $H(x,y)$. The answer is still given by the stochastic telegraph equation
    \eqref{eq_intro_telegraph}, \eqref{eq_intro_stoch}, but the variance of the white noise
    $v(x,y)$ becomes a \emph{linear} function of the limiting profile.

\end{enumerate}

 In the first version of this text the Central Limit Theorem of $(3)$ was presented as a conjecture with two heuristic arguments in favor of its validity. Later on, \cite{ST} proved the conjecture by combining the four point relation with certain new ideas. This prompted us to return to our original heuristic approaches, and we were eventually able to turn one of them into a complete proof (different from the one in \cite{ST}). It is this proof that is presented in Section \ref{Section_CLT_conjecture} below; the second heuristic approach has been moved to an appendix.

\subsection{The classical telegraph equation and its discretization} As many of our results are based on the
analysis of the telegraph equation \eqref{eq_intro_telegraph} and its discrete counterpart encoded in the
four point relation of Theorem \ref{Theorem_4_point}, we need some information about its solutions.
There is a classical part here (see, e.g., \cite{Courant}) -- existence/uniqueness of the solutions
to hyperbolic PDEs and an integral representation of the  solutions through the \emph{Riemann
function} of the equation. We review this part at the beginning of Section \ref{Section_PDEs}. We further demonstrate
in Theorem \ref{Theorem_discrete_PDE_through_integrals} that the discrete version of the telegraph
equation admits a similar theory, with the Riemann function replaced by an appropriate discrete
analogue. This greatly simplifies the proofs, as the convergence of the discretization to the
telegraph equation itself becomes a corollary of the convergence for the (explicit) Riemann functions.

\smallskip

Motivated by the fact that we obtained the telegraph equation from a stochastic system of non-intersecting paths, we further develop a theory for the representations of its solutions as path integrals. This may be viewed as an analogue of the Feynman-Kac formula for the parabolic equations. For the homogeneous equation \eqref{eq_intro_telegraph} with $u(x,y)\equiv 0$, such a theory was previously known -- \cite{Gold}, \cite{Kac}, see also \cite{P}, explain that a solution at $(x,y)$ can be represented as an expectation of the boundary data at the point where a \emph{persistent Poisson random walk} started at $(x,y)$ exits the quadrant, see Theorem \ref{Theorem_telegraph_as_expectation} for the exact statement.

For the inhomogeneous equation we find a stochastic representation (that we have not seen before) in terms of \emph{two} persistent Poisson random walks. The additional term is the integral of the right-hand side $u(X,Y)$ over the domain between two (random) paths with sign depending on which path is higher. We refer to Theorem \ref{Theorem_telegraph_as_expectation} for more details.

In addition, we develop, in Theorems \ref{Theorem_discrete_as_expectation_boundary}, \ref{Theorem_discrete_as_expectation_noise}, a stochastic representation for the solutions of the discretization of the telegraph equation. The result is similar: one needs to launch a random walk from the observation point and compute the expectation at the exit point to get the influence of the boundary data, and one needs to sum the inhomogeneity of the equation over the domain between trajectories of two random walks. The needed random walk combinatorially is the same path of the six-vertex model, but with flipped stochastic weights, as in Figure \ref{Fig_weights_rev}.

\subsection*{Acknowledgements} We are very grateful to M.~A.~Borodin for bringing the telegraph equation and its treatment in \cite{Courant}
to our attention. We would like to thank I.~Corwin and H.~Spohn for very helpful discussions,  P.~Diaconis for pointing us to the lectures \cite{P}, M.~Wheeler for the discussions which eventually led us to the discovery of the four--point relation,  H.~Shen and L.-C.~Tsai for telling us about their work \cite{ST}, and three referees for the careful proofreading of the manuscript.
Both authors were partially supported by the NSF grant DMS-1664619.
A.B.~was partially supported by the NSF grant  DMS-1607901. V.G.~was partially supported by the NEC Corporation Fund for Research in Computers and Communications and by the Sloan Research Fellowship.

\section{The domain wall boundary conditions}

\label{Section_domain_wall}

In this section we focus on the domain wall boundary conditions: the paths enter at every integer
point of the $y$--axis and no paths enter through the $x$--axis, as in Figures
\ref{Fig_state_space}, \ref{Fig_path_height}. We prove the Law of Large Numbers and the Central
Limit Theorem for the height function.

\subsection{Formulation of LLN and CLT}

\begin{theorem} \label{Theorem_LLN}
 For each $\alpha\ge 0$, in the limit regime \eqref{eq_limit_regime} we have
 $$
  \lim_{L\to\infty} \frac{1}{L}H(Lx,Ly)=\mathbf h(x,y), \qquad \text{(convergence in
 probability)}
 $$
 where $\mathbf h(x,y)$ is the only real (deterministic) solution of
 \begin{multline}
 \label{eq_LLN}
 \frac{\bigl(\q^{-\mathbf h(x,y)} \q^{y-x}+\alpha^{-1}
    \bigr)
    \bigl(\q^{\mathbf h(x,y)}-1\bigr)}{1+\alpha^{-1}}
  \\  =\frac{1}{2\pi \ii} \oint_{-1} \exp\left(
      \ln(\q)\left(- x \frac{\s z}{1+ \s z }   + y\frac{z}{1+z} \right)\right) \frac{d z}{z} ,
 \end{multline}
 with integration in positive direction around the singularity at $-1$ and avoiding the singularities at $0$ and $-\frac{1}{\s}$. At $\alpha=0$ the left--hand side of \eqref{eq_LLN} becomes $\q^{\mathbf h(x,y)}-1$.
\end{theorem}
\begin{remark}
\label{Remark_DW_in_betas}
 In terms of $\beta_1$ and $\beta_2$, the right--hand side of \eqref{eq_LLN} can be rewritten as
 \begin{equation}
 \label{eq_LLN_betas}
   \frac{1}{2\pi \ii} \oint_{-\beta_1} \exp\left(
      (\beta_1-\beta_2)\left(- x \frac{z}{\beta_2+  z }   + y\frac{z}{\beta_1+z} \right)\right) \frac{d z}{z}
 \end{equation}
 with a positively oriented integration contour encircling $z=-\beta_1$, but not $-\beta_2$ or
 $0$.
\end{remark}

\begin{proposition} \label{Proposition_q_0}
 In the setting of Theorem \ref{Theorem_LLN} with $\alpha=0$, consider the limit $\q\to 0$ with fixed value of $\s<1$. Then
 \begin{equation}
 \label{eq_BCG_LLN}
  \lim_{\q\to 0} \mathbf h(x,y)=\begin{cases}0, & \frac{x}{y}>\s^{-1},\\ \dfrac{ (\sqrt{\s
 x}-\sqrt{y})^2}{1-\s}, & \s  \le \frac{x}{y}\le \s^{-1}\\ y-x, & \frac{x}{y}<\s.  \end{cases}
 \end{equation}
\end{proposition}
Note that the right--hand side of \eqref{eq_BCG_LLN} is precisely the limit shape of the stochastic
six--vertex model in the asymptotic regime of fixed $q$ as $L\to\infty$, as obtained in
\cite[Theorem 1.1]{BCG}.

\bigskip

Let us apply the differential operator $f\mapsto f_{xy}+\beta_1 f_y+\beta_2 f_x$ to
\eqref{eq_LLN_betas}. We can differentiate under the integral sign, which gives
 \begin{multline}
 \label{eq_tele_for_contour}
   \frac{1}{2\pi \ii} \oint_{-\beta_1}\frac{d z}{z}  \exp\left(
      (\beta_1-\beta_2)\left(- x \frac{z}{\beta_2+  z }   + y\frac{z}{\beta_1+z} \right)\right)
 \\ \times \left[ -(\beta_1-\beta_2)^2 \frac{z}{\beta_1+z} \cdot \frac{z}{\beta_2+z}+\beta_1(\beta_1-\beta_2) \frac{z}{\beta_1+z}-\beta_2(\beta_1-\beta_2) \frac{z}{\beta_2+z} \right]=0.
 \end{multline}
This shows that a functional of the limit shape (which is $\q^{\h(x,y)}$ in $\alpha=0$ case and the
left-hand side of \eqref{eq_LLN} for general $\alpha$) satisfies the equation $f_{xy}+\beta_1
f_y+\beta_2 f_x=0$, which is a variant of the telegraph equation, cf.\ e.g.\ \cite{Courant}. In
Section \ref{Section_LLN_4p} we upgrade the Law of Large Numbers at $\alpha=0$ to general boundary
conditions and prove that the link to the telegraph equation persists.

\bigskip

For a point $(x,y)\in \mathbb{Z}_{>0}\times \mathbb{Z}_{>0}$ define
\begin{equation}
\label{eq_O}
 \O(x,y)=-\alpha^{-1} q^{H(x,y)} + q^{y-x+1-H(x,y)}.
\end{equation}

\begin{theorem} \label{Theorem_CLT}
 Fix $k\in\mathbb Z_{>0}$ and reals $y>0$ and $x_1\ge x_2 \ge\dots\ge
x_k>0$.  For each $\alpha\ge 0$, in the limit regime \eqref{eq_limit_regime}  the random variables
$$
 \frac{H(Lx_i,Ly)-\E H(Lx_i,Ly)}{\sqrt{L}}, \quad i=1,\dots,k,
$$
converge as $L\to\infty$ (in the sense of moments) to a centered Gaussian vector. The asymptotic
covariance is given in terms of $\O(x,y)$ by
\begin{multline}
\label{eq_covariance_integral}
 \lim_{L\to\infty} L \frac{\mathrm{Cov}(\O(Lx_1,Ly), \O(L x_2,Ly))}{(1+\alpha^{-1})^2}\\=
\frac{\ln(\q) }{(2\pi \ii)^2} \oint_{-1} \oint_{-1} \frac{z_1}{z_1-z_2}
\prod_{i=1}^{2}\left[\exp\left(
      \ln(\q)\left(- x_i \frac{\s z_i}{1+ \s z_i }   + y\frac{z_i}{1+z_i} \right)\right) \frac{d z_i}{z_i}\right]
      \\+\frac{\ln(\q)}{2\pi\ii} \oint_{-1}  \exp\left(
      \ln(\q)\left(- x_1 \frac{\s z}{1+ \s z}   + y\frac{z}{1+z} \right)\right) \frac{d z}{z}
      \\ \times  \frac{1}{1+\alpha^{-1}}\left[ \q^{y-x_2}+\alpha^{-1}+ \frac{1}{2\pi\ii} \oint_{-1}  \exp\left(
      \ln(\q)\left(- x_2 \frac{\s z}{1+ \s z}   + y\frac{z}{1+z} \right)\right) \frac{d z}{z}
     \right],
\end{multline}
where $x_1\ge x_2$, positively oriented integration contours enclose $-1$, but not $0$ or $-\frac{1}{\s}$, and
for the first integral the $z_1$--contour is inside the $z_2$--contour.
 If $\alpha=0$, then
 \begin{multline}
\label{eq_covariance_integral_zero}
 \lim_{L\to\infty} L \mathrm{Cov}(q^{H(Lx_1,Ly)}, q^{H(Lx_2,Ly)})\\=
\frac{\ln(\q)}{(2\pi \ii)^2} \oint_{-1} \oint_{-1} \frac{z_1}{z_1-z_2}
\prod_{i=1}^{2}\left[\exp\left(
      \ln(\q)\left(- x_i \frac{\s z_i}{1+ \s z_i }   + y\frac{z_i}{1+z_i} \right)\right) \frac{d z_i}{z_i}\right]
      \\+\frac{\ln(\q)}{2\pi\ii} \oint_{-1}  \exp\left(
      \ln(\q)\left(- x_1 \frac{\s z}{1+ \s z}   + y\frac{z}{1+z} \right)\right) \frac{d z}{z}, \quad x_1\ge x_2,
\end{multline}
with similar integration contours.
\end{theorem}
\begin{remark}
 Expanding
 \begin{multline*}q^{H(Lx,Ly)}=q^{\E H(Lx,Ly)} \biggl(1+\ln(q) \bigl(H(Lx,Ly)-\E H(Lx,Ly)\bigr)\\ + \bigl(\ln(q)\bigr)^2 \frac{\bigl(H(Lx,Ly)-\E H(Lx,Ly)\bigr)^2}{L^2}+\dots\biggr),
 \end{multline*}
and noticing that $\ln(q)$ is of order $L^{-1}$, one can derive the covariance of $H(Lx,Ly)$ from
that of $q^{H(Lx,Ly)}$, or from that of $\O(Lx,Ly)$. However, the resulting formulas are much
bulkier than \eqref{eq_covariance_integral}, \eqref{eq_covariance_integral_zero}, and we have not
found a good way to simplify them.
\end{remark}

At $\alpha=0$, we can generalize Theorem \ref{Theorem_CLT}: in Section
\ref{Section_CLT_conjecture} we describe its upgrade to general boundary conditions
and link it to a stochastic telegraph equation.

\medskip

In the remainder of this section we prove  Theorems \ref{Theorem_LLN}, \ref{Theorem_CLT}, and
Proposition \ref{Proposition_q_0}.

\subsection{Observables}

The asymptotic analysis of this section is based on (algebraic) results from \cite{Bor_dyn}, that
generalize those of \cite{BCG}, \cite{CP}, \cite{BP1}, \cite{BP2}; more powerful results can be found in
\cite{Ag_dynamic}.

As before, we use the notation $\O(x,y)=-\alpha^{-1} q^{H(x,y)} + q^{y-x+1-H(x,y)}$.

\begin{theorem}[{\cite[Theorem 10.1]{Bor_dyn}}] \label{Theorem_observable}
 For any fixed $y\ge 1$ and $x_1\ge x_2\ge\dots\ge x_n \in \mathbb Z_{>0}$ the expectation
 \begin{equation}
 \label{eq_obs_def}
  E_N(x_1,\dots,x_N):=\frac{1}{(-\alpha^{-1};q)_n} \E \left[ \prod_{k=1}^n \left(q^{y-x_k+1} - \alpha^{-1} q^{2k-2} -q^{k-1} \O(x_k,y) \right) \right]
 \end{equation}
 is equal to
\begin{multline}
\label{eq_Observable_contour_integral}
 \frac{q^{n(n-1)/2}}{(2\pi \ii)^n} \oint \dots \oint \prod_{1\le i<j\le n}
 \frac{z_i-z_j}{z_i-qz_j} \\ \times \prod_{i=1}^n \left[ \left( \frac{1+ q^{-1} \frac{1-b}{1-qb}  z_i}{1+ \frac{1-b}{1-qb} z_i}   \right)^{x_i-1}
 \left( \frac{1+z_i}{1+q^{-1}z_i} \right)^{y} \frac{d z_i}{z_i} \right],
\end{multline}
with positively oriented integration contours encircling $-q$ and no other poles of the integrand.
In particular, $E_N(x_1,\dots,x_N)$ does not depend on $\alpha$.
\end{theorem}
\begin{remark}
 The expression $q^{y-x+1} - \alpha^{-1} q^{2k-2} -q^{k-1} \O(x,y)$ in \eqref{eq_obs_def} can be written as
 $$
     \left(q^{y-x+1} q^{-H(x,y)}+\alpha^{-1} q^{k-1}\right)\left (q^{H(x,y)}-  q^{k-1}\right).
 $$
 In the case $\alpha=0$, the observable $E_N$ simplifies to
 \begin{equation}
  E_N(x_1,\dots,x_N)\Bigr|_{\alpha=0}= \E \left[ \prod_{k=1}^n \left( q^{H(x_k,y)}-  q^{k-1}\right) \right].
 \end{equation}
\end{remark}
\begin{remark}
 The formula \eqref{eq_Observable_contour_integral} matches \cite[Theorem 4.12]{BCG}, with $x_1=x_2=\dots=t+1$, $y=x$. Note
 that there is a shift by $1$ because of slightly different coordinate systems.
\end{remark}

\begin{proposition} \label{Proposition_contour_deformation}
 In \eqref{eq_Observable_contour_integral}, for each $n\ge 1$, and for $q,b$ sufficiently close to $1$,
 one can deform the contours so that they
 still include the poles at $-q$, and in addition are nested: $z_i$ is inside $q
 z_j$ for $1\le i<j\le n$. This deformation does not change the value of the
 integral.
\end{proposition}
We omit the proof of Proposition \ref{Proposition_contour_deformation}, as it is a direct contour deformation similar to \cite[Theorem 8.13]{BP1}, see also discussion after Proposition 2.2 in \cite{Bor16}.
In what follows we always use the result of Theorem \ref{Theorem_observable} on the contours of
Proposition \ref{Proposition_contour_deformation}.

\begin{comment}
\begin{proof}
 We start from the contours being tiny circles around $-q$ and deform the $z_n$
 contour to be a large circle centered at $-q$. In the process, we must collect
 residues at points $z_n=q^{-1} z_k$, $k=1,\dots,n-1$. For a given $k$, the residue
 is given by the $(n-1)$--dimensional integral

\begin{multline}
\label{eq_deformation_residue}
 \frac{q^{n(n-1)/2}}{(2\pi \ii)^{n-1}} \oint \dots \oint \prod_{1\le i<j\le n-1}
 \frac{z_i-z_j}{z_i-qz_j} \prod_{1\le i \le n-1;\, i\ne k} \frac{z_i-q^{-1}z_k}{z_i-z_k} \frac{z_k-q^{-1}z_k}{q}
 \\ \times
\prod_{i=1}^{n-1}
 \left[ \left( \frac{1+ q^{-1} \frac{1-b}{1-qb}  z_i}{1+ \frac{1-b}{1-qb} z_i}   \right)^{x_i-1}
 \left( \frac{1+z_i}{1+q^{-1}z_i} \right)^{y} \frac{d z_i}{z_i} \right]
\\ \times \prod_{i=1}^{n-1} \left[ \left( \frac{1+ q^{-2} \frac{1-b}{1-qb}  z_k}{1+ q^{-1}\frac{1-b}{1-qb}  z_k}   \right)^{x_i-1}
 \left( \frac{1+q ^{-1}z_k}{1+q^{-2}z_k} \right)^{y} \frac{1}{q ^{-1}z_k} \right].
\end{multline}
Note that the integrand of \eqref{eq_deformation_residue} as a function of $z_k$ has
no singularities inside the integration contour. Indeed, the factors $z_i-z_k$ in
denominators cancel out with similar factor in numerators, and the same holds for
the factor $(1+q^{-1}z_k)^y$. Therefore, the integral \eqref{eq_deformation_residue}
vanishes.

We can further deform $z_{n-1}$--contour, \dots, $z_2$--contour. Using the same
argument, no nontrivial residues are picked up  on each step, and the integral
remains unchanged.
\end{proof}
\end{comment}

\subsection{Limit of expectation}
\label{Section_expectation_integral}

 Straightforward limit transition in the $N=1$ version of Theorem \ref{Theorem_observable} yields that
 $$
  \lim_{L\to\infty} \E \left[ \frac{\q^{y-x} - \alpha^{-1} -\O(Lx,Ly)}{1-\alpha^{-1}} \right]
 $$
 is the expression in the right--hand side of \eqref{eq_LLN}.

 %, while the $N=2$ version for $x_1=x_2$ implies
 %that
 %$$
 %\lim_{L\to\infty} \E \left[ \left(\frac{\q^{y-x} - \alpha^{-1} -\O(Lx,Ly)}{1-\alpha^{-1}}
 %\right)^2\right]
%$$
 %is square of the expression of \eqref{eq_LLN}. We conclude that $\O(Lx,Ly)$
 %converges to a constant as $L\to\infty$. Further note that $\O(Lx,Ly)$ is a
 %strictly decreasing continuous function of $0<q^{H(Lx,Ly)}<1$. Therefore, as
 %$L\to\infty$, $q^{H(Lx,Ly)}$ becomes deterministic,
 %$$
 % \lim_{L\to\infty} q^{H(Lx,Ly)}=\q^{\mathbf h(x,y)},
% $$
% where $\q^{\mathbf h(x,y)}$ is uniquely determined through \eqref{eq_LLN}.
% \end{proof}
 Second order expansion of $N=1$ version of Theorem \ref{Theorem_observable} can be
 similarly used to obtain the second order expansion of $\E[ \O(Lx,Ly)]$ as $L\to\infty$. This
 expectation is used for the centering in Theorem \ref{Theorem_CLT}.

\subsection{Limit of covariance}
\label{Section_covariance_integral}

 Applying $N=2$ version of Theorem
 \ref{Theorem_observable},  we get for $x_1\ge x_2$
\begin{multline}
\label{eq_x18}
%\label{eq_covariance_integral}
 \lim_{L\to\infty} L\bigl[E_2(Lx_1,Lx_2)-E_1(L x_1) E_1(Lx_2)\bigr]=
 \frac{L}{(2\pi \ii)^2} \oint \oint
 \left[\frac{q z_1-q z_2}{z_1-qz_2} - 1\right] \\ \times \prod_{i=1}^2 \left[ \left( \frac{1+ q^{-1} \frac{1-b}{1-qb}  z_i}{1+ \frac{1-b}{1-qb} z_i}   \right)^{x_i-1}
 \left( \frac{1+z_i}{1+q^{-1}z_i} \right)^{y} \frac{d z_i}{z_i} \right]
\\=
\frac{\ln(\q)}{(2\pi \ii)^2} \oint \oint \frac{z_1}{z_1-z_2}
\prod_{i=1}^{2}\left[\exp\left(
      \ln(\q)\left(- x_i \frac{\s z_i}{1+ \s z_i }   + y\frac{z_i}{1+z_i} \right)\right) \frac{d z_i}{z_i}\right]
\end{multline}
where the contours (see Proposition \ref{Proposition_contour_deformation}) are such that they both enclose $-1$ and
$z_1$--contour is inside the $z_2$--contour.
%On the other hand, denoting
%$$
% \mathbf h_i = \frac{1}{L}\E[ H(Lx_i,Ly)],\quad  \Delta H_i= \frac{1}{\sqrt{L}} \biggl(H(Lx_i,y_i)-\E H(Lx_i,y)\biggr),
%$$
%we have $H(Lx_i,Ly)= L\mathbf h_i + \sqrt{L}\, \Delta H_i$. Therefore,
%\begin{multline*}
% \O(Lx_i,Ly)=-\alpha^{-1} \mathbf q^{\mathbf h_i+ \Delta H_i/\sqrt{L}}+\mathbf q^{y-x_i-\mathbf h_i- \Delta H_i/\sqrt{L}+1/L}
% \\= - \alpha^{-1} \mathbf q^{\mathbf h_i}\left(1+\ln(\q) \frac{\Delta H_i}{\sqrt{L}} + O\left(\frac{\Delta H_i^2}{L}\right)\right)
%   + \mathbf q^{y-x_i-\mathbf h_i}\left(1-\ln(\q) \frac{\Delta H_i}{\sqrt{L}}+O\left(\frac{1}{L}\right)+O\left(\frac{\Delta H_i^2}{L}\right) \right)
%\end{multline*} Taking into account that $q^{\mathbf h_i}<1$, we can rewrite the last expression
%as
%\begin{equation}
% \O(Lx_i,Ly)=- \alpha^{-1} \mathbf q^{\mathbf h_i}+ \mathbf q^{y-x-\mathbf
%h_i}-\ln(\q)\frac{\Delta H_i}{\sqrt{L}}\left( \alpha^{-1} \mathbf q^{\mathbf h_i}+ \mathbf
%q^{y-x_i-\mathbf h_i}\right)+ O\left(\frac{\Delta H_i^2}{L}\right)+O\left(\frac{1}{L}\right).
%\end{equation}
On the other hand,
\begin{multline} \label{eq_x2}
 E_2(Lx_1,Lx_2)=
 \frac{1}{(1+\alpha^{-1})(1+\alpha^{-1} q)}
 \E \Biggl[ \prod_{k=1}^2 \Bigl(q^{Ly-Lx_k+1} - \alpha^{-1} q^{2k-2}\\ -q^{k-1} \E \O(Lx_k,Ly)-q^{k-1}(\O(Lx_k,Ly)-\E\O(Lx_k,Ly)) \Bigr) \Biggr]
 \\=\frac{\prod_{k=1}^2 \E[q^{Ly-Lx_k+1} - \alpha^{-1} q^{2k-2} -q^{k-1}  \O(Lx_k,Ly)]}{(1+\alpha^{-1})(1+\alpha^{-1} q)}\\ + \frac{q\, \mathrm{Cov}(\O(Lx_1,Ly),\O (Lx_2,Ly))}{(1+\alpha^{-1})(1+q\alpha^{-1})}.
  \end{multline}
Thus, as $L\to\infty$ in the regime \eqref{eq_limit_regime},
\begin{multline*}
E_2(Lx_1,Lx_2)-E_1(L x_1) E_1(Lx_2)= \frac{q\, \mathrm{Cov}(\O(Lx_1,Ly),\O (Lx_2,Ly))}{(1+\alpha^{-1})(1+q\alpha^{-1})}
\\+  \frac{\E \left[q^{Ly-Lx_1+1} - \alpha^{-1} -\O(Lx_1,Ly) \right]}{(1+\alpha^{-1})}  \\ \times
 \Biggl(\frac{\E \left[q^{Ly-Lx_2+1} - \alpha^{-1}q^2 -q \O(Lx_2,Ly) \right]}{(1+q\alpha^{-1})} \\ -\frac{\E \left[q^{Ly-Lx_2+1} - \alpha^{-1} -\O(Lx_2,Ly) \right]}{(1+\alpha^{-1})}\Biggr),
 \end{multline*}
 which can be transformed into
 \begin{multline*}
  \frac{q\, \mathrm{Cov}(\O(Lx_1,Ly),\O (Lx_2,Ly))}{(1+\alpha^{-1})(1+q\alpha^{-1})}+ O\bigl( (1-q)^2 \bigr) \\ +
  \frac{(1-q)\alpha^{-1}\prod_{k=1}^2\E \left[q^{Ly-Lx_k+1} - \alpha^{-1} -\O(Lx_k,Ly) \right] }{(1+\alpha^{-1})^3}\\ +
  \frac{\E \left[q^{Ly-Lx_1+1} - \alpha^{-1} -\O(Lx_1,Ly) \right]}{(1+\alpha^{-1})^2}
 \left(\alpha^{-1}(1-q^2) + (1-q)\E \O(Lx_2,Ly)\right).
\end{multline*}
We conclude that
\begin{multline}
\label{eq_x17}
 \lim_{L\to\infty} \frac{L\, \mathrm{Cov}(\O(Lx_1,Ly),\O (Lx_2,Ly))}{(1+\alpha^{-1})^2}
 \\ = \lim_{L\to\infty}L\Bigl[E_2(Lx_1,Lx_2)-E_1(L x_1) E_1(Lx_2)\Bigr]\\+
 \ln(\q)
    \lim_{L\to\infty}\Bigl[ E_1(Lx_1)\Bigr]
   \frac{\lim\limits_{L\to\infty}\Bigl[ E_1(Lx_2)]+
    \q^{y-x_2}+\alpha^{-1} }{1+\alpha^{-1}}.
\end{multline}
Using \eqref{eq_x17}, \eqref{eq_x18}, and the computation of Section
\ref{Section_expectation_integral} we arrive at \eqref{eq_covariance_integral}.
\subsection{Cumulant-type sums}
\label{Section_cumulant_sums}

Our proof of the asymptotic Gaussianity in Theorem \ref{Theorem_CLT}  relies on a combinatorial
statement presented in this section.

Let $\S_n$ denote the set of all \emph{set partitions} of $\{1,\dots,n\}$. An element $s\in\S_N$ is
a collection $S_1,\dots, S_k$ of disjoint subsets of $\{1,\dots,n\}$ such that
$$
\bigcup_{m=1}^k S_m=\{1,\dots,n\}.
$$
The number of non-empty sets in $s\in\S_n$ will be called the \emph{length} of $s$ and denoted as
$\ell(s)$.

Fix $n=1,2,\dots$ and suppose that for each subset $A\subset \{1,2,\dots,n\}$ we are given a number
$M_A$ called the ``joint moment of $A$''. Then we define the corresponding joint cumulant $C_n$
through
\begin{equation}
\label{eq_cumulant_through_moment}
 C_n:=\sum_{s\in \S_n} (-1)^{\ell(s)+1} \bigl(\ell(s)-1\bigr)! \prod_{A\in s} M_A.
\end{equation}

 \begin{theorem}
 \label{theorem_shifted_cumulants}
  Fix $n>2$. Take $n$ random variables $\xi_1,\dots,\xi_n$, $n$ deterministic
  real numbers $r_1$, \dots, $r_n$,
   $n(n-1)/2$ real numbers $a_{ij}$, $1\le i<j \le n$, and an auxiliary small parameter $\eps>0$. Define two different sets of moments $M_A$, $M'_A$ for $A=\{i_1<i_2<\dots<i_m\}
  \subset \{1,\dots,n\}$ through
  \begin{equation}
  \label{eq_two_types_of_cums}
   M_A=\E \left[\prod_{k=1}^m \xi_{i_k}\right],\qquad M'_A=\E \left[\prod_{k=1}^m (r_{i_k}+\eps\cdot \xi_{i_k})\right] \prod_{1\le k<l\le m}
   (1+\eps^2\cdot a_{i_k,i_l}).
  \end{equation}
  Then the corresponding cumulants $C_n$, $C'_n$ given by \eqref{eq_cumulant_through_moment} are related through
  \begin{equation}
  \label{eq_cumulant_and_cumulant}
   C'_n=\eps^n\cdot C_n + \eps^{n+1} \cdot P(\eps,r_i,a_{ij},\xi_i)  \quad \text{ or } \quad C_n=\eps^{-n} \cdot C'_n-\eps \cdot P(\eps,r_i,a_{ij},\xi_i),
  \end{equation}
  where the remainder $P$ is a polynomial in $\eps$, $r_i$, $a_{ij}$, $1\le i,j\le n$, and  joint moments of $\xi_i$ of the total order up to $n$.
 \end{theorem}
 \begin{remark}
  \label{Remark_shifted_cumulants_homo}
   If $a_{ij}$ depend only on the second index,  $a_{i,j}=\tilde a_j$, then $M'_A$ can be rewritten as
   \begin{equation}
   M'_A=\E \left(\prod_{k=1}^m \left[(r_{i_k}+\eps\cdot \xi_{i_k})
   (1+\eps^2\cdot \tilde a_{i_k})^{k-1}\right]\right).
  \end{equation}
  This is the form which appears in our proof of Theorem \ref{Theorem_CLT}.
 \end{remark}
\begin{proof}[Proof of Theorem \ref{theorem_shifted_cumulants}]
 Let us expand $M'_A$ into a large sum, opening the parentheses, substitute into $C'_n$ and collect the terms.
 Each term is a product of (usual) moments $M_B$, numbers $r_{i_k}$ and $a_{i_k,i_l}$, and powers of $\eps$.
 We plug in the expansions into the definition of $C'_n$ and further expand and collect the same terms as much as possible.

 Let us introduce a combinatorial encoding for each term of the resulting sum. We start with $n$ vertices,
  representing the indices $\{1,2,\dots,n\}$. We proceed by drawing edges between some of the
  vertices: an edge joining $i$ with $j$ represents the factor  $\eps^2\cdot a_{i,j}$, $i<j$.
  Some of the vertices will be linked into (disjoint) clusters: a cluster with vertices $i_1,\dots,i_m$
  represents the factor $\eps^m \E \bigl[\prod_{k=1}^m \xi_{i_k}\bigr]$. Any vertex $t$  that does
   not belong to any cluster produces the factor $r_t$. We call the resulting combinatorial structure
   a clustered graph and identify it with the expression obtained by multiplying the factors corresponding
   to its edges and clusters.

 {\bf Claim.} For each clustered graph with non-zero contribution to $C'_n$ one of the following holds:
  \begin{enumerate}
    \item Either there are no clusters and the remaining graph is connected,

    \hskip-1.5cm or

    \item Each vertex is connected (by a path consisting of edges) to a vertex belonging to
        a cluster (in other words, each edge--connected component intersects with a cluster).
  \end{enumerate}

 Put it otherwise, the claim says that if we fix a clustered graph for which neither of the conditions holds, then the
 sum of the terms in $C'_n$ corresponding to this graph
 vanishes.  Before proving the claim note that it implies the statement of the theorem.
 Indeed, if there are no clusters, then we must have at least $n-1$ edges,
 which produces the factor $\eps^{2(n-1)}=O(\eps^{n+1})$. Otherwise, each vertex in a
 cluster produces a factor of $\eps$, and all vertices outside the clusters produce at least
$\eps^{m+1}$, where $m\ge 1$ is their number. Altogether we again get $O(\eps^{n+1})$. We conclude
that the only structures that have the power of $\eps$ smaller than $\eps^{n+1}$ are those with no
edges at all and with all vertices belonging to some clusters. This gives $\eps^{n}$ prefactor and
these terms precisely combine into the conventional cumulant $C_n$.

 \medskip

 We now prove the claim. Fix a clustered graph $G$ for which neither of the properties hold.
 Then this graph has an edge--connected component $A$ which does not intersect with clusters
 and $A\ne \{1,\dots,n\}$.
 Take a set partition $s_0$ of the set $\{1,\dots,n\}\setminus A$. Note that each set
  partition $s$ in \eqref{eq_cumulant_through_moment} for which  the graph $G$ arises in the decomposition
  (when $M_A$ are replaced by $M_A'$), is necessarily obtained by taking such $s_0$ and then either adding $A$ to
  one of the sets, or by putting $A$ as a new set of the partition. Each choice leads to one
  appearance of $G$. Let us sum over all these choices. For that suppose that $s_0$ has $r$ parts.
  When we  add $A$ to one of the sets of $s_0$, then the resulting partition has $r$ parts,
  and therefore the corresponding coefficient in \eqref{eq_cumulant_through_moment} is $(-1)^{r+1}(r-1)!$.
  On the other hand, if $A$ creates a new set, then the coefficient becomes $(-1)^{r+2} r!$. Since
  there are precisely $r$ sets to which $A$ can be added and $r \cdot (-1)^{r+1} (r-1)!+ (-1)^{r+2} r!=0$,
  we see that the total contribution of $G$ in \eqref{eq_cumulant_through_moment} (with $M_A'$ instead of $M_A$)
   vanishes.
\end{proof}

\subsection{Proof of LLN and CLT}
\begin{proof}[Proof of Theorem \ref{Theorem_LLN}]
 In Section \ref{Section_expectation_integral} we have shown that $\E(\O(Lx,Ly))$ converges to
 the expression given by \eqref{eq_LLN}. The covariance computation of Section \ref{Section_covariance_integral}
 implies that $\lim_{L\to\infty}\E(\O(Lx,Ly)-\E(\O(Lx,Ly)))^2=0$ and, therefore, $\O(Lx,Ly)$
 converges in probability to the deterministic limit given by \eqref{eq_LLN}. Since $\frac{1}{L}H(Lx,Ly)$
 is obtained from $\O(Lx,Ly)$ by applying
 a strictly monotone uniformly Lipschitz map, cf.\ \eqref{eq_O}, we deduce the convergence for $\frac{1}{L}H(Lx,Ly)$ as well.
\end{proof}

\begin{proof}[Proof of Theorem \ref{Theorem_CLT}]
In Section \ref{Section_covariance_integral} we obtained the formulas for the asymptotic covariance
of $L^{1/2} \O(Lx_k,Ly)$ which matches \eqref{eq_covariance_integral},
\eqref{eq_covariance_integral_zero}. It remains to prove the asymptotic Gaussianity, for which  we
are going to show that the joint cumulants of $L^{1/2} \O(Lx_k,Ly)$ of orders higher than $2$
vanish as $L\to\infty$.

Fix $n>3$ and  take $n$--tuple $x_1\le x_2\le \dots\le x_n$. We aim to prove that the $n$th joint
cumulant of $\{\O(Lx_k,Ly)\}_{k=1}^n$, which we denote $C_n$, decays faster than $L^{-n/2}$ as
$L\to\infty$.

For a set $A=\{i_1<i_2<\dots<i_m\}\subset\{1,2,\dots,n\}$, let $M'_A=E_m(i_1,i_2,\dots,i_m)$, as
given by \eqref{eq_Observable_contour_integral}. As in  Section \ref{Section_cumulant_sums}, we
denote through $C'_n$ the corresponding joint ``cumulant''. Contour integral expressions of Theorem
\ref{Theorem_observable} combined with \cite[Lemma 4.2]{BG_GFF} (with $\gamma=1$) yields that
$C'_n=o(L^{-n/2})$ as $L\to\infty$.

Note that \emph{a priori} $C'_n$ is different from the conventional cumulant $C_n$. However, we can
relate them using Theorem \ref{theorem_shifted_cumulants}. For that we write
$$
\O(Lx,Ly)=\O_{\infty}(x,y)+L^{-1/2} \Delta\O(x,y),
$$
where $\O_{\infty}(x,y)=\E \O(Lx,Ly)$ and $\Delta\O(x,y)$ is the fluctuation, for which we know (from the covariance computation of Section \ref{Section_covariance_integral}) that it is tight as $L\to\infty$.

Then we transform $E_m(Lx_1,\dots,Lx_m)$ as
\begin{multline}
\label{eq_x19}
  \E\prod_{k=1}^m \frac{q^{Ly-Lx_k+1} - \alpha^{-1} q^{2k-2} -q^{k-1} \O(Lx_k,Ly)}{1+\alpha^{-1} q^{k-1}}\\
  =\E\prod_{k=1}^m \frac{q^{Ly-Lx_k+1}- \alpha^{-1}q^{2(k-1)}-q^{k-1} \O_{\infty}(x_k,y)-q^{k-1}L^{-1/2} \Delta\O(x,y)}{(1+\alpha^{-1})
  (1+\frac{\alpha^{-1}}{1+\alpha^{-1}}(q^{k-1}-1))}.
\end{multline}
Let us examine the $k$th factor of \eqref{eq_x19}. The numerator splits into four terms, each of them has the form appearing in Theorem \ref{theorem_shifted_cumulants}. We need to deal with the denominator. For that we choose an integer $M>n/2$ and expand
\begin{multline*}
\frac{1}{(1+\alpha^{-1})
  (1+\frac{\alpha^{-1}}{1+\alpha^{-1}}(q^{k-1}-1))}=
  \frac{1}{1+\alpha^{-1}}\Biggl[1-\frac{\alpha^{-1}}{1+\alpha^{-1}}(q^{k-1}-1)\\+
  \left(\frac{\alpha^{-1}}{1+\alpha^{-1}}(q^{k-1}-1)\right)^2+\dots + \left(\frac{\alpha^{-1}}{1+\alpha^{-1}}(q^{k-1}-1)\right)^M+o\left( (q-1)^M \right) \Biggr].
\end{multline*}
Note that we can ignore $o( (q-1)^M)$, as this term has smaller order than the desired cumulants.
In the rest, we expand each $(q^{k-1}-1)
^b$ into $b+1$ terms using the Binomial theorem. Altogether we get $1+2+\dots+(M+1)=(M+1)(M+2)/2$ terms.

 We plug the resulting sum into the $k$th factor of \eqref{eq_x19} and get a sum of $2(M+1)(M+2)$ terms. Each term has a form
$$
 r \cdot  \bigl[ (1+ (q-1))^u \bigr]^{k-1}\text{ or } L^{-1/2} \xi \bigl[ (1+ (q-1))^u \bigr]^{k-1},
$$
where $u$ is a positive integer, $r$ is a deterministic number, $\xi$ is a random variable. We arrive at an expression of the form of the definition of $M'_A$ in
\eqref{eq_two_types_of_cums}, see Remark \ref{Remark_shifted_cumulants_homo}. The conclusion is that \eqref{eq_x19} turns into a sum of finitely many expressions, each of which has the form of $M'_A$ (for various choices of parameters) in Theorem \ref{theorem_shifted_cumulants}.

At this point we would like to apply Theorem \ref{theorem_shifted_cumulants} with $\eps=L^{-1/2}$.
Note that the ``cumulants'' $C'_n$ in this theorem are multilinear over the choices of $r_i$ and
$\xi_{i}$. In other words, if we fix $1\le t\le n$, set $r_t= r_t[1]+r_t[2]$,
$\xi_t=\xi_t[1]+\xi_t[2]$ and denote the resulting cumulants through $C'_n[1]$, $C'_n[2]$, then
$C'_n=C'_n[1]+C'_n[2]$. Thus, after we expand the $k$th factor in \eqref{eq_x19} into $2(M+1)(M+2)$
terms for each $k=1,\dots,m$ and further plug the expansions into ``cumulant'' $C'_n$, then using
the multilinearity we get a sum of $n\cdot 2(M+1)(M+2)$ ``cumulants''. For each of those we apply
Theorem \ref{theorem_shifted_cumulants} to reduce them to the conventional cumulants. At this point
most of the terms vanish, as they involve the conventional cumulant of a constant (in fact, zero)
random variable. In order $L^{-n/2}$ the only remaining term is $L^{-n/2}$ times the conventional
cumulant of $\Delta\O(x_1,y),\dots\Delta\O(x_n,y)$. Since by \cite[Lemma 4.2]{BG_GFF}, the entire
sum, $C'_n$, is $o\left(L^{-n/2}\right)$, we conclude that the latter cumulant, $C_n$, is
$o\left(L^{n/2}\right)$.
\end{proof}

\subsection{$\q\to 0$ limit}
Here we prove Proposition \ref{Proposition_q_0}. Although an extension of this
computation to the case of general $\alpha$ is possible,  we do not address it here.

 At $\alpha=0$, we take the statement of Theorem \ref{Theorem_LLN} and absorb
 $1$ as the residue at $0$ of the contour integral, getting the formula
\begin{equation}
\label{eq_LLN_form}
    \q^{\mathbf h(x,y)}
    =\frac{1}{2\pi \ii} \oint \exp\left(
      \ln(\q)\left(- x \frac{\s z}{1+ \s z }   + y\frac{z}{1+z} \right)\right) \frac{d z}{z} ,
\end{equation}
with integration contour enclosing $0$ and $-1$, but not $-\s^{-1}$.
At this point, we restrict ourselves to the case
\begin{equation}
\label{eq_xy_regime}
 \s \le \frac{x}{y} \le \s^{-1}.
\end{equation}

The $\q\to 0$ limit means that $\ln(\q)$ is a large parameter. We study the asymptotics of
\eqref{eq_LLN_form} through the steepest descent method. We thus need to find critical points of
the argument of the exponent, i.e.\ to solve
\begin{equation}
0=\frac{\partial}{\partial z} \left(- x \frac{\s z}{1+ \s z }   + y\frac{z}{1+z}
\right)=-\frac{\s x}{(1+\s z)^2} + \frac{y}{(1+z)^2}.
\end{equation}
The solutions $z_c$ are given by
\begin{equation}
\label{eq_critical}
 \frac{1+\s z_c}{1+z_c}=\pm \sqrt{\frac{\s x}{y}}, \qquad
  z_c=\frac{1-\left(\pm
 \sqrt{\frac{\s x}{y}}\right) }{\pm
 \sqrt{\frac{\s x}{y}}-\s},\qquad
  {1+\s z_c}=\frac{\s^{-1}-1}{\s^{-1}-\left(\pm \sqrt{\frac{y}{\s x}}\right)}.
\end{equation}
We need the solution with
$$
 \frac{\partial^2}{\partial z^2} \left(- x \frac{\s z}{1+ \s z }   + y\frac{z}{1+z}
\right)<0,
$$
as we want the steepest descent contour to be orthogonal to the real axis (note that our large parameter $\ln(\q)$ is negative). I.e., we need
$$
 2\frac{\s^2 x}{(1+\s z)^3} -2\frac{y}{(1+z)^3}<0,
$$
which is true if
\begin{equation}
\label{eq_case_1}
\begin{cases}\left(\frac{1+\s z}{1+ z}\right)^3 >\frac{\s^2 x}{y},\\ 1+\s z>0,\end{cases}
\quad \text{ or } \quad
\begin{cases}\left(\frac{1+\s z}{1+ z}\right)^3 <\frac{\s^2 x}{y},\\ 1+\s z<0.\end{cases}
\end{equation}
Note that due to \eqref{eq_xy_regime} and \eqref{eq_critical}, $1+\s z_c>0$ for both solutions
Therefore, the solution with $-\sqrt{\frac{\s x}{y}}$  does not satisfy \eqref{eq_case_1}, while the second one does. We
conclude that the correct solution has $+\sqrt{\frac{\s x}{y}}$ in \eqref{eq_critical}, i.e.,
$$
 z_c=\frac{1-
 \sqrt{\frac{\s x}{y}} }{
 \sqrt{\frac{\s x}{y}}-\s}.
$$
Using \eqref{eq_xy_regime} we see that $z_c>0$, and, therefore, we can deform the contour in \eqref{eq_LLN_form} to run through the
critical point. The usual critical point approximation arguments show that the integral then
behaves as
\begin{equation}
\label{eq_x6} \q^{\mathbf h(x,y)} \sim \exp\left(
      \ln(\q)\left(- x \frac{\s z_c}{1+ \s z_c }   + y\frac{z_c}{1+z_c}
      \right)\right) \frac{1}{z_c} \cdot \frac{1}{2\pi\ii} \int_{-\ii\infty}^{\ii\infty} \exp(\kappa_c u^2) du ,
\end{equation}
where $\kappa_c$ is half of the second derivative at the critical point --- the integral is
evaluated to $\sqrt{2\pi/\kappa_c}$. Therefore,
\begin{equation}
 \lim_{\q\to 0} \mathbf h(x,y)= - x \frac{\s z_c}{1+ \s z_c }   +
 y\frac{z_c}{1+z_c}
% =
% - x \frac{\s \frac{1-
% \sqrt{\frac{\s x}{y}} }{
% \sqrt{\frac{\s x}{y}}-\s}}{1+ \s \frac{1-
% \sqrt{\frac{\s x}{y}} }{
% \sqrt{\frac{\s x}{y}}-\s} }   + y\frac{\frac{1-
% \sqrt{\frac{\s x}{y}} }{
% \sqrt{\frac{\s x}{y}}-\s}}{1+\frac{1-
% \sqrt{\frac{\s x}{y}} }{
% \sqrt{\frac{\s x}{y}}-\s}}
% \\=- x
% \frac{\s \left(1- \sqrt{\frac{\s x}{y}} \right)}
% {{ \sqrt{\frac{\s x}{y}}-\s}+ \s \left(1- \sqrt{\frac{\s x}{y}} \right) }
% y\frac{1-\sqrt{\frac{\s x}{y}} }
% {{ \sqrt{\frac{\s x}{y}}-\s}+{1-\sqrt{\frac{\s x}{y}} }}
% =
% - x
% \frac{\s \left(\sqrt{\frac{y}{\s x}}- 1 \right)}
% {1- \s  }
% + y\frac{1-\sqrt{\frac{\s x}{y}} }
% {-\s+1}
% \\
% =\frac{-\sqrt{\s x y}+\s x+y-\sqrt{\s xy}}{1-\s}
 =\frac{ (\sqrt{\s
 x}-\sqrt{y})^2}{1-\s},
\end{equation}
which is precisely \eqref{eq_BCG_LLN}. By combinatorics of the model, $\mathbf
h(x,y)=0$  for $x/y=\s^{-1}$ implies that also $\mathbf h(x,y)=0$ for all
$x/y>\s^{-1}$, as there are no paths to the right from the line $x/y=\s^{-1}$.
Similarly, $\mathbf h(x,y)=y-x$ for $x/y=\s$ implies that $\mathbf h(x,y)=y-x$ for
$x/y<\s$, as there is maximal possible number of paths to the left from the line
$x/y=\s$. In the formula \eqref{eq_LLN} this can be also seen: the integral will now
be dominated not by the neighborhood of the critical point, but by the residue at
$0$ or $\infty$, which appears when we deform the contour to reach the critical
point.

\section{Four point relation} \label{Section_four_point}

All our results for more general (than domain wall) boundary conditions are based on the following
statement.

\begin{theorem}
\label{Theorem_4_point}
 Consider the stochastic six--vertex model in the quadrant with arbitrary (possibly, even random) boundary conditions.
 For each $x,y\ge 0$ we have an identity
 \begin{multline}
 \label{eq_4_point_relation}
  q^{H(x+1,y+1)}-b \cdot q^{H(x,y+1)}
   - bq \cdot q^{H(x+1,y)}+(b+bq-1) \cdot q^{H(x,y)}= \xi(x+1,y+1),
 \end{multline}
 where the conditional expectation and variance for $\xi$ are
 \begin{equation}
 \label{eq_4_point_no_correlation}
 \E\bigl[ \xi(x+1,y+1) \mid H(u,v), u\le x \text{ or } v\le y  \bigr]=0,
 \end{equation}
 \begin{multline}
 \label{eq_4_point_covariance}
   \E \bigl[ \xi^2 (x+1,y+1) \mid H(u,v), u\le x \text{ or } v\le y  \bigr]\\=
   \bigr(qb(1-b)+b(1-qb)\bigl) \Delta_x \Delta_y
   + b(1-qb)(1-q) q^{H(x,y)} \Delta_x
   - b(1-b)(1-q) q^{H(x,y)} \Delta_y,
 \end{multline}
 with
 $$
  \Delta_x=q^{H(x+1,y)}-q^{H(x,y)},\quad \Delta_y=q^{H(x,y+1)}-q^{H(x,y)},
 $$
\end{theorem}
\begin{remark}
 \label{Remark_noise_uncor}
 The relation \eqref{eq_4_point_no_correlation} implies that $\xi(x,y)$ are uncorrelated, i.e.,
 $\E \xi(x,y) \xi(x',y')=0$ for any $(x,y)\ne (x',y')$.
\end{remark}
\begin{proof}[Proof of Theorem \ref{Theorem_4_point}]
  Let us denote $H(x,y)$ through $h$.
  We fix the types of vertices at positions $(x,y)$, $(x+1,y)$, $(x,y+1)$ and sample the vertex at $(x+1,y+1)$ according to the probabilities of Figure \ref{Fig_weights}. There are four cases to consider.

  \begin{enumerate}
  \item If no paths enter into the vertex $(x+1,y+1)$ from below or from the left, then the type
      of the vertex is $I$ and $H(x+1,y)=H(x,y+1)=H(x+1,y+1)=h$, $\Delta_x=\Delta_y=0$. In
      particular, $\xi(x+1,y+1)=0$, and, therefore, its conditional expectation and variance
      vanish, which agrees with \eqref{eq_4_point_no_correlation}, \eqref{eq_4_point_covariance}.
   \item If two paths enter into the vertex $(x+1,y+1)$ (one from below and one from the left),
       then the type of the vertex is $II$, and $H(x+1,y)=h-1$, $H(x,y+1)=h+1$, $H(x+1,y+1)=h$,
       $\Delta_x =q^h (q^{-1}-1)$, $\Delta_y=q^h(q-1)$. This implies
       $\xi(x+1,y+1)=q^h(1-bq-bq\cdot q^{-1}-(1-b-bq))=0$. Again, the conditional expectation and
       variance vanish matching \eqref{eq_4_point_no_correlation}, \eqref{eq_4_point_covariance}.
   \item If the path enters into the vertex $(x+1,y+1)$ from below, but no path enters from the
       left, then we choose between the vertex types $IV$ and $VI$ with probabilities $bq$ and
       $1-bq$, respectively. In both cases $H(x+1,y)=h-1$, $H(x,y+1)=h$,
       $\Delta_x=q^h(q^{-1}-1)$, $\Delta_y=0$. In the first case of type $IV$, $H(x+1,y+1)=h-1$
       and
    $$\xi(x+1,y+1)=q^h(q^{-1}-b-bq \cdot q^{-1}+(b+bq-1))=q^h(q^{-1}-b)(1-q).$$
    In the second case of type $VI$, $H(x+1,y+1)=h$ and
    $$\xi(x+1,y+1)=q^{h}(1-b-bq\cdot q^{-1}+ (b+bq-1))=q^h b(q-1).$$
    The conditional expectation of $\xi(x+1,y+1)$ becomes
    $$
      bq\cdot q^h(q^{-1}-b)(1-q)+ (1-bq)\cdot q^h b(q-1)=0.
    $$
    The conditional variance is
    $$
     bq \cdot \bigl( q^h(q^{-1}-b)(1-q) \bigr)^2+ (1-bq) \bigl(q^h b(q-1)\bigr)^2= b(1-bq)(1-q)(q^{-1}-1) q^{2h},
    $$
    which matches \eqref{eq_4_point_covariance}.
    \item If the path enters into the vertex $(x+1,y+1)$ from the left, but no path enters from
        below, then we choose between the vertex types $III$ and $V$ with probabilities $b$ and
        $1-b$, respectively. In both cases $H(x+1,y)=h$, $H(x,y+1)=h+1$, $\Delta_x=0$,
        $\Delta_y=q^h(q-1)$. In the first case of type $III$, $H(x+1,y+1)=h+1$ and
        $$
         \xi(x+1,y+1)=q^h(q-b\cdot q-bq+(b+bq-1))=q^h(1-b)(q-1).
        $$
        In the second case of type $V$, $H(x+1,y+1)=h$ and
        $$
         \xi(x+1,y+1)=q^h(1-b\cdot q-bq+(b+bq-1))=q^hb(1-q).
        $$
        The conditional expectation of $\xi(x+1,y+1)$ becomes
        $$
          b\cdot q^h(1-b)(q-1)+ (1-b)\cdot q^hb(1-q)=0.
        $$
        The conditional variance of $\xi(x+1,y+1)$ is
        $$
          b\cdot \bigl(q^h(1-b)(q-1)\bigr)^2+ (1-b)\cdot \bigl(q^hb(1-q)\bigr)^2=b(1-b)(1-q)^2 q^{2h},
        $$
        which matches \eqref{eq_4_point_covariance}. \qedhere
  \end{enumerate}
\end{proof}

At times it will be convenient to use the integrated form of \eqref{eq_4_point_relation}.

\begin{corollary}
\label{corollary_4_point_integrated}
 In the notations of Theorem \ref{Theorem_4_point}, for each $X,Y\ge 1$ we have
 \begin{multline}
 \label{eq_4_point_relation_integrated}
  -(1-b) \sum_{x=1}^{X-1} q^{H(x,0)} - (1-bq) \sum_{y=1}^{Y-1}  q^{H(0,y)}+(1-b) \sum_{x=1}^{X-1} q^{H(x,Y)}
  \\ + (1-bq) \sum_{y=1}^{Y-1}  q^{H(X,y)}
 + (b+bq-1) q^{H(0,0)}- bq \cdot q^{H(X,0)} - b \cdot q^{H(0,Y)} + q^{H(X,Y)}
\\  =\sum_{x=1}^{X}\sum_{y=1}^{Y} \xi(x,y).
\end{multline}
\end{corollary}
\begin{proof}
We sum \eqref{eq_4_point_relation} over $x=0,\dots,X-1$, $y=0,\dots,Y-1$.
\end{proof}

\section{The telegraph partial differential equation}

\label{Section_PDEs}

We saw in Theorem \ref{Theorem_LLN} and equation \eqref{eq_tele_for_contour}  that the limit shape (after a non-linear transformation) solves
the telegraph equation. In order to move forward, we need to collect the facts about this equation
and its solutions. Some parts of this section are based on \cite[Chapter V]{Courant}.

\subsection{Existence and uniqueness of solutions}

Take three arbitrary real parameters $\lambda$, $\mu$, $\nu$ and a continuous function
$g(x,y):\mathbb R_{\ge 0} \times \mathbb R_{\ge 0} \to \mathbb R$. Consider the following integral equation
for an unknown continuous function $\phi(x,y)$, $x\ge 0$, $y\ge 0$:
\begin{multline}
\label{eq_integrated_telegraph}
 \phi(X,Y)+ \lambda \int_0^X \phi(x,Y) dx  +\mu \int_0^Y \phi(X,y) dy\\  +\nu \int_0^X \int_0^Y \phi(x,y)dxdy = g(X,Y).
\end{multline}
\begin{proposition}
\label{Proposition_integrated_PDE_solution}
 For each $a,b>0$, the equation \eqref{eq_integrated_telegraph} has a continuous solution
 $\phi(x,y)$ in $[0,a]\times [0,b]$. The solution is unique.
\end{proposition}
\begin{proof}
Because of the invariance of the form of the equation with respect to translations, it suffices to prove the claim for small $a$ and $b$; we will require that
$$
 (|\lambda|+|\mu|+|\nu|)(a+b+ab)<1.
$$

 Let $\mathbf C_{a,b}$ denote the Banach space of continuous functions on $[0,a]\times [0,b]$ equipped
 with the supremum norm. Let $\Theta:\mathbf C_{a,b}\to \mathbf C_{a,b}$ be defined through
 \begin{multline*}
  [\Theta f](X,Y)\\=g(X,Y)-\lambda \int_0^X f(x,Y) dx  -\mu \int_0^Y f(X,y) dy-\nu \int_0^X \int_0^Y f(x,y)dxdy.
 \end{multline*}
 We claim that for sufficiently small $a$, $b$ the map $\Theta$ is a contraction. Indeed,
 \begin{multline*}
  \|\Theta f_1 -\Theta f_2\|= \sup_{0\le X\le a, \, 0 \le Y \le b} \Bigl| \lambda \int_0^X (f_1(x,Y)-f_2(x,Y))dx
 \\ + \mu \int_0^Y (f_1(X,y)-f_2(X,y))dy+\nu \int_0^X \int_0^Y (f_1(x,y)-f_2(x,y))dxdy \Bigr| \\ \le
 (|\lambda|+|\mu|+|\nu|)(a+b+ab) \|f_1-f_2\|.
 \end{multline*}
 By the contraction mapping principle (Banach fixed--point theorem), there exists a unique $\phi$ such that $\Theta \phi=\phi$,
 which gives the unique solution to \eqref{eq_integrated_telegraph}.
\end{proof}
Differentiating \eqref{eq_integrated_telegraph}, we rewrite it as a partial differential equation (with $\tilde g=g_{xy}$)
\begin{equation}
\label{eq_telegraph_general}
 \phi_{xy}(x,y)+ \lambda \phi_y(x,y)+ \mu \phi_x(x,y)+\nu \phi(x,y) = \tilde g(x,y), \quad x,y>0.
\end{equation}
 For various choices of $\lambda$, $\mu$, $\nu$ and $\tilde g$ this equation has various names, e.g.\
 the telegraph equation or Klein--Gordon equation.

 The solutions to \eqref{eq_telegraph_general} with different $\lambda$, $\mu$, $\nu$ are readily
 related to each other by an observation that if $\phi$ solves \eqref{eq_telegraph_general}, then $\psi(x,y)=e^{wx + vy}\phi(x,y)
 $ solves
 \begin{equation}
 \label{eq_telegraph_transform}
 \psi_{xy}+
 (\lambda-w) \psi_y +
 (\mu-v) \psi_x+(\nu-w\mu-v\lambda+wv) \psi = \tilde g(x,y) \exp(wx+vy).
  \end{equation}

 \begin{proposition} \label{Proposition_PDE_solution}
  Take $a,b>0$ and consider the equation \eqref{eq_telegraph_general} on an unknown continuous
  function $\phi:[0,a]\times [0,b]\to\mathbb R$ with continuous mixed derivative $\phi_{xy}$ in the interior
  of the rectangle. If $\tilde g(x,y)$ is continuous, and
  \eqref{eq_telegraph_general} is supplemented with boundary condition
  $$
   \phi(x,0)=\chi(x),\quad \phi(0,y)=\psi(y),
  $$
  with given continuously differentiable $\chi$ and $\psi$ that have the same value at the origin, then \eqref{eq_telegraph_general} has a unique
  solution.
 \end{proposition}
 \begin{remark} When the boundary data or $\tilde g(x,y)$ are less regular,
 then one need to understand the solution $\phi$ in a generalized sense through \eqref{eq_integrated_telegraph}, \eqref{eq_solution_integrated}.
 In the next section we provide an explicit formula \eqref{eq_inhom_solution_integral} for the solution, which can be also used
 for extending to more general initial data, see Remark \ref{Remark_by_parts} below.
 \end{remark}
 \begin{proof}[Proof of Proposition \ref{Proposition_PDE_solution}]
  Using transformation \eqref{eq_telegraph_transform} if necessary, we may and will consider only
  the case $\lambda=\mu=0$. We integrate the equation to get
  \begin{multline}
  \label{eq_solution_integrated}
   \phi(X,Y)-\phi(X,0)-\phi(0,Y)+\phi(0,0)+\nu \int_0^X \int_0^Y \phi(x,y)dx dy \\ =
   \int_0^X \int_0^Y \tilde g(x,y) dx dy,
  \end{multline}
  which is \eqref{eq_integrated_telegraph} with
  $$
   g(X,Y)= \int_0^X \int_0^Y \tilde g(x,y) dx dy +\chi(X)+\psi(Y)-\chi(0).
  $$
  By Proposition \ref{Proposition_integrated_PDE_solution}, there is a unique continuous solution.
  Since $\phi(X,Y)$ in \eqref{eq_solution_integrated} is given by the sum of double integrals of continuous
  functions and two other continuously differentiable functions, its mixed partial derivative
  exists and is continuous. Thus, we can differentiate \eqref{eq_solution_integrated} returning to \eqref{eq_telegraph_general}.
 \end{proof}

\subsection{Solutions as contour integral}

Define the \emph{Riemann function} (for the equation \eqref{eq_general_PDE} below) through

\begin{multline}
 \R(X,Y; x,y)=\frac{1}{2\pi \ii} \oint_{-\beta_1}   \frac{(\beta_2-\beta_1)\, dz}{(z+\beta_1)(z+\beta_2)} \\ \times \exp\left[
 (\beta_1-\beta_2) \left(-(X-x) \frac{z}{z+ \beta_2} + (Y-y) \frac{z}{z+\beta_1}\right)
 \right] ,
\end{multline}
where the integration goes in positive direction and encircles $-\beta_1$, but not $-\beta_2$. Note that we can also integrate in the negative
direction around $-\beta_2$ for the same result, because the residue of the integrand at infinity vanishes.

\begin{theorem} \label{Theorem_general_solution} Consider the equation
 \begin{equation}
 \label{eq_general_PDE}
 \phi_{XY}(X,Y)+ \beta_1 \phi_Y(X,Y)+ \beta_2 \phi_X(X,Y)=u(X,Y), \quad X,Y>0,
\end{equation}
 with boundary conditions
 \begin{equation}
 \label{eq_inhomogeneous_PDE_boundary}
  \phi(x,0)=\chi(x),\quad \phi(0,y)=\psi(y),
 \end{equation}
where $\chi$ and $\psi$ are continuously differentiable with $\psi(0)=\chi(0)$. The
solution (afforded by Proposition \ref{Proposition_PDE_solution}) has the form
\begin{multline}
\label{eq_inhom_solution_integral}
 \phi(X,Y)=\psi(0) \R(X,Y;0,0)+\int_0^Y  \R(X,Y; 0,y)  \bigl(\psi'(y) + \beta_2 \psi(y) \bigr) dy
\\
+\int_0^X \R(X,Y; x,0) \bigl(\chi'(x)+\beta_1\chi(x)\bigr) dx + \int_0^X \int_0^Y \R(X,Y; x,y) u(x,y) dx dy.
\end{multline}
\end{theorem}
\begin{remark}
\label{Remark_by_parts}
 If we integrate by parts the terms involving $\psi'(y)$ and $\chi'(x)$ in
 \eqref{eq_inhom_solution_integral}, then using the smoothness $\R(X,Y; x,y)$  we get an expression which continuously depends on the
 boundary data $\psi(y)$, $\chi(x)$ (in the supremum norm). This can be used to define the
 solution to \eqref{eq_general_PDE} for non-differentiable $\chi(x)$, $\psi(y)$.
\end{remark}

\begin{proof}
 The function $\R(X,Y; x,y)$ satisfies the following properties, which are checked by direct
 differentiation under the integral sign:
 \begin{enumerate}
  \item $\R_{XY}+\beta_1 \R_Y+\beta_2\R_X=0$,
  \item $[\R_X+\beta_1\R]_{Y=y} =0=[\R_x-\beta_1\R]_{Y=y}$,
  \item $[\R_Y+\beta_2\R]_{X=x}=0=[\R_y-\beta_2\R]_{X=x}$,
  \item $[\R]_{X=x,\, Y=y}=1$.
 \end{enumerate}

Using these properties we apply the differential operator $F\mapsto F_{XY}+\beta_1 F_Y+\beta_2 F_X$
to each term in \eqref{eq_inhom_solution_integral}. The first term gives $0$ by the first property.
The second term gives (using the first two properties)
\begin{multline*}
 \int_0^Y   \bigl(\R_{XY}(X,Y; 0,y)+\beta_1 \R_{Y}(X,Y; 0,y)+\beta_2 \R_{X}(X,Y; 0,y)\bigr)  \\ \times \bigl(\psi'(y) + \beta_2 \psi(y) \bigr) dy
 \\+ \left[ \bigl(\R_X(X,Y; 0,y)+\beta_1 \R(X,Y,0,y)\bigr) (\psi'(y) + \beta_2 \psi(y) \bigr)\right]_{y=Y}=0.
\end{multline*}
The third term also vanishes by similar reasoning with the first and third
properties. The fourth term gives (using all four properties)
\begin{multline*}
 \int_0^X \int_0^Y \bigl(\R_{XY}(X,Y; x,y)+\beta_1 \R_{Y}(X,Y; x,y)+\beta_2 \R_{X}(X,Y; x,y)\bigr) u(x,y) dx dy
 \\
 + \int_0^X [\R_{X}(X,Y; x,y)+\beta_1 \R(X,Y;x,y)]_{y=Y} dx
 \\ + \int_0^Y [\R_{Y}(X,Y;x,y)+\beta_2 \R(X,Y;x,y)]_{x=X} dy
 \\+ [\R(X,Y;x,y) u(x,y)]_{x=X, y=Y}= u(X,Y).
\end{multline*}
We conclude that \eqref{eq_inhom_solution_integral} satisfies
\eqref{eq_general_PDE}. It remains to check the boundary conditions. At $X=0$, the
third and fourth terms in \eqref{eq_inhom_solution_integral} vanish. Integrating by
parts and using the third and fourth properties, we obtain
\begin{multline*}
 \psi(0) \cdot \R(0,Y;0,0)+\int_0^Y  \R(0,Y; 0,y)  \bigl(\psi'(y) + \beta_2 \psi(y) \bigr) dy\\ =
 \R(0,Y; 0,Y)  \psi(Y)-
  \int_0^Y  (\R_y(0,Y; 0,y)-\beta_2\R(0,Y;0,y))  \psi(y) dy= \psi(Y).
\end{multline*}
At $Y=0$, the second and fourth terms in \eqref{eq_inhom_solution_integral} vanish.
Integrating by parts and using the second and fourth properties, we then get
\begin{multline*}
 \chi(0) \R(X,0;0,0)+\int_0^X \R(X,0; x,0) \bigl(\chi'(x)+\beta_1\chi(x)\bigr) dx\\=
 \chi(X) \R(X,0;X,0)+\int_0^X (\R_x(X,0;x,0)-\beta_1 \R(X,0,x,0))\chi(x)dx =\chi(X).
\end{multline*}
\end{proof}

\subsection{Discretization}
The telegraph equation has a natural discretization, which we present here. (We have not seen it in the literature before.)

Consider the following equation for an unknown function $\Phi(x,y)$, $x,y=0,1,2,\dots$:
\begin{equation}
\label{eq_discrete_PDE}
 \Phi(x+1,y+1)-b_1 \Phi(x,y+1)-b_2\Phi(x+1,y)+(b_1+b_2-1)\Phi(x,y)=u(x+1,y+1)
\end{equation}
with a given right-hand side $u$ and subject to boundary conditions
\begin{equation}
\label{eq_discrete_PDE_boundary}
\Phi(x,0)=\chi(x), \quad  \Phi(0,y)=\psi(Y), \quad X,Y=0,1,2,\dots,\quad \chi(0)=\psi(0).
\end{equation}
We take $b_1$ and $b_2$ to be arbitrary distinct real numbers satisfying $0<b_1,b_2<1$. Although, these restrictions can be easily removed if needed (this mould lead to natural modifications of the formulas below).

\begin{proposition}
\label{Proposition_discrete_PDE_unique}
 The equations \eqref{eq_discrete_PDE}, \eqref{eq_discrete_PDE_boundary} have a unique solution.
\end{proposition}
\begin{proof}
 Using \eqref{eq_discrete_PDE} and starting from \eqref{eq_discrete_PDE_boundary}, we
 recursively define the values of $\Phi(x,y)$ first for the point $(1,1)$, then for the
 points $(1,2)$, $(2,1)$, then for the points $(1,3)$, $(2,2)$, $(3,1)$, etc.
\end{proof}

Define the discrete Riemann function through
\begin{multline}
\label{eq_Discrete_R}
 \Rd(X,Y; x,y)=\frac{1}{2\pi \ii} \oint_{-\frac{1}{b_2(1-b_1)}} \frac{ (b_2-b_1)\, dz}{(1+ b_2(1-b_1) z)(1+ b_1(1-b_2) z)} \\ \times \left(\frac{1+  b_1(1-b_1)z}{1+ b_2(1-b_1) z}
  \right)^{X-x} \left( \frac{1+b_2 (1-b_2)z}{1+b_1(1-b_2)z} \right)^{Y-y},
\end{multline}
where the integration goes in positive direction and encircles $-\frac{1}{b_2(1-b_1)}$, but not $-\frac{1}{b_1(1-b_2)}$. Note that we can also integrate in the negative
direction around ${-\frac{1}{b_1(1-b_2)}}$ for the same result.

\begin{theorem} \label{Theorem_discrete_PDE_through_integrals} The solution to \eqref{eq_discrete_PDE}, \eqref{eq_discrete_PDE_boundary} has the form
\begin{multline}
\label{eq_discrete_PDE_solution}
 \Phi(X,Y)=\chi(0) \Rd(X,Y;0,0)+\sum_{y=1}^Y \Rd(X,Y;0,y) \bigl(\psi(y)-b_2 \psi(y-1)\bigr)\\+\sum_{x=1}^X \Rd(X,Y;x,0) \bigl(\chi(x)-b_1 \chi(x-1)\bigr)
  +\sum_{x=1}^X \sum_{y=1}^Y \Rd(X,Y;x,y) u(x,y).
\end{multline}
\end{theorem}

\begin{proof}
Directly from the definition, we see that the function $\Rd$ satisfies:
\begin{enumerate}
 \item  $\Rd(X+1,Y+1)-b_1 \Rd(X,Y+1)-b_2\Rd(X+1,Y)\\ +(b_1+b_2-1) \Rd(X,Y)=0$,

\smallskip
 \item  $[\Rd(X+1)-b_1\Rd(X)]_{y=Y}=0=[\Rd(x-1)-b_1\Rd(x)]_{y=Y}$,

\smallskip
 \item  $[\Rd(Y+1)-b_2\Rd(Y)]_{x=X}=0=[\Rd(y-1)-b_2\Rd(y)]_{x=X}$,

\smallskip
 \item $[\Rd(X,Y;x,y)]_{x=X,y=Y}=1$.
\end{enumerate}

We apply the difference operator $F\mapsto F(X+1,Y+1)-b_1 F(X,Y+1)-b_2 F(X+1,Y)+(b_1+b_2-1)
F(X,Y)$ to each of the four terms of \eqref{eq_discrete_PDE_solution} using the properties of
$\Rd$. The first term gives zero by the first property. The second term gives (using the first and
second properties)
\begin{multline}
 \sum_{y=1}^Y \bigl(\Rd(X+1,Y+1;0,y)-b_1\Rd(X,Y+1;0,y)-b_2\Rd(X+1,Y;0,y)\\ +(b_1+b_2-1) \Rd(X,Y;0,y) \bigr) \bigl(\psi(y)-b_2 \psi(y-1)\bigr)
 \\ + \bigl(\Rd(X+1,Y+1;0,Y+1)-b_1 \Rd(X,Y+1;0,Y+1)\bigr) \bigl(\psi(Y+1)-b_2 \psi(Y)\bigr)=0.
\end{multline}
The third term gives zero for similar reasons via the first and third properties.
The fourth term gives (using all four properties)
\begin{multline}
 \sum_{x=1}^X \sum_{y=1}^Y \bigl(\Rd(X+1,Y+1;x,y)-b_1 \Rd(X,Y+1;x,y)-b_2 \Rd(X+1,Y;x,y)\\ +(b_1+b_2-1) \Rd(X,Y;x,y) \bigr) u(x,y)
\\+ \sum_{x=1}^X \bigl( \Rd(X+1,Y+1;x,Y+1)-b_1 \Rd(X,Y+1;x,Y+1)\bigr) u(x,Y+1)
\\+ \sum_{y=1}^Y \bigl( \Rd(X+1,Y+1;X+1,y) - b_2 \Rd(X+1,Y; X+1,y) \bigr) u(X+1,y)
\\+ \Rd(X+1,Y+1;X+1,Y+1) u(X+1,Y+1)= u(X+1,Y+1).
\end{multline}
We conclude that \eqref{eq_discrete_PDE_solution} satisfies \eqref{eq_discrete_PDE}, and it remains
to check the boundary conditions.

At $X=0$, note that by the third property of $\Rd$, $\Rd(0,Y;0,y)=b_2^{-y}
\Rd(0,Y;0,0)$. Therefore, we have (using the fourth property as well)
\begin{multline}
 \Phi(0,Y)= \Rd(0,Y;0,0) \left(\psi(0) +\sum_{y=1}^Y b_2^{-y} \bigl(\psi(y)-b_2 \psi(y-1)\bigr) \right)\\=
 \Rd(0,Y;0,0) \psi(Y) b_2^{-Y}= \Rd(0,Y;0,Y)\psi(Y)=\psi(Y).
\end{multline}
At $Y=0$, by the second property, $\Rd(X,0;x,0)=b_1^{-x} \Rd(X,0;0,0)$, and thus,
\begin{multline}
 \Phi(X,0)= \Rd(X,0;0,0) \left(\chi(0) +\sum_{x=1}^X b_1^{-x} \bigl(\chi(x)-b_1 \chi(x-1)\bigr) \right)\\=
 \Rd(X,0;0,0) \chi(X) b_1^{-X}= \Rd(X,0;X,0)\chi(X)=\chi(X). \qedhere
\end{multline}
\end{proof}

\subsection{Solutions as path integrals: discrete case}

\label{Section_paths_discrete}

In this section we interpret the formula of Theorem \ref{Theorem_discrete_PDE_through_integrals} as an expectation of a certain path integral. Essentially, this is a development of a version of the Feynman-Kac formula for the difference equation \eqref{eq_discrete_PDE}.

\bigskip

\begin{figure}[t]
\begin{center}
{\scalebox{0.6}{\includegraphics{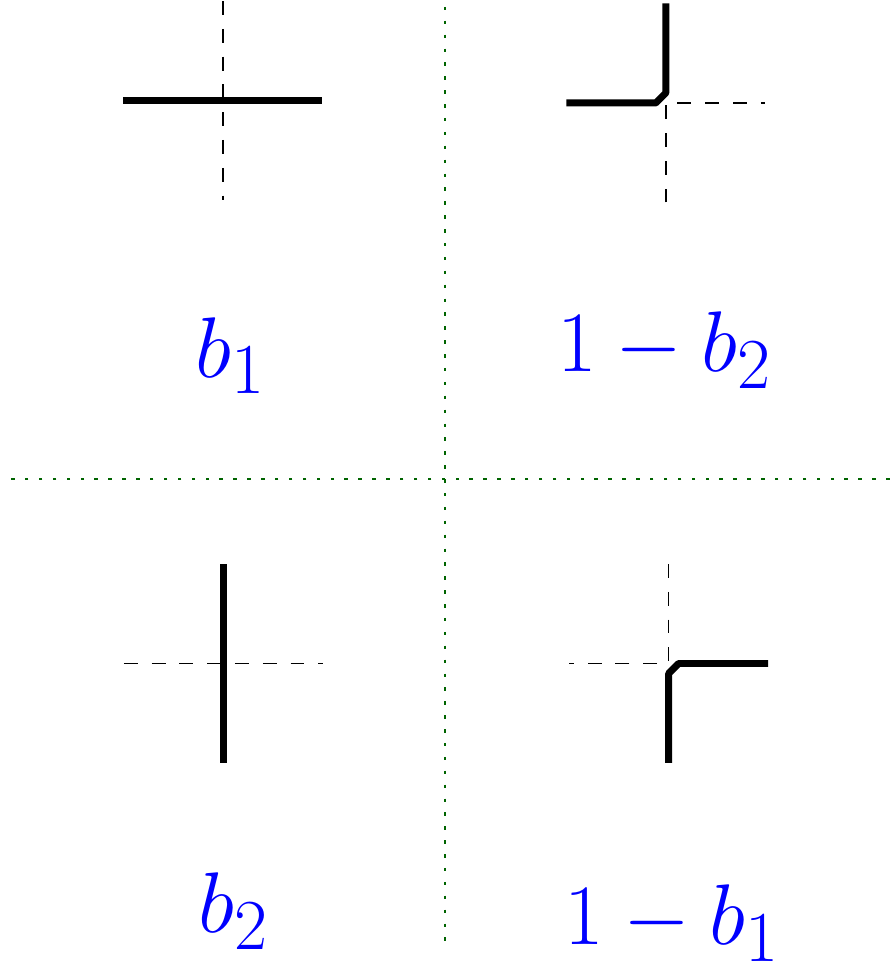}}}
 \caption{The weights of the random walk towards the origin.
 \label{Fig_weights_rev}}
\end{center}
\end{figure}

Consider a random path that starts at a point $(X,Y)$ in the positive quadrant and moves in the direction of decreasing $x$ and $y$. At each step, the path moves by one to the left, or down, or makes a turn. The choices are made according to probabilities of Figure \ref{Fig_weights_rev}. These weights are obtained from the weights of Figure \ref{Fig_weights} by central symmetry $(x,y)\mapsto (-x,-y)$. In other words, the weights of the straight segments remained the same, while the weights of corners were swapped in order to preserve stochasticity.

\begin{theorem}
\label{Theorem_discrete_as_expectation_boundary}
  Consider the equation \eqref{eq_discrete_PDE}, \eqref{eq_discrete_PDE_boundary} with
  $u(X,Y)=0$, $X,Y\ge 0$, and $\chi(0)=\psi(0)=0$. For convenience,
  extend $\chi(-a)=\psi(-a)=0$, $a>0$. The solution $\Phi(X,Y)$ admits
  the following stochastic formula. Take a (reversed, with probabilities of
   Figure \ref{Fig_weights_rev}) path leaving $(X+1,Y)$ to the left in horizontal direction, and
    let $\mathbf y$ denote the ordinate of the first point when it reaches the
     line $x=0$. Take another path leaving $(X,Y+1)$ down in vertical direction, and
      let $\mathbf x$ denote the abscissa of the first point when it
       reaches the line $y=0$. Then
  \begin{equation}
  \label{eq_solution_as_expectation}
   \Phi(X,Y)=\E \left[\psi(\mathbf y) \right]+ \E \left[\chi(\mathbf x) \right].
  \end{equation}
\end{theorem}

\smallskip

We will give a proof a little later, and now we will see what happens when $u\ne 0$.

Suppose that we are given a trajectory $\mathcal T$ of a path build out of the
blocks of Figure \ref{Fig_weights_rev}. For a point $(x,y)\in\mathbb Z \times
\mathbb Z$ we say that $(x,y)$ is \emph{weakly below} $\mathcal T$, if any of the
points of the square $(x-1/2,x+1/2)\times (y-1/2,y+1/2)$ is below (i.e., has a
smaller vertical coordinate and the same horizontal coordinate) than a point of the
path. Similarly, we say that $(x,y)$ is \emph{weakly to the left} from $\mathcal
T$, if any point of $(x-1/2,x+1/2)\times(y-1/2,y+1/2)$ is to the left of a point of
the path.

Now suppose that we are given two paths $\mathcal T_{-}$ and $\mathcal T_{|}$. Define
\begin{equation}
\label{eq_signed_indicator}
 \mathcal I_{\mathrm {between}}(x,y)= \mathbf 1_{(x,y) \text{ is weakly below } \mathcal T_-}+
 \mathbf 1_{(x,y)\text{ is weakly to the left from } \mathcal T_{|}}-1.
\end{equation}
In other words, $\mathcal I_{\mathrm {between}}(x,y)$ is $\pm 1$ between the paths $\mathcal
T_-$, $\mathcal T_+$ and vanishes otherwise. The sign depends on which path is higher. An
illustration of the values of this function is shown in Figure \ref{Fig_between}.

\begin{figure}[t]
\begin{center}
{\scalebox{0.6}{\includegraphics{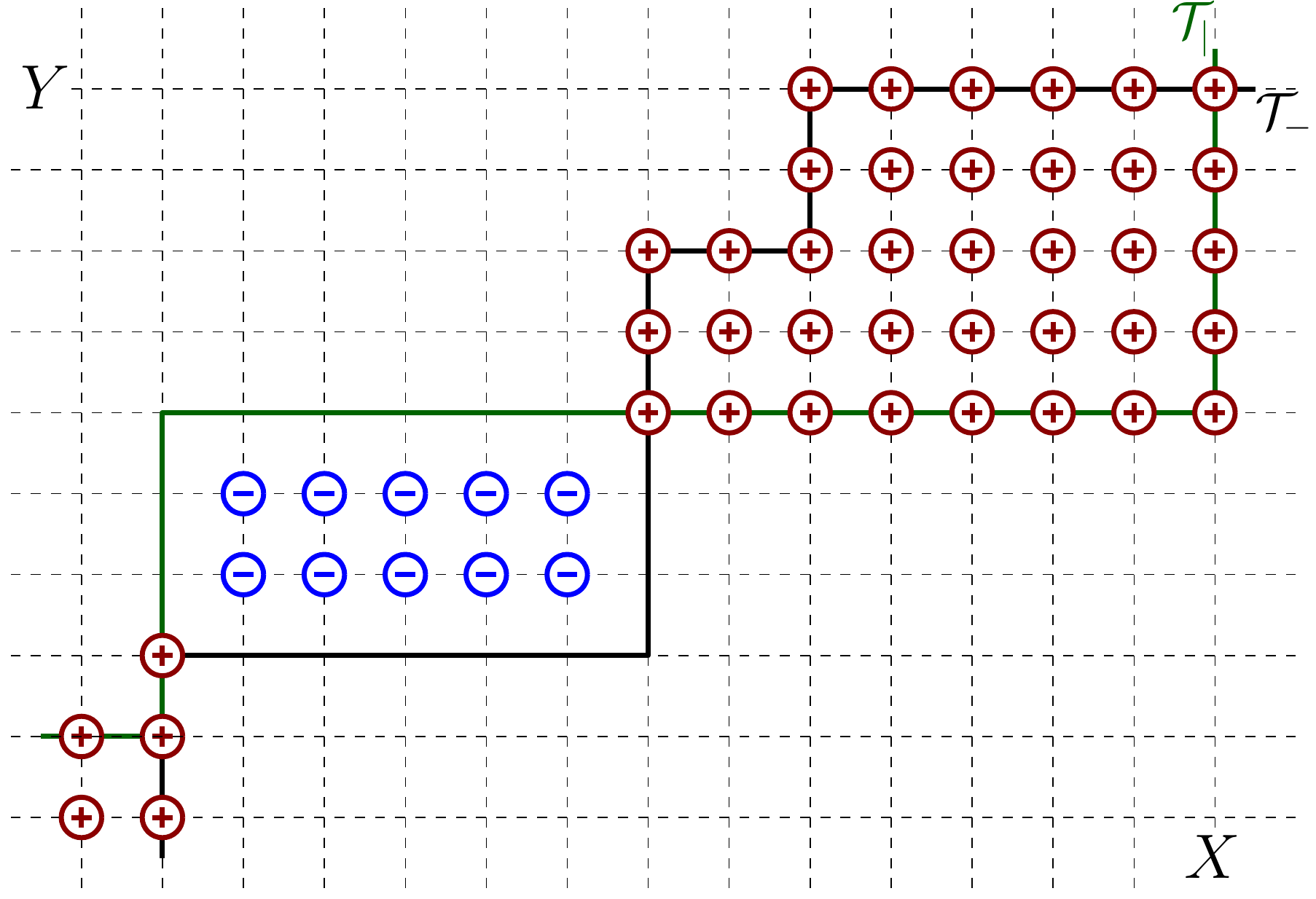}}}
 \caption{Two paths $\mathcal T_-$, $\mathcal T_|$, and the function $\mathcal I_{\mathrm {between}}(x,y)$: values $+1$ and $-1$ are shown by $\textcircled{+}$ and $\textcircled{-}$, respectively.
 \label{Fig_between}}
\end{center}
\end{figure}

\begin{theorem}
\label{Theorem_discrete_as_expectation_noise}
  Consider the equation \eqref{eq_discrete_PDE}, \eqref{eq_discrete_PDE_boundary} with
  $\chi(x)=\psi(y)=0$, $x,y\ge 0$.  The solution $\Phi(X,Y)$ admits
  the following stochastic formula. Take a (reversed, with probabilities of
   Figure \ref{Fig_weights_rev}) path $\mathcal T_-$ leaving $(X+1,Y)$ to the left in horizontal direction and
    another path $\mathcal T_|$ leaving $(X,Y+1)$ down in vertical direction. Then
  \begin{equation}
  \label{eq_solution_as_expectation_noise}
   \Phi(X,Y)=\E \left[ \sum_{x=1}^X \sum_{y=1}^Y u(x,y) \mathcal I_{\mathrm {between}}(x,y) \right].
  \end{equation}
  where we use the definition \eqref{eq_signed_indicator}.
   In words, $\Phi(X,Y)$ is the expected signed sum of all the inhomogeneities of \eqref{eq_discrete_PDE} between the paths.
\end{theorem}

 By linearity of the equation, the solution to \eqref{eq_discrete_PDE} when both $u$ and $\chi$, $\psi$ are non-vanishing  is the sum of the right--hand sides in \eqref{eq_solution_as_expectation}, \eqref{eq_solution_as_expectation_noise}.

\begin{corollary}
\label{Corollary_solution_expectation}
 In the notations of Theorem \ref{Theorem_discrete_as_expectation_boundary}, \ref{Theorem_discrete_as_expectation_noise} consider the case when both $u(x,y)$ and $\chi$, $\psi$ are non-vanishing. Then
  \begin{equation}
  \label{eq_solution_as_expectation_discrete_sum}
   \Phi(X,Y)=\E \left[\psi(\mathbf y) \right]+ \E \left[\chi(\mathbf x) \right]+\E \left[ \sum_{x=1}^X \sum_{y=1}^Y u(x,y) \mathcal I_{\mathrm {between}}(x,y) \right].
  \end{equation}
\end{corollary}

\begin{proof}[Proof of Theorem \ref{Theorem_discrete_as_expectation_boundary}]
 By linearity, it suffices to consider the case
 \begin{equation}
 \label{eq_x46}
  \chi\equiv 0,  \qquad \psi(y)=\begin{cases} 1,& y=y_0,\\
 0,&\text{otherwise}. \end{cases}
 \end{equation}
In this case the right--hand side of \eqref{eq_solution_as_expectation} becomes the probability of intersecting the line $x=1/2$ at point $(1/2,y_0)$. Let us compute this probability.

\smallskip

We start by considering a particular case of the stochastic six--vertex model (with the weights of Figure \ref{Fig_weights} at $\alpha=0$) when we have only one
path. In this case the expectation of the height function has a simple probabilistic meaning:
\begin{align}
\label{eq_height_one_path}
 \E&\left[ \frac{1-q^{H(x+1,y)}}{1-q} \right]\\&\notag =
 {\rm Prob}\bigl(\text{ the path passes to the right from }(x+1/2,y+1/2)\, \bigr)
\\\notag &= {\rm Prob}\bigl(\text{ the path passes below }(x+1/2,y+1/2)\, \bigr).
\end{align}
In this formula we think about the paths as having integer coordinates, and we introduced shifts by
$1/2$ to avoid ambiguity for the case when the path passes exactly through the point of interest.

Suppose that the path enters the positive quadrant through the point $(1,y_0)$ coming from the left. Then by Theorem
\ref{Theorem_4_point}, \eqref{eq_height_one_path} denoted as $F^-_{y_0}(X,Y)$ (the superscript
$^-$ indicates that the path enters horizontally) solves
\begin{equation}
\label{eq_x41}
 F^-_{y_0}(X+1,Y+1)-b_1 F^-_{y_0}(X,Y+1)-b_2 F^-_{y_0}(X+1,Y)+(b_1+b_2-1) F^-_{y_0}(X,Y)=0,
\end{equation}
with
\begin{equation}
  F^-_{y_0}(X,0)=0,\qquad F^-_{y_0}(0,Y)=\begin{cases} 0,& Y<{y_0},\\ 1, & Y\ge {y_0}. \end{cases}
\end{equation}
 Theorem \ref{Theorem_discrete_PDE_through_integrals} gives a closed formula:
\begin{equation}
\label{eq_x42}
F^-_{y_0}(X,Y)=\Rd(X,Y;0,y_0)+(1-b_2) \sum_{y=y_0+1}^Y \Rd(X,Y;0,y).
\end{equation}
Consider the difference
$$
 P_{-,-}(0,y_0; X,Y):=F^-_{y_0}(X,Y)-F^-_{y_0}(X,Y-1).
$$
Relation \eqref{eq_height_one_path} implies that it computes the probability that the path, which entered the quadrant
horizontally at $(1,y_0)$, ends horizontally at $(X+1/2,Y)$ (i.e., the path \emph{enters into}
$(X+1,Y)$ from the left). Using \eqref{eq_x42} we get
\begin{multline}
\label{eq_x43}
 P_{-,-}(0,y_0; X,Y)=(1-b_2)\sum_{y=y_0+1}^{Y-1} (\Rd(X,Y;0,y)-\Rd(X,Y-1;0,y))  \\ +(1-b_2)\Rd(X,Y;0,Y)+ \Rd(X,Y;0,y_0)-\Rd(X,Y-1;0,y_0).
\end{multline}
Since $\Rd(X,Y;x,y)$ depends only on differences $X-x$, $Y-y$, the sum telescopes and \eqref{eq_x43} simplifies to
\begin{equation}
\label{eq_x45}
 P_{-,-}(0,y_0; X,Y)= \Rd(X,Y;0,y_0)-b_2\Rd(X,Y;0,y_0+1).
\end{equation}

By translation invariance, the same formula holds for the path which starts not by
entering from the left into $(1,y_0)$, but into an arbitrary point $(x_0+1,y_0)$:
\begin{equation}
\label{eq_x44}
 P_{-,-}(x_0,y_0; X,Y)= \Rd(X,Y;x_0,y_0)-b_2\Rd(X,Y;x_0,y_0+1).
\end{equation}
Note that this holds for $Y=y_0$ as well, if we agree that ${\Rd(X,y_0;x_0,y_0+1)=0}$.

%We proceed to the difference in $X$--direction:
%$$
% P_{-,|}(0,y_0; X,Y):=F_{y_0}(X,Y)-F_{y_0}(X+1,Y).
%$$
%\eqref{eq_height_one_path} implies that it computes the probability that the path, which started
%horizontally at $(1/2,y_0)$, ends vertically at $(X,Y-1/2)$ (i.e.\ the path \emph{enters} into
%$(X,Y)$ from below). Using \eqref{eq_x42} we compute
%\begin{multline}
% P_{-,|}(0,y_0; X,Y)=\sum_{y=y_0+1}^{Y-1} (\Rd(X,Y-1;0,y)-\Rd(X+1,Y-1;0,y)) (1-b_2) \\ +\Rd(X,Y-1;0,y_0)-\Rd(X+1,Y-1;0,y_0)
%\end{multline}
%Unfortunately, there is no telescoping structure in this formula and we will keep it in this form
%for now. \textcolor{green}{[In fact we will never use this formula. Should I remove it?]}

By symmetry, we can also obtain similar formulas for the case when the path starts by entering  from below into a point $(x_0,y_0+1)$. %We have
%\begin{equation}
%\label{eq_x43}
%F^|_{x_0}(X,Y)=\Rd(X,Y;x_0,0)+(1-b_2) \sum_{x=x_0+1}^X \Rd(X,Y;x,0)
%\end{equation}
The probability of this path entering into $(X,Y+1)$ from below is
\begin{equation}
\label{eq_x49}
 P_{|,|}(x_0,y_0; X,Y)= \Rd(X,Y;x_0,y_0)-b_1\Rd(X,Y;x_0+1,y_0).
\end{equation}
%\begin{multline}
% P_{|,-}(0,y_0; X,Y)= \sum_{x=x_0+1}^{X-1} (\Rd(X-1,Y;x,0)-\Rd(X-1,Y+1;x,0)) (1-b_1) \\ +\Rd(X-1,Y;x_0,0)-\Rd(X-1,Y+1;x_0,0)
%\end{multline}

\bigskip

Let us return to proving \eqref{eq_solution_as_expectation} in the particular case \eqref{eq_x46}. We need to show that
\begin{equation}
\label{eq_x47}
\Phi(X,Y)=P_{-,-}(-X,-Y; 0,-y_0).
\end{equation}
Note that we changed the signs of the coordinates to reflect the fact that the walk in the direction of growing $(x,y)$ with weights of Figure \ref{Fig_weights} differs from the one from Figure \ref{Fig_weights_rev} that we need to use.

The definition of $P_{-,-}$ readily implies that \eqref{eq_x47} satisfies the boundary condition \eqref{eq_discrete_PDE_boundary}, \eqref{eq_x46}. On the other hand, note that since $\Rd(X,Y;x,y)$ depends only on $(X-x)$, $(Y-y)$, the first property in the proof of Theorem \ref{Theorem_discrete_PDE_through_integrals} is equivalent to
\begin{multline}
\label{eq_x48}
 \Rd(X,Y;x-1,y-1)-b_1 \Rd(X,Y;x,y-1)-b_2\Rd(X,Y;x-1,y)\\+(b_1+b_2-1) \Rd(X,Y;x,y)=0.
\end{multline}
Combining \eqref{eq_x45} with \eqref{eq_x48}, we conclude that \eqref{eq_x47} satisfies \eqref{eq_discrete_PDE}.
\end{proof}

\begin{proof}[Proof of Theorem \ref{Theorem_discrete_as_expectation_noise}]
By linearity, it suffices to prove \eqref{eq_solution_as_expectation_noise} for the case when $u(x,y)$ is nonzero only at one point, where it equals $1$. In this case, by Theorem \ref{Theorem_discrete_PDE_through_integrals} the solution is
$$
 \Phi(X,Y)=\mathbf 1_{X\ge x_0} \mathbf 1_{Y\ge y_0} \Rd(X,Y;x_0,y_0).
$$

When either $X<x_0$ or $Y<y_0$, matching with \eqref{eq_solution_as_expectation_noise} is immediate, so we will only consider the case $X\ge x_0$, $Y\ge y_0$. Then \eqref{eq_solution_as_expectation_noise} suggests that we need to compute the expectation of $\mathcal I_{\mathrm {between}}(x_0,y_0)$.

Using the notations from the proof of Theorem \ref{Theorem_discrete_as_expectation_boundary} and \eqref{eq_x44}, \eqref{eq_x49}, we have
\begin{multline}
\label{eq_x51}
\E[\mathcal I_{\mathrm {between}}(x_0,y_0)+1]\\ =  \sum_{y=y_0}^Y P_{-,-}(-X,-Y; -x,-y)+\sum_{x=x_0}^X P_{|,|}(-X,-Y; -x,-y)\\=
\sum_{y=y_0}^Y (\Rd(-x_0,-y; -X,-Y) -  b_2 \Rd(-x_0,-y;-X,-Y+1))\\ +
\sum_{x=x_0}^X (\Rd(-x,-y_0; -X,-Y)- b_1 \Rd(-x,-y_0;-X+1,-Y))\\=
\sum_{y=y_0}^Y (\Rd(X,Y; x_0,y) -  b_2 \Rd(X,Y;x_0,y+1))\\ +
\sum_{x=x_0}^X (\Rd(X,Y;x,y_0)- b_1 \Rd(X,Y;x+1,y_0)),
\end{multline}
where we agree that $\Rd(X,Y;X+1,y_0)=\Rd(X,Y; x_0,Y+1)=0$.

On the other hand, let us sum \eqref{eq_x48} over $x=x_0+1\dots,X+1$, $y=y_0+1\dots,Y+1$ except for $(x,y)=(X+1,Y+1)$. Note that the formula \eqref{eq_Discrete_R} for $\Rd$ makes sense even when $x>X$, and moreover it vanishes identically. This implies that \eqref{eq_x48} still holds for such $x$ (as its proof is just a computation showing identical vanishing of the integrand). Similarly, we can deform the contour in \eqref{eq_Discrete_R}, so that it encloses $-\frac{1}{b_1(1-b_2)}$ instead of $-\frac{1}{b_2(1-b_1)}$. Then the result vanishes for $y>Y$, and therefore, \eqref{eq_x48} holds again. Note however, that we can not take both $x>X$ and $y>Y$ simultaneously, as then the argument no longer works.

We get
\begin{multline}
\label{eq_x50}
 \Rd(X,Y;x_0,y_0) +(1-b_1)\sum_{x=x_0+1}^{X} \Rd(X,Y;x,y_0)+(1-b_2)\sum_{y=y_0+1}^{Y} \Rd(X,Y;x_0,y)\\ -\Rd(X,Y;X,Y)=0.
\end{multline}
Recall that $\Rd(X,Y;X,Y)=1$. Thus, \eqref{eq_x51} turns into
$$
 \E[\mathcal I_{\mathrm {between}}(x_0,y_0)+1] = 1+\Rd(X,Y; x_0,y_0).\qedhere
$$
\end{proof}

\subsection{Solutions as path integrals: continuous case}
\label{Section_paths_continuous}

In this section we develop a continuous analogue of Section \ref{Section_paths_discrete} and
present the Feynman-Kac formula for the solution of the telegraph equation \eqref{eq_general_PDE}.

The basic stochastic object is the \emph{persistent Poisson random walk}. It starts from $(X,Y)\in \mathbb R_{>0}^2$ and moves towards the origin along vertical and horizontal directions. Whenever it moves horizontally, it turns down with intensity $\beta_1>0$. Whenever it moves vertically, it turns to the left with intensity $\beta_2>0$. This process is the limit of the random walks of Section \ref{Section_paths_discrete} with weights of Figure \ref{Fig_weights_rev} in the limit regime \eqref{eq_limit_regime}. There is one choice to be made --- when the path leaves $(X,Y)$ it can start by going horizontally or vertically. We denote the resulting (random) trajectories through $\mathcal T_-$ and $\mathcal T_|$, respectively.

\begin{theorem}
 \label{Theorem_telegraph_as_expectation} Consider the telegraph equation \eqref{eq_general_PDE}, \eqref{eq_inhomogeneous_PDE_boundary}. Assume that $\psi(0)=\chi(0)=0$ and extend these functions to negative arguments as identical zeros.
  The solution $\phi(X,Y)$ admits the following stochastic formula. Consider two (independent) persistent Poisson paths $\mathcal T_-$ and $\mathcal T_|$, leaving $(X,Y)$ horizontally and vertically, respectively. Let $\mathbf y$ be the ordinate of the first intersection of $\mathcal T_-$ with the $y$--axis, and let $\mathbf x$ be the abscissa of the first intersection of $\mathcal T_|$ with the $x$--asix. Further, for any point $(x,y)\in\mathbb R^2_{>0}$, define
  $$
   \mathcal I_{\mathrm {between}}(x,y)=\begin{cases} 1,& (x,y) \text{ is between } \mathcal T_- \text{ and } \mathcal T_| \text{ with } \mathcal T_-\text{ above},\\
   -1, & (x,y) \text{ is between } \mathcal T_- \text{ and } \mathcal T_| \text{ with } \mathcal T_-\text{ below},\\
   0, &\text{ otherwise.}
   \end{cases}
  $$
  Then
  \begin{equation}
  \label{eq_solution_as_expectation_continuous}
  \phi(X,Y)=\E \chi(\mathbf x)+\E \psi(\mathbf y)+\E\left[ \int_0^X\int_0^Y  \mathcal I_{\mathrm {between}}(x,y) u(x,y) dx dy\right].
  \end{equation}
\end{theorem}
\begin{proof}
 Consider the limit transition \eqref{eq_limit_regime} with simultaneous rescaling by $L$ of the coordinates $x$ and $y$, boundary conditions $\chi$, $\psi$, the right--hand side $u(x,y)$, and the solutions $\Phi(X,Y)$. Then Corollary \ref{Corollary_solution_expectation} and the straightforward limit relation
 $$
   \lim_{L\to\infty} \Rd(LX,LY;Lx,Ly)=\R(X,Y;x,y),
 $$
 implies that the solution to the difference relation \eqref{eq_discrete_PDE} turns into the solution to the telegraph equation \eqref{eq_general_PDE}. Simultaneously, the same limit transition turns the random walks of Section \ref{Section_paths_discrete} into persistent Poisson random walks.

 We conclude that \eqref{eq_solution_as_expectation_continuous} is the $L\to\infty$ limit of \eqref{eq_solution_as_expectation_discrete_sum}.
\end{proof}

\section{Law of Large Numbers through four point relation}

\label{Section_LLN_4p}

From now on we set $\alpha=0$ and study only the stochastic six-vertex model. Our aim is to extend
Theorem \ref{Theorem_LLN} to arbitrary boundary conditions. Our main technical tool is the
four point relation of Section \ref{Section_four_point}.

\subsection{LLN for general boundary conditions}

\begin{theorem}
\label{Theorem_LLN_general}
 Fix $a,b>0$, take two 1-Lipschitz monotone functions $\chi:[0,a]\to\mathbb R$, $\psi:[0,b]\to \mathbb R$
 such that $\chi(0)=\psi(0)$. Suppose that the boundary condition in the stochastic six-vertex model
 is chosen so that as $L\to\infty$, $\frac1L H(Lx,0)\to
 \chi(x)$ and $\frac{1}{L} H(0,Ly)\to \psi(y)$ uniformly on $x\in[0,a]$, $y\in [0,b]$.

 Define the function $\q^{\h}:[0,a]\times[0,b]\to\mathbb R$ as the solution to the PDE
  \begin{multline}
 \label{eq_limit_shape_PDE}
  \frac{\partial^2}{\partial x \partial y} \bigl(\q^{\h(x,y)} \bigr)+\beta_2\frac{\partial}{\partial x} \left( \q^{\h(x,y)} \right)
  +\beta_1 \frac{\partial}{\partial y} \bigl(\q^{\h(x,y)}\bigr)=0, \\ \q^{\h(x)}=\chi(x),\quad \q^{\h(0,y)}=\psi(y,0).
 \end{multline}

 Then the height function of the stochastic six-vertex model ($\alpha=0$) satisfies the Law of
 Large Numbers in the limit regime \eqref{eq_limit_regime}:
 \begin{equation}
 \label{eq_uniform_convergence}
  \lim_{L\to\infty} \sup_{(x,y)\in[0,a]\times[0,b]} \left|\frac{1}{L} H(Lx,Ly)-\h(x,y)\right|=0, \qquad \text{ in probability.}
 \end{equation}
\end{theorem}
\begin{remark} Proposition \ref{Proposition_PDE_solution} says that
\eqref{eq_limit_shape_PDE} has a unique solution in the quadrant $x,y\ge 0$ for any continuously
differentiable boundary data on the lines $x=0$, $y=0$. When the boundary data are less regular,
one has to consider the integrated form \eqref{eq_integrated_telegraph} of the equation instead.
Note that $\h(x,0)$ and $\h(0,y)$ must be $1$--Lipschitz by the definition of the height
function.
\end{remark}
\begin{remark}
\label{Remark_back_to_six}
 In terms of the partial derivatives of $\h(x,y)$ and $\q$, $\s$ parameters, the equation \eqref{eq_limit_shape_PDE}
 turns into a \emph{non-linear} PDE
 \begin{equation}
 \label{eq_limit_shape_PDE_2}
  \frac{1}{\ln(\q)} \h_{xy}+\h_x\h_y+\frac{1}{\s-1} \h_x+\frac{\s}{\s-1}\h_y=0.
 \end{equation}
 In terms of $\rho=\h_x$ it gives (writing \eqref{eq_limit_shape_PDE_2} as an expression of $\h_y$ through $\h_x$, $\h_{xy}$ and differentiating  with respect to $x$)
 \begin{equation}
 \label{eq_limit_shape_PDE_3}
  \frac{1}{\ln(\q)}\left( \rho_{xy}+\frac{(1-\s)\rho_x\rho_y}{\s+(\s-1)\rho}\right)
  + \rho_x \cdot\frac{\s}{\s-1} \cdot \frac{1}{\s+(\s-1)\rho}+\rho_y \cdot \frac{1}{\s-1}\cdot (\s+(\s-1)\rho).
 \end{equation}
 As $\q\to 0$, \eqref{eq_limit_shape_PDE_3} becomes the equation for the limit shape of the
 stochastic six-vertex model discussed in \cite{RS}, in agreement with Proposition \ref{Proposition_q_0} above.

 Another limit is $\s\to 1$ with fixed $\q$, which turns \eqref{eq_limit_shape_PDE_2} into
 $\h_x+\h_y=0$. The limit shape $\h$ becomes constant along the lines $x-y=const$.
\end{remark}

\begin{proof}[Proof of Theorem \ref{Theorem_LLN_general}]
 The function $\frac{1}{L}H(Lx,Ly)$ is monotone and $1$--Lipschitz in each of its variables.
 Therefore, by the Arzela--Ascoli theorem, the sequence of functions $\E q^{H(Lx,Ly)}$ has subsequential limits
 (with respect to supremum norm topology on continuous functions in $[0,a]\times[0,b]$) which are also
 Lipschitz. Let $\tilde \h(x,y)$ be one of such limits.
 Taking the expectation of \eqref{eq_4_point_relation_integrated}, we obtain
  \begin{multline}
  \label{eq_4_point_integrated_expectation}
  -(1-b) \sum_{x=1}^{LX-1} \E q^{H(x,0)} - (1-bq) \sum_{y=1}^{LY-1}   \E  q^{H(0,y)}\\ +(1-b) \sum_{x=1}^{LX-1} \E q^{H(x,LY)}
  + (1-bq) \sum_{y=1}^{LY-1}   \E q^{H(LX,y)}
 \\ + (b+bq-1)  \E q^{H(0,0)}- bq \cdot  \E q^{H(LX,0)} - b \cdot  \E q^{H(0,LY)} +  \E q^{H(LX,LY)}
 =0.
\end{multline}
 Sending $L\to\infty$ in \eqref{eq_4_point_integrated_expectation}, we get for all $0\le X\le a$, $0\le Y\le b$
  \begin{multline}
  \label{eq_4_point_integrated_limit}
  -\beta \int_0^X \q^{\tilde \h(x,0)} dx - (\beta-\ln(\q)) \int_0^Y \q^{\tilde \h(0,y)}dy +\beta \int_0^X \q^{\tilde \h(x,Y)}dx
  \\ + (\beta-\ln(\q)) \int_0^Y   \q^{\tilde \h(X,y)}dy
 -  \q^{\tilde \h(0,0)}- \q^{\tilde \h(X,0)} - \q^{\tilde \h(0,Y)} +  \q^{\tilde \h(X,Y)}
 =0.
\end{multline}
By Proposition \ref{Proposition_integrated_PDE_solution}, the integral equation
\eqref{eq_4_point_integrated_limit} has a unique solution. Hence, all limiting points $\tilde \h$
coincide with a unique limit $\h$, and $\q^\h$ solves \eqref{eq_limit_shape_PDE}.

\smallskip

So far we have shown that the expectation $\E q^{H}$ converges to $\q^{\h}$, and next we show that
the fluctuations decay to $0$.

\smallskip

 Set $U(x,y)=q^{H(Lx,Ly)}-\E q^{H(Lx,Ly)}$. Subtracting
 \eqref{eq_4_point_integrated_expectation} from \eqref{eq_4_point_relation_integrated}, we obtain
 \begin{multline}
  \label{eq_4_point_integrated_error}
  U(X,Y)+(1-b) \sum_{x=1}^{LX-1} U(x/L,Y)
  + (1-bq) \sum_{y=1}^{LY-1}   U(X,y/L)
 =\sum_{x=1}^{LX}\sum_{y=1}^{LY} \xi(x,y).
\end{multline}
 We claim that the maximum of right--hand side of \eqref{eq_4_point_integrated_error} over $(X,Y)\in[0,a]\times [0,b]$
 converges to $0$ in probability as $L\to\infty$. Indeed, consider the function
 $$
 V(X,Y)=\sum_{x=1}^{LX}\sum_{y=1}^{LY} \xi(x,y).
$$
 Since $U(X,Y)$, $(X,Y)\in[0,a]\times[0,b]$, is Lipschitz, \eqref{eq_4_point_integrated_error} implies that so
 is $V(X,Y)$. Thus, it suffices to show that for  \emph{some fixed} $X$ and $Y$, $V(X,Y)\to 0$ in
 probability. Using \eqref{eq_4_point_no_correlation}, see Remark \ref{Remark_noise_uncor},  we
 get
 \begin{equation}
  \label{eq_error_variance}
   \E [V(X,Y)]^2= \sum_{x=1}^{LX}\sum_{y=1}^{LY} \E [\xi(x,y)]^2
 \end{equation}
 We further use \eqref{eq_4_point_covariance} to compute each term of the right-hand side. Note
 that $|\Delta_x|<C(1-q)$, $|\Delta_y|<C(1-q)$ for a constant $C>0$ which depends only on $a,b$. It
 follows that as $L\to\infty$, $\E[\xi(x,y)]^2\le \mathrm{const}\cdot L^{-3}$ and \eqref{eq_error_variance} goes to
 $0$ as $\mathrm{const}\cdot L^{-1}$. Thus, $V(X,Y)$ converges to $0$ in probability.

 \smallskip

 The uniformly bounded random functions $U(X,Y)$ are uniformly Lipschitz on $[0,a]\times[0,b]$ as $L\to\infty$.
 Therefore, their laws are tight (in Skorohod topology) as $L\to\infty$, see, e.g., \cite[Corollary 3.7.4]{EK}. Any
 subsequential limit $\tilde U$ has continuous trajectories and must solve the $L=\infty$ version of \eqref{eq_4_point_integrated_error},
 which reads
\begin{multline}
  \label{eq_4_point_integrated_error_limit}
  \tilde U(X,Y)+ \beta_1 \int_0^X \tilde U(x,Y) dx
  + \beta_2 \int_0^Y \tilde U(X,y)dy
 =0, \quad 0\le x,y\le M.
\end{multline}
By Proposition \ref{Proposition_integrated_PDE_solution}, the only solution to
\eqref{eq_4_point_integrated_error_limit} is $\tilde U\equiv 0$. Thus, the law of $U(X,Y)$,
$(X,Y)\in[0,a]\times[0,b]$, converges to the law of the zero function.

We have thus shown that $\sup_{(x,y)\in[0,a]\times [0,b]}| q^{H(Lx,Ly)}-\q^\h(x,y)|\to 0$ in
probability as $L\to\infty$, which implies \eqref{eq_uniform_convergence}.
\end{proof}

\begin{remark} An alternative way to prove Theorem \ref{Theorem_LLN_general} is to use Theorems
\ref{Theorem_4_point} and \ref{Theorem_discrete_PDE_through_integrals} to represent $q^{H}$ through
the  Riemann function. The convergence of the discrete Riemann function to its continuous
counterpart of Theorem \ref{Theorem_general_solution} would then imply the description of the
limiting profile through the telegraph equation.
\end{remark}

\subsection{Consistency check}

\label{Section_consistency}

We would like to directly see that the result of Theorem \ref{Theorem_LLN_general}
complemented with formulas for the solution of Theorem \ref{Theorem_general_solution} matches the
contour integral expression of Theorem \ref{Theorem_LLN} at $\alpha=0$.

Let us find formulas for the solution to \eqref{eq_general_PDE} with specific
boundary condition. We take $u(X,Y)=0$, $\phi(X,0)=\q^{-p_1
X}=\exp(-(\beta_1-\beta_2)p_1 X)$, $\phi(0,Y)=\q^{p_2 Y}=\exp((\beta_1-\beta_2)p_2
Y)$ for two constants $p_1,p_2$. Then the solution is
\begin{multline}
%\label{eq_inhom_solution_integral}
 2\pi \ii \phi(X,Y)=2\pi \ii \R(X,Y;0,0)\\ +2\pi \ii \int_0^Y  \R(X,Y; 0,y)  \bigl(p_2(\beta_1-\beta_2) + \beta_2 \bigr)\exp((\beta_1-\beta_2)p_2 y) dy
\\+2\pi \ii \int_0^X \R(X,Y; x,0) \bigl(-p_1(\beta_1-\beta_2)+\beta_1\bigr)\exp(-(\beta_1-\beta_2)p_1 x) dx
\end{multline}
Plugging in the definition of $\R$ and integrating in $x$ and $y$, this can be transformed to
(with the notation $p_i=\frac{\rho_i}{1+\rho_i}$, so that $\rho_i=\frac{p_i}{1-p_i}$)
\begin{multline}
\label{eq_x16}
 \oint_{-\beta_1} \exp\left[
 (\beta_1-\beta_2) \left(-X \frac{z}{z+ \beta_2} + Y \frac{z}{z+\beta_1}\right)
 \right]
 \frac{ (\beta_2 \rho_1-\beta_1 \rho_2) dz}{(z-\beta_1
\rho_2)(z-\beta_2 \rho_1)}
\\-
  \oint\limits_{-\beta_1}\frac{\rho_1\beta_2+\beta_1}{z-\rho_1\beta_2}
\exp\left[
 (\beta_1-\beta_2)Y \frac{z}{z+\beta_1}
 \right] \left( \exp\left[
 -\frac{\rho_1}{1+\rho_1}(\beta_1-\beta_2) X \right]  \right) \frac{dz}{(z+\beta_1)}.
\end{multline}
Note that the residue at $z=\rho_1 \beta_2$ for both terms in \eqref{eq_x16}
coincides with
$$
 \exp\left[
 (\beta_1-\beta_2) \left(-X \frac{\rho_1}{1+ \rho_1} + Y \frac{\rho_1 \beta_2}{\rho_1 \beta_2+\beta_1}\right)
 \right].
$$
Thus, we can include $\rho_1 \beta_2$ into the integration contours. After that, the
second integral vanishes, and we get the final expression
\begin{equation}
\label{eq_solution_double_Bernoully}
 \oint_{-\beta_1,\, \rho_1 \beta_2} \exp\left[
 (\beta_1-\beta_2) \left(-X \frac{z}{z+ \beta_2} + Y \frac{z}{z+\beta_1}\right)
 \right]
 \frac{ (\beta_2 \rho_1-\beta_1 \rho_2) dz}{(z-\beta_1
\rho_2)(z-\beta_2 \rho_1)}.
\end{equation}

In particular, when $p_1=0$, $p_2=1$ (i.e., $\rho_1=0$, $\rho_2=+\infty$), we return to the domain wall boundary conditions, and the contour integral transforms into
\begin{equation}
\label{eq_solution_step}
 \oint_{-\beta_1,\, 0} \exp\left[
 (\beta_1-\beta_2) \left(-X \frac{z}{z+ \beta_2} + Y \frac{z}{z+\beta_1}\right)
 \right]
 \frac{ dz}{z},
\end{equation}
in agreement with Theorem \ref{Theorem_LLN} (cf.\ Remark \ref{Remark_DW_in_betas}). Note that $0$ is included in the contour, as here we deal with $\q^{\h(x,y)}$, while \eqref{eq_LLN_betas} corresponded to $\q^{\h(x,y)}-1$.

\section{CLT for general boundary conditions}

\label{Section_CLT_conjecture}

We say that a function $f:[a,b]\to \mathbb R$ is piecewise $C^1$--smooth, if it is continuous on the segment $[a,b]$ and there exists a finite partition  $a=x_0<x_1<\dots<x_n=b$ such that $f$ is continuously differentiable on each open interval $(x_{i-1},x_i)$, $1\le i \le n$, and its derivative has left and right limits at each point $x_i$, $1\le i \le n-1$.

The goal of this section is to prove the following statement.

\begin{theorem}\label{Conjecture_general_CLT} In the setting of Theorem \ref{Theorem_LLN_general}, assume additionally that the boundary conditions $\chi(x)$, $\psi(y)$ are piecewise $C^1$--smooth\footnote{We believe that the statement is true for arbitrary monotone and $1$--Lipschitz   $\chi$ and $\psi$. However, without the piecewise-smoothness condition the justification of convergence of the sum \eqref{eq_x58} to the integral
\eqref{eq_SPDE_covariance} needs additional technical efforts.}. Then the fluctuation field
$\sqrt{L}\bigl( q^{H(Lx,Ly)}-\E q^{H(Lx,Ly)} \bigr)$ converges as $L\to\infty$ (in the sense of convergence of finite-dimensional distributions) to a random Gaussian field
$\phi(x,y)$, $x,y\ge 0$, which solves
\begin{multline}
 \label{eq_SDE_betas}
  \phi_{xy}+\beta_1\phi_{y}
  + \beta_2 \phi_x
  \\= \eta  \cdot \sqrt{
  (\beta_1+\beta_2)
 \q^{\h}_x \q^{\h}_y
   +  (\beta_2-\beta_1) \beta_2\, \q^{\h} \q^{\h}_x
   - (\beta_2-\beta_1) \beta_1\, \q^{\h} \q^{\h}_y}
\end{multline}
with zero boundary conditions $\phi(x,0)=\phi(0,y)=0$, where $\eta$ is the two--dimensional white
noise, and $\q^\h$ is the limit shape afforded by Theorem \ref{Theorem_LLN_general}.
\end{theorem}
\begin{remark}
  The first version of this text stated Theorem \ref{Conjecture_general_CLT} as a conjecture; we also provided two heuristic arguments for it.  The conjecture was proved by Shen and Tsai a few months later, see  \cite{ST}. On the other hand, we later realized that one of our heuristic arguments could be also turned into a complete proof (different from the one in \cite{ST}); it is this proof that we include below. Our other heuristic argument can be found in the appendix.
\end{remark}
\begin{remark}
There are two ways to make sense of the solution to \eqref{eq_SDE_betas}. One can use the
integrated form \eqref{eq_integrated_telegraph} to smooth out the white noise. Alternatively, one
can use the formula for the solution of Theorem \ref{Theorem_general_solution}.
\end{remark}
\begin{remark}
If we denote $\phi(x,y)=\psi(x,y) \q^{\h(x,y)} \ln(\q),$ so that
$$
\psi(x,y)=\lim_{L\to\infty} \frac{H(Lx,Ly)-\E H(Lx,Ly)}{\sqrt{L}},
$$
then \eqref{eq_SDE_betas} is rewritten as
\begin{multline}
 \label{eq_SDE_betas_2}
  \psi_{xy}+\beta_1 \psi_y
  + \beta_2 \psi_x +(\beta_1-\beta_2) ( \psi_y \h_x + \psi_x  \h_y)
  \\= \eta  \cdot  \sqrt{
  (\beta_1+\beta_2)
 \h_x \h_y
   - \beta_2 \, \h_x
   + \beta_1\, \h_y}.
\end{multline}
\end{remark}
\begin{remark} We checked on a
computer the consistency between \eqref{eq_SDE_betas} and Theorem \ref{Theorem_CLT}. Namely, using
Theorem \ref{Theorem_general_solution}, the solution to \eqref{eq_general_PDE} has the covariance
\begin{multline}
\label{eq_x24}
 \mathrm{Cov}(\phi(X_1,Y_1),\phi(X_2,Y_2))\\ =\int\limits_0^{X_1 \wedge X_2}\,
 \int\limits_0^{Y_1\wedge Y_2} \R(X_1,Y_1;x,y) \R(X_2,Y_2; x,y) V^{\infty}(x,y)\, dx dy,
\end{multline}
with $V^{\infty}$ as in the second line of \eqref{eq_SPDE_covariance} below.
Plugging into \eqref{eq_x24} the contour integral expressions for $\R$ and the expressions for $\q^\h$ of Theorem \ref{Theorem_LLN} for
the domain wall boundary conditions we arrive at a $6$--fold integral expression. On the other
hand it has to be equal to the double contour integral of Theorem \ref{Theorem_CLT} (for points on
the same horizontal line, as in that theorem). We actually do not know how to verify it rigorously without using Theorem \ref{Conjecture_general_CLT}, but
evaluation of both expressions using Maple software (using symbolic computations of terms for converging series) shows that they are indeed equal.
\end{remark}

In the rest of this section we prove Theorem \ref{Conjecture_general_CLT}. The idea is to combine Theorems \ref{Theorem_4_point} and \ref{Theorem_discrete_PDE_through_integrals} with Martingale Central Limit theorem to reach the result. We detail only one-point convergence, as convergence of finite-dimensional distributions is proven in the same way by invoking multi-dimensional CLT instead of its one-dimensional counterpart.

\smallskip

 We combine Theorem \ref{Theorem_4_point} with Theorem \ref{Theorem_discrete_PDE_through_integrals}
 to get
\begin{multline}
\label{eq_x20}
  q^{H(X,Y)}=q^{H(0,0)} \Rd(X,Y;0,0)+
  \sum_{y=1}^Y \Rd(X,Y;0,y) \bigl(q^{H(0,y)}-b_2 q^{H(0,y-1)}\bigr)\\+
  \sum_{x=1}^X \Rd(X,Y;x,0) \bigl(q^{H(x,0)}-b_1 q^{H(x-1,0)}\bigr)
  +\sum_{x=1}^X \sum_{y=1}^Y \Rd(X,Y;x,y) \xi(x,y).
\end{multline}
The first three terms in \eqref{eq_x20} are deterministic, while the expectation of $\xi(x,y)$
vanishes. Therefore, rescaling $(X,Y)\mapsto (LX,LY)$, we get
\begin{equation}
\label{eq_x21}
  q^{H(LX,LY)}-\E q^{H(LX,LY)}=\sum_{x=1}^{LX} \sum_{y=1}^{LY} \Rd(LX,LY;x,y) \xi(x,y).
\end{equation}
 We now compute the $L\to\infty$ limit of the variance of \eqref{eq_x21}. Relation \eqref{eq_4_point_no_correlation} implies that $\xi(x,y)$ is uncorrelated noise; denote its variance by $V(x,y)$. Then
\begin{equation}
\label{eq_x52}
  \E(q^{H(LX,LY)}-\E q^{H(LX,LY)})^2=\E\left[\sum_{x=1}^{LX} \sum_{y=1}^{LY} [\Rd(LX,LY;x,y)]^2 V(x,y)\right].
\end{equation}
$V(x,y)$ is computed through \eqref{eq_4_point_covariance} to be
\begin{multline}
\label{eq_x53}
   V(x,y)=\bigr(qb(1-b)+b(1-qb)\bigl) \Delta_x \Delta_y
   \\ + b(1-qb)(1-q) q^{H(x,y)}  \Delta_x
   - b(1-b)(1-q)  q^{H(x,y)} \Delta_y.
\end{multline}
Choose a small parameter $\theta>0$. We split the summation domain $[1,LX]\times [1,LY]$ in \eqref{eq_x52} into disjoint squares of size $\theta L \times \theta L$ (and possibly smaller rectangles near the boundary of the domain). Take one such square $[LX_0,LX_0+L\theta]\times [LY_0,LY_0+L\theta]$ and consider the part of the sum corresponding to the indices $x$ and $y$ inside it. We first approximate the sum in the right--hand side of \eqref{eq_x52} without expectation and then take the expectation at the last step. Note that $|V(x,y)|<{\rm const}\cdot L^{-3}$, since $1-b$, $1-qb$, $1-q$, $\Delta_x$, and $\Delta_y$ all decay as $L^{-1}$. Therefore, the random variable under expectation in \eqref{eq_x52} multiplied by $L$ is uniformly bounded. Hence, convergence in probability would imply convergence of expectation in \eqref{eq_x52}.

Let us deal with the terms in the second line of \eqref{eq_x53} and concentrate on ${b(1-qb)(1-q) [q^{H(x,y)}  \Delta_x]}$. Since $H(x,y)$ is $1$--Lipschitz in both variables, using Theorem \ref{Theorem_LLN_general}, we get
 $$
 q^{H(x,y)}=q^{H(L X_0,LY_0)}+O(\theta)= q^{\h(X_0,Y_0)}+o(1)+O(\theta),
$$
where the remainder $o(1)$ tends to $0$ in probability as $L\to\infty$  uniformly in $(x,y)\in[1,LX]\times[1,LY]$, and remainder $O(\theta)$ is bounded from above by a deterministic constant tending to zero with speed $\theta$ as $\theta\to 0$. Also
$$
[\Rd(LX,LY;x,y)]^2 =[\R(X,Y;X_0,Y_0)]^2+O(\theta).
$$
Without loss of generality, we may assume that $q<1$. Then $\Delta_x$ is a \emph{positive number}, hence summations of $(o(1)+O(\theta))\cdot \Delta_x$ cause no problems: if real numbers $a_1,a_1,\dots,a_k$ are positive and real numbers $e_1,\dots,e_k$ satisfy $|e_i|<C$, then $|a_1 e_1+a_2 e_2+\dots+a_k e_k|\le C(a_1+\dots+a_k)$. We conclude that
\begin{multline}
 \sum_{\begin{smallmatrix} x\in  [LX_0, LX_0+L\theta]\\ y\in  [LY_0,LY_0+L\theta] \end{smallmatrix}
} [\Rd(LX,LY;x,y)]^2 b(1-qb)(1-q) q^{H(x,y)}  \Delta_x\\ = -L^{-2} \beta_2\ln(\q) [\R(X,Y;X_0,Y_0)]^2 \q^{\h(X_0,Y_0)} \\ \times \left(\sum_{y\in
 [LY_0,LY_0+L\theta]} (q^{H(LX_0+L\theta+1,y)}-q^{H(LX_0,y)})\right)\\ + (o(1)+O(\theta))\cdot L^{-2} \cdot (\theta L) \cdot \sup_{y} (q^{H(LX_0+L\theta+1,y)}-q^{H(LX_0,y)}).
\end{multline}
Applying Theorem \ref{Theorem_LLN_general} again, we get
\begin{multline}
\label{eq_x54}
 -\theta L^{-1} \beta_2\ln(\q) [\R(X,Y;X_0,Y_0)]^2 \q^{\h(X_0,Y_0)} \left( \q^{\h(X_0+\theta,Y_0)}-\q^{\h(X_0,Y_0)}\right)\\ +\theta L^{-1} o(1) + (o(1)+O(\theta))L^{-1}  \theta^2.
\end{multline}
Similarly, the asymptotic behavior of the sum of the terms arising from ${- b(1-b)(1-q)  q^{H(x,y)} \Delta_y}$ in the third line of \eqref{eq_x53} is
\begin{multline}
\label{eq_x55}
 \theta L^{-1} \beta_1\ln(\q) [\R(X,Y;X_0,Y_0)]^2 \q^{\h(X_0,Y_0)} \left( \q^{\h(X_0,Y_0+\theta)}-\q^{\h(X_0,Y_0)}\right)\\+\theta L^{-1} o(1) + (o(1)+O(\theta)) L^{-1}  \theta^2.
\end{multline}
The next step is to deal with the first line of \eqref{eq_x53}, which is more complicated due to the product $\Delta_x \Delta_y$. The key observation here is that the random variable $\Delta_x\Delta_y$ vanishes unless the vertex at $(x+1,y+1)$ has type $II$, as in Figure \ref{Fig_weights}; in the latter case $\Delta_x\Delta_y$ is $q^{2H(x,y)}(1-q)(1-q^{-1})$. Arguing similarly to the previous two cases, we then write
\begin{multline}
\label{eq_x56}
\sum_{ \begin{smallmatrix} x\in  [LX_0, LX_0+L\theta]\\ y\in  [LY_0,LY_0+L\theta] \end{smallmatrix}
} [\Rd(LX,LY;x,y)]^2 \bigr(qb(1-b)+b(1-qb)\bigl) \Delta_x \Delta_y\\ =(o(1)+O(\theta))\cdot L^{-1} \cdot \theta^2-L^{-3}[\R(X,Y;X_0,Y_0)]^2(\beta_1+\beta_1)\ln^2(\q) \q^{2\h(X_0,Y_0)}\\ \times \#\{\text{type }II\text{ vertices in }[LX_0, LX_0+L\theta]\times [LY_0,LY_0+L\theta]\}.
\end{multline}

\begin{figure}[t]
\begin{center}
%{\scalebox{0.45}{\includegraphics{Configuration_height_transpose.pdf}}} \qquad \qquad
{\scalebox{0.45}{\includegraphics{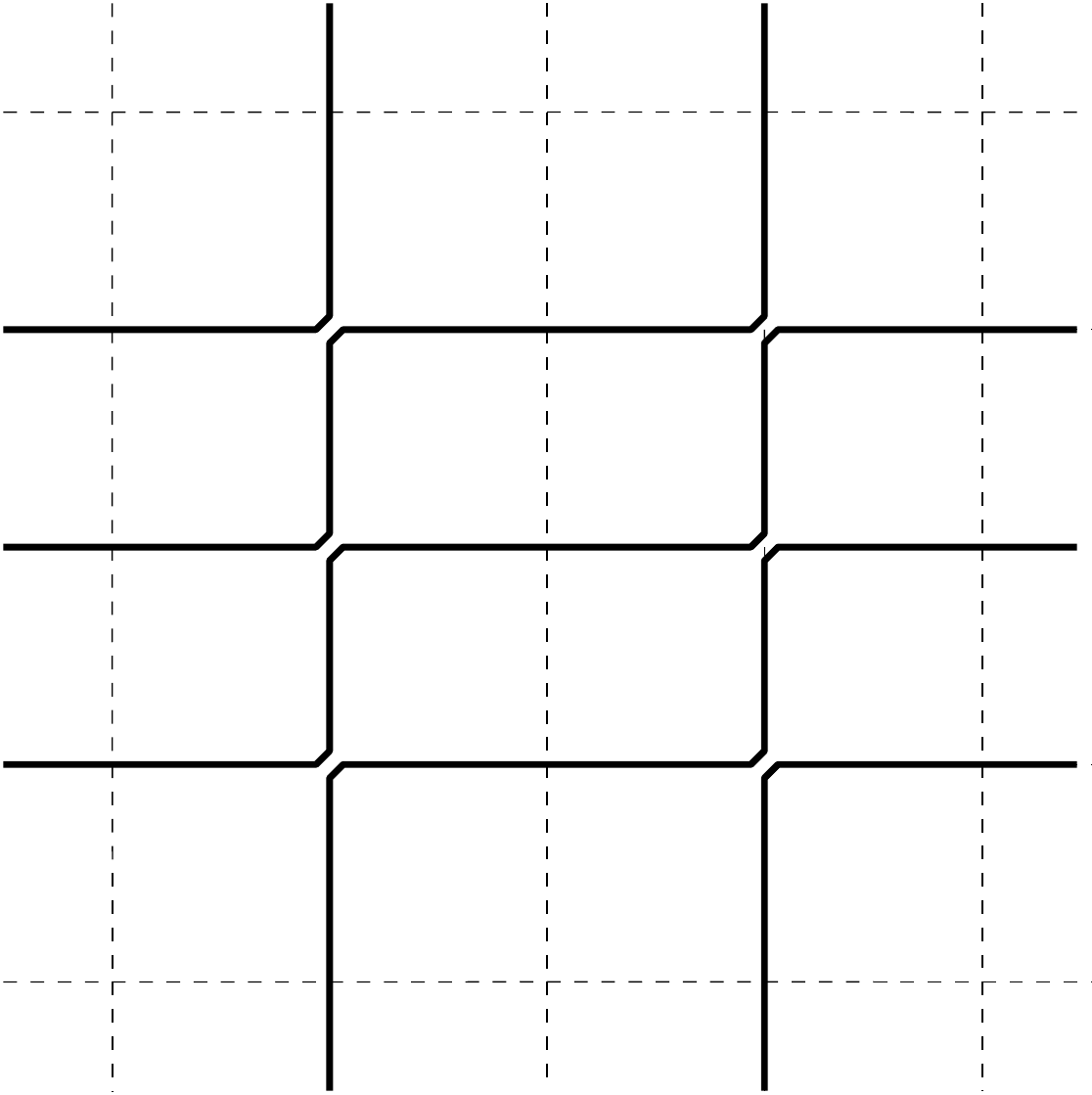}}}

 \caption{ When there are no corner-type vertices (types $V$ and $VI$), the configuration of the six-vertex model looks like a grid with the number of intersections (i.e., type $II$ vertices) equal to the product of the numbers of vertically and horizontally incoming paths. $6=2\times 3$ in the picture.
 \label{Fig_grid} }
\end{center}

\end{figure}
We would like to understand the last line of \eqref{eq_x56}. For that let $\square$ denote the square
$ [LX_0, LX_0+L\theta]\times [LY_0,LY_0+L\theta]$. Suppose that along the
bottom part of $\square$, $n$ paths are entering inside it, and along the left part of $\square$, $m$ paths are entering inside. Further, suppose that there are $\mathcal C$ vertices of types $V$ and $VI$ inside $\square$ --- these vertices represent ``corners''. Note that if $\mathcal C$=0, then the number of type $II$ vertices in $\square$ is $n\cdot m$. Indeed, if we reinterpret the type $II$ vertex as two paths transversally intersecting each other (rather than touching), then each of $n$ paths which entered vertically, must intersect each of the $m$ paths which entered horizontally, cf.\ Figure \ref{Fig_grid}. Let us view the general $\mathcal C>0$ case as a perturbation of $\mathcal C=0$. Then each of $\mathcal C$ corners might change the number of type $II$ vertices at most by $\theta L$, as adding this corner changes the behavior of only one path. The conclusion is that
\begin{equation}
 \left| (\text{Number of type $II$ vertices in }\square)-nm\right|\le \theta L \cdot \mathcal C.
\end{equation}
Let us now find an upper bound for $\mathcal C$. Let $U$ be the sum of $\theta^2 L^2$ i.i.d.\ Bernoulli random variables $\xi_i$ with $\mathrm{Prob}(\xi_i=1)=1-\min(b_1, b_2)$, $\mathrm{Prob}(\xi_i=0)=
\min(b_1, b_2)$. Then the definition of the stochastic six-vertex model implies that $\mathcal C\le U$ in the sense of stochastic dominance. In particular, $\E \mathcal C\le \mathrm{const}\cdot\frac{\theta^2 L^2}{L}$, and $\mathcal C\le \mathrm{const}\cdot \frac{\theta^2 L^2}{L}$ with probability tending to $1$ as $L\to\infty$.

We conclude that
\begin{equation}
 \left| (\text{Number of type $II$ vertices in }\square)-nm\right|\le \mathrm{const} \cdot \theta^3 L^2
\end{equation}
both in expectation and with high probability as $L\to\infty$. Finally,
$$
 n=H(LX_0, L Y_0)- H(LX_0+L\theta, L Y_0)=L(\h(X_0,Y_0)-\h(X_0+\theta,Y_0))+L\cdot o(1),
$$
$$
 m=L(\h(X_0,Y_0+\theta)-\h(X_0,Y_0))+L\cdot o(1),
$$
and  \eqref{eq_x56} turns into
\begin{multline}
\label{eq_x57}o(1)\cdot L^{-1}+ O(\theta^3)\cdot L^{-1} +L^{-1}[\R(X,Y;X_0,Y_0)]^2(\beta_1+\beta_1)\ln^2(\q) \q^{2\h(X_0,Y_0)} \\ \times
 (\h(X_0+\theta,Y_0)-\h(X_0,Y_0))\cdot (\h(X_0,Y_0+\theta)-\h(X_0,Y_0)).
\end{multline}
We now combine the terms from \eqref{eq_x54}, \eqref{eq_x55}, \eqref{eq_x57} and obtain
\begin{multline}
\label{eq_x58}
L \left[\sum_{x=1}^{LX} \sum_{y=1}^{LY} [\Rd(LX,LY;x,y)]^2 V(x,y)\right]\\ =  \sum_{0\le i \le X/\theta}\sum_{0\le j \le Y/\theta} [\R(X,Y; \theta i, \theta j)]^2  \Bigl[
-\theta \beta_2\ln(\q)\q^{\h(\theta i,\theta j)} \\ \times  \left( \q^{\h(\theta(i+1),\theta j)} -\q^{\h(\theta i,\theta j)}\right)+ \theta \beta_1\ln(\q)\q^{\h(\theta i,\theta j)} \left( \q^{\h(\theta i,\theta(j+1))}-\q^{\h(\theta i,\theta j)}\right)
\\ +
 (\beta_1+\beta_1)\ln^2(\q) \q^{2\h(\theta i,\theta j)}
 (\h(\theta(i+1),\theta j)-\h(\theta i,\theta j))\cdot (\h(\theta i,\theta(j+1))-\h(\theta i,\theta j))
 \Bigr]
 \\ + o(1) + O(\theta),
\end{multline}
where $o(1)$ is a random term which (for any fixed $\theta>0$) converges to $0$ in probability as $L\to\infty$, and $O(\theta)$ is a $\theta$-dependent random variable, whose absolute value is almost surely bounded by ${\rm const}\cdot \theta$.

At this point we first send $L\to\infty$ and then $\theta\to 0$. Note that the sum in the right-hand side of \eqref{eq_x58} is deterministic, so there is no randomness involved in the $\theta\to 0$ limit. Recall that $\q^\h$ solves the Telegraph equation \eqref{eq_limit_shape_PDE}. The boundary data $\chi(x)$, $\psi(y)$ are two piecewise $C^1$--smooth functions. Hence, due to integral representation of the solution \eqref{eq_inhom_solution_integral}, $\q^\h$ and therefore also $\h$ inherit smoothness: $\h_x$ is piecewise-continuous in $x$ and continuous in $y$; $\h_y$ is continuous in $x$ and piecewise-continuous in $x$. Hence, all the terms in \eqref{eq_x58} are smooth and as $\theta\to 0$ the sum converges to an integral. We conclude that
\begin{multline}
\label{eq_SPDE_covariance}
 \lim_{L\to\infty} L \left[\sum_{x=1}^{LX} \sum_{y=1}^{LY} [\Rd(LX,LY;x,y)]^2 V(x,y)\right]\\  =\int_0^X \int_0^Y dxdy\, [\R(X,Y; x, y)]^2  \Bigl[
- \beta_2\ln(\q)\q^{\h(x,y)}  \q^{\h(x,y)}_x\\ + \beta_1\ln(\q)\q^{\h(x,y)}  \q^{\h(x,y)}_y
+
 (\beta_1+\beta_1)\ln^2(\q) \q^{2\h(x,y)}
 \h_x(x,y)  \h_y(x,y)
 \Bigr],
\end{multline}
both in probability and in expectation.
Since $\ln(\q)=\beta_1-\beta_2$ and $\q^\h_x=\ln(\q) \q^\h_x$, $\q^\h_y=\ln(\q) \q^\h \h_y$, \eqref{eq_SPDE_covariance} matches the variance of the solution to \eqref{eq_SDE_betas} at point $(X,Y)$ when written in the form of Theorem \ref{Theorem_general_solution}.

\smallskip

 If instead of variance, we compute the $L\to\infty$ limit of the covariance of \eqref{eq_x21} at $(X,Y)=(X_1,Y_1)$ and $(X,Y)=(X_2,Y_2)$, then the argument is very similar. Indeed, since the noise $\xi(x,y)$ is uncorrelated, \eqref{eq_x52} is replaced with
\begin{multline}
\label{eq_x61}
  \E\bigl[(q^{H(LX_1,LY_1)}-\E q^{H(LX_1,LY_1)})(q^{H(LX_2,LY_2)}-\E q^{H(LX_2,LY_2)})\bigr]\\=\E\left[\sum_{x=1}^{L\min(X_1,X_2)} \, \sum_{y=1}^{L\min(Y_1,Y_2)} \Rd(LX_1,LY_1;x,y) \Rd(LX_2,LY_2;x,y) V(x,y)\right].
\end{multline}
Repeating the asymptotic analysis of \eqref{eq_x52}, we arrive at an analogue of \eqref{eq_SPDE_covariance}:
\begin{align*}
\label{eq_SPDE_covariance_multi}
 \lim_{L\to\infty} &L \left[\sum_{x=1}^{L\min(X_1,X_2)} \, \sum_{y=1}^{L\min(Y_1,Y_2)} \Rd(LX_1,LY_1;x,y) \Rd(LX_2,LY_2;x,y) V(x,y)\right]\\  &=\int_0^{\min(X_1,X_2)} \int_0^{\min(Y_1,Y_2)} dxdy\, \R(X_1,Y_1; x, y) \R(X_2,Y_2; x, y) \\ &\times \Bigl[
- \beta_2\ln(\q)\q^{\h(x,x)}  \q^{\h(x,y)}_x + \beta_1\ln(\q)\q^{\h(x,y)}  \q^{\h(x,y)}_y
\\&\quad \quad +
 (\beta_1+\beta_1)\ln^2(\q) \q^{2\h(x,y)}
 \h_x(x,y)  \h_y(x,y)
 \Bigr],
\end{align*}
 which matches the covariance of the solution to \eqref{eq_SDE_betas} at points $(X_1,Y_1)$ and $(X_2,Y_2)$ when written in the form of Theorem \ref{Theorem_general_solution}.

\smallskip

It remains to prove the asymptotic Gaussianity of \eqref{eq_x21}. Let us linearly order the integer points inside the rectangle $[1,LX]\times [1,LY]$ as follows: $(1,1)$, $(2,1)$, $(1,2)$, $(3,1)$, $(2,2)$, $(1,3)$, $(4,1)$, $(3,2)$, $(2,3)$, $(1,4)$,\dots, i.e., we sequentially trace the diagonals $x+y=\mathrm{const}$. Theorem \ref{Theorem_4_point} implies that then
 $\Rd(LX,LY;x,y) \xi(x,y)$ is a martingale difference in $(x,y)$, and we can apply the Martingale Central Limit Theorem, see, e.g., \cite[Section 3]{HH}. There are two conditions to check:
 \begin{enumerate}
  \item The conditional variance, which by Theorem \ref{Theorem_4_point} is given by
    $$
     \sum_{x=1}^{LX} \sum_{y=1}^{LY} [\Rd(LX,LY;x,y)]^2 V(x,y),
     $$
     with $V$ as in \eqref{eq_x53},
should have the same $L\to\infty$ behavior as the unconditional variance \eqref{eq_x52}, in the sense that the ratio tends to $1$ in probability.
 \item The Lindeberg condition should hold, which in our setting reads
 \begin{equation}
 \label{eq_x59}
  \lim_{L\to\infty} \sum_{x=1}^{LX} \sum_{y=1}^{LY} \E \Bigl[L \cdot  \xi^2(x,y)  I_{L\cdot \xi^2(x,y)  >\eps}]=0,\quad \text{ for each } \eps>0.
 \end{equation}
 \end{enumerate}
 The first condition is a reformulation of \eqref{eq_SPDE_covariance}, and therefore, it is already proven.
 For the Lindeberg condition, note that by its definition \eqref{eq_4_point_relation}, $|\xi(x,y)|$ is uniformly bounded by $C/L$ for a deterministic constant $C$.
 Thus, the indicator $I_{ \xi^2(x,y) L >\eps}$ becomes empty as $L\to\infty$, and the expression \eqref{eq_x59} vanishes for large $L$.
 The asymptotic Gaussianity follows, and the proof of Theorem \ref{Conjecture_general_CLT} is complete.

\section{Low density limit}

The Law of Large Numbers of Section \ref{Section_LLN_4p} and the  Central Limit Theorem
of Section \ref{Section_CLT_conjecture} admit a low density degeneration in which the asymptotic
equations become \emph{linear}. The degeneration is explained in this section.

\smallskip

We still work in the asymptotic regime \eqref{eq_limit_regime}, but we change the asymptotic
behavior of the boundary conditions $H(x,0)$ and $H(0,y)$, as compared to Theorems
\ref{Theorem_LLN_general} and \ref{Conjecture_general_CLT}. We introduce a new parameter
$0<\delta<1$ and assume that $H(Lx,0)$ and $H(0,Ly)$ grow proportionally to $L^{1-\delta}$. This
means that there are much fewer paths entering the quadrant from the bottom and from the left. Hence,
the density of lines everywhere in the quadrant would stay low and tend to $0$ as
$L\to\infty$.

\begin{theorem}
\label{Theorem_LLN_CLT_low}
 Fix $a,b>0$, and $0<\delta<1$. Take two continuous monotone functions $\chi:[0,a]\to\mathbb R$, $\psi:[0,b]\to \mathbb R$
 such that $\chi(0)=\psi(0)$. Suppose that the boundary condition in the stochastic six-vertex model
 is chosen so that as $L\to\infty$, $ L^{\delta-1} H(Lx,0)\to
 \chi(x)$ and $L^{\delta-1} H(0,Lx)\to \psi(y)$ uniformly on $(x,y)\in[0,a]\times[0,b]$.

 Define the function $\h:[0,a]\times[0,b]\to\mathbb R$ as the solution to the PDE
  \begin{equation}
 \label{eq_limit_shape_PDE_low}
   \h_{xy}+\beta_2  \h_{x}
  +\beta_1 \h_y=0, \quad x,y\ge 0;\qquad \h(x,0)=\chi(x,0),\quad \h(0,y)=\psi(y,0),
 \end{equation}
 and a random field $\phi:[0,a]\times[0,b]\to\mathbb R$ as a solution to
\begin{equation}
 \label{eq_SDE_low}
  \phi_{xy}+\beta_1\phi_{y}
  + \beta_2 \phi_x
  = \eta  \cdot \sqrt{
    \beta_1\, \h_y-\beta_2\,\h_x
   }
\end{equation}
with zero boundary conditions $\phi(x,0)=\phi(0,y)=0$, where $\eta$ is the two--dimensional white
noise.
 Then the height function $H(x,y)$ of the stochastic six-vertex model ($\alpha=0$) satisfies (for $(x,y)\in[0,a]\times [0,b]$)
 \begin{equation}
 \label{eq_expectation_convergence_low}
  \lim_{L\to\infty} \E  \frac{H(Lx,Ly)}{L^{1-\delta}}=\h(x,y),
 \end{equation}
 \begin{equation}
   \label{eq_fluctuation_convergence_low}
   \lim_{L\to\infty} \frac{H(Lx,Ly)-\E H(Lx,Ly)}{\sqrt{L^{1-\delta}}}=\phi(x,y).
 \end{equation}
\end{theorem}

%\begin{remark} We prove the convergence in \eqref{eq_fluctuation_convergence_low} in the sense of %finite--dimensional distributions. It probably can be upgraded to the functional convergence in distribution.
%
%The second order corrections in \eqref{eq_expectation_convergence_low} can be also readily accessed by %our tools.
%\end{remark}

Let us present an interpretation of Theorem \ref{Theorem_LLN_CLT_low}. Consider an
$L^{1-\delta}\times L^{1-\delta}$ box inside $[1,LX]\times [1,LY]$. The height function $H(x,y)$ changes by a constant
when we cross the box and, therefore, there are finitely many paths inside. Each path has rare
turns and, as $L\to\infty$, it turns into a \emph{persistent Poisson random walk}:

\noindent$\bullet$ Whenever a path travels to the right, it turns upwards with intensity $\beta_1$,
$\bullet$ whenever a path travels upwards, it turns to the right with intensity $\beta_2$.

Recall that the paths were interacting with each other through the non--intersecting condition. Let
us now change the way we view the vertices of type $V$ of Figure \ref{Fig_weights}: instead of thinking that
paths touch each other, let us imagine that we observe an \emph{intersection} of vertical
and horizontal paths. Now paths simply do not feel each other; the only
interaction is that whenever paths intersect, they cannot turn at exactly the same moment.
However, since intersections are rare, this interaction is negligible as $L\to\infty$. We conclude
that in an $L^{1-\delta}\times L^{1-\delta}$ box the configuration as $L\to\infty$ is probabilistically
indistinguishable from a collection of \emph{independent} persistent Poisson random walks. Gluing
together all $L^{1-\delta}\times L^{1-\delta}$ boxes, we conclude that the entire configuration in
$[1,LX]\times [1,LY]$ looks like that.

 Thus, Theorem \ref{Theorem_LLN_CLT_low} can be
treated as the Law of Large Numbers and Central Limit Theorem for the height function of a
collection of independent persistent Poisson random walks with prescribed densities of entry points on the boundary of the quadrant. We find it somewhat surprising that the
stochastic PDE \eqref{eq_SDE_low} appears in such a simple setup. It should be possible to prove
this Poisson result directly without appealing to the discretization provided by the six-vertex
model, but we leave this question out of the scope of the article.

The proof of Theorem \ref{Theorem_LLN_CLT_low} is similar to those of Theorems \ref{Theorem_LLN_general}, \ref{Conjecture_general_CLT}, the details are presented  in the appendix.

\section{Appendix A: Proof of Theorem \ref{Theorem_LLN_CLT_low}}
\label{Section_Low_density_proof}

Theorem \ref{Theorem_4_point} written in terms of $q^H-1$ and combined with Theorem \ref{Theorem_discrete_PDE_through_integrals} implies that
\begin{multline}
\label{eq_x25}
  q^{H(X,Y)}-1=
  \sum_{y=1}^Y \Rd(X,Y;0,y) \bigl[ (q^{H(0,y)}-1)-b_2 (q^{H(0,y-1)}-1)\bigr]\\+
  \sum_{x=1}^X \Rd(X,Y;x,0) \bigl[ (q^{H(x,0)}-1)-b_1 (q^{H(x-1,0)}-1)\bigr]
  \\ +\sum_{x=1}^X \sum_{y=1}^Y \Rd(X,Y;x,y) \xi(x,y).
\end{multline}

The first two terms of the right--hand side of \eqref{eq_x25} are deterministic and give $\E(q^H-1)$, while the third one is responsible for the fluctuations. Resuming \eqref{eq_x25} and using $q^{H(0,0)}=1$, we obtain
\begin{align}
\label{eq_x26}
  \E&[q^{H(X,Y)}-1]\\=&\notag
  \Rd(X,Y;0,Y)(q^{H(0,Y)}-1)\\ &+ \notag\sum_{y=1}^{Y-1} [\Rd(X,Y;0,y)-b_2 \Rd(X,Y;0,y+1)]  (q^{H(0,y)}-1)\\ \notag
    &+  \Rd(X,Y;X,0)(q^{H(X,0)}-1) \\ &+\notag \sum_{x=1}^{X-1} [\Rd(X,Y;x,0)-b_1 \Rd(X,Y;x+1,0)] (q^{H(x,0)}-1).
\end{align}
We now pass to the limit $L\to\infty$ in \eqref{eq_x26}. For that note the deterministic inequality
$$
 |H(x,y)|\le |H(La,0)|+|H(0,Lb)|, \quad 0\le x \le L a, \, 0\le x \le L b,
$$
which implies
\begin{equation}
\label{eq_x28}
 q^{H(x,y)}-1=\ln(q) H(x,y)+ O\bigl( [\ln(q) H(x,y)]^2\bigr)=\ln(q) H(x,y) +O (L^{-2\delta}).
\end{equation}
In addition, with the notation of Section \ref{Section_PDEs},
$$
 \lim_{L\to\infty} \Rd(LX,LY;Lx,Ly)=\R(X,Y;x,y),
$$
\begin{multline*}
  \lim_{L\to\infty} L (\Rd(LX,LY;Lx,Ly)- b_2 \Rd(LX,LY;Lx,Ly+1))\\=\beta_2 \R(X,Y;x,y)- \R_y(X,Y;x,y),
\end{multline*}
\begin{multline*}
  \lim_{L\to\infty} L (\Rd(LX,LY;Lx,Ly)- b_1 \Rd(LX,LY;Lx+1,Ly))\\=\beta_1 \R(X,Y;x,y)- \R_x(X,Y;x,y).
\end{multline*}
We conclude that
\begin{multline}
\label{eq_x27}
  \lim_{L\to\infty} \E \frac{H(LX,LY)}{L^{1-\delta}}\\=   \R(X,Y;0,Y) \h(0,Y)+
  \int_0^Y [\beta_2 \R(X,Y;0,y)- \R_y(X,Y;0,y)] \h(0,y)dy\\
    +  \R(X,Y;X,0)\h(X,0)+ \int_0^X [\beta_1 \R(X,Y;x,0)- \R_x(X,Y;x,0)] \h(x,0)dx.
\end{multline}
When integrated by parts, \eqref{eq_x27} matches the formula of Theorem \ref{Theorem_general_solution} for the solution to \eqref{eq_limit_shape_PDE_low}.

\smallskip

Thus, \eqref{eq_expectation_convergence_low} is proved and we proceed to \eqref{eq_fluctuation_convergence_low}.  Using \eqref{eq_x25} we have
\begin{equation}
\label{eq_x35}
  q^{H(LX,LY)}-\E q^{H(LX,LY)}= \sum_{x=1}^{LX} \sum_{y=1}^{LY} \Rd(LX,LY;x,y) \xi(x,y).
\end{equation}
The remaining proof proceeds in the following two steps: we first show that the finite--dimensional distributions of \eqref{eq_x35} converge to those of the Gaussian process $(\beta_1-\beta_2) \phi(X,Y)$, and then deduce the limit for the centered height function $H(LX,LY)$ as a corollary. In fact, in the first step we will detail only one--point convergence; the convergence of any finite--dimensional distributions is proven in the same way by invoking the multi--dimensional Central Limit Theorem instead of the one--dimensional version (cf.\ the proof of Theorem \ref{Conjecture_general_CLT} above).

\medskip

Let us investigate the variance of the right--hand side of \eqref{eq_x35} as $L\to\infty$. From \eqref{eq_4_point_no_correlation}, \eqref{eq_4_point_covariance} the variance equals
\begin{multline}
\label{eq_x29}
 \sum_{x=1}^{LX} \sum_{y=1}^{LY} \Rd(LX,LY;x,y)^2\\ \times  \E \Biggl[  \bigr(qb(1-b)+b(1-qb)\bigl)
 (q^{H(x,y)}-q^{H(x-1,y)}) (q^{H(x,y)}-q^{H(x,y-1)})
  \\ + b(1-qb)(1-q) q^{H(x,y)} (q^{H(x,y)}-q^{H(x-1,y)})
 \\  - b(1-b)(1-q) q^{H(x,y)} (q^{H(x,y)}-q^{H(x,y-1)}) \Biggr].
 \end{multline}
 We split \eqref{eq_x29} into two parts: the leading contribution and vanishing terms. The former is
 given by the third and fourth lines with $L\to\infty$ approximations $q^H\approx 1$ and $q^{H(x,y)}-q^{H(x-1,y)}\approx \ln(q) (H(x,y)-H(x-1,y))$:
\begin{multline}
\label{eq_x30}
 b(1-qb)(1-q) \ln(q) \sum_{x=1}^{LX} \sum_{y=1}^{LY} \Rd(LX,LY;x,y)^2\, \E \Bigl[    H(x,y)-H(x-1,y)\Bigr]
\\-  b(1-b)(1-q) \ln(q) \sum_{x=1}^{LX} \sum_{y=1}^{LY} \Rd(LX,LY;x,y)^2\, \E \Bigl[
  H(x,y)-H(x,y-1)\Bigr].
 \end{multline}
 We sum by parts in \eqref{eq_x30} and compute the limit $L\to\infty$. For the first sum we get
 \begin{multline}
 \label{eq_x14}
 b(1-qb)(1-q) \ln(q) \\ \times \Biggl[\sum_{x=1}^{LX} \sum_{y=1}^{LY} [\Rd(LX,LY;x,y)^2-\Rd(LX,LY;x+1,y)^2]\E H(x,y)
 \\+\sum_{y=1}^{LY}  \Rd(LX,LY;LX+1,y)^2\, \E H[LX,y]-\sum_{y=1}^{LY}  \Rd(LX,LY;1,y)^2\, \E H(0,y)\Biggr].
 \end{multline}
 The explicit formula \eqref{eq_Discrete_R} implies that  $L(\Rd(LX,LY;Lx,Ly)^2-\Rd(LX,LY;Lx+1,Ly)^2) \to -\frac{\partial}{\partial x} \R^2(X,Y;x,y)$ as $L\to\infty$. Combining with \eqref{eq_expectation_convergence_low}, we obtain the $L\to\infty$ asymptotics of \eqref{eq_x14}:
  \begin{multline}
  \label{eq_x31}
  L^{-1-\delta}  \beta_2 (\beta_2-\beta_1)^2  \Biggl[\int_0^X\int_0^Y \left(-\frac{\partial}{\partial x} \R^2(X,Y;x,y) \right)  \h(x,y) dxdy
 \\+\int_0^Y \R(X,Y;X,y)^2 \h[X,y]dy-\int_0^Y \R(X,Y;0,y)^2 \h(0,y)dy\Biggr].
 \end{multline}
 We further integrate by parts in \eqref{eq_x31} and do the same computation for the second sum in \eqref{eq_x30}. The final result is
 \begin{equation}
 \label{eq_x34}
  L^{-1-\delta} (\beta_2-\beta_1)^2 \int_0^X\int_0^Y \R^2(X,Y;x,y) \bigl(\beta_2\h_x(x,y)-\beta_1\h_y(x,y)\bigr) dxdy.
 \end{equation}
 Note that this is precisely the variance of $(\beta_1-\beta_2)\phi(X,Y)$, when we use Theorem \ref{Theorem_general_solution} to solve \eqref{eq_SDE_low}.

 The next step is to show that the remaining terms in \eqref{eq_x29} indeed do not contribute to the leading asymptotic behavior. We start from the second line in \eqref{eq_x29}. Note that $ (q^{H(x,y)}-q^{H(x-1,y)}) (q^{H(x,y)}-q^{H(x,y-1)})\le 0$ and $\Rd$ is uniformly bounded as $L\to\infty$ (because it converges to $\R$). Thus, the absolute value of the first line in \eqref{eq_x29} is bounded by (here $C$ is a positive constant) \begin{equation}
 \label{eq_x32}
  \frac{C}{L} \E \sum_{x=1}^{LX} \sum_{y=1}^{LY}  (q^{H(x-1,y)}-q^{H(x,y)}) (q^{H(x,y)}-q^{H(x,y-1)}).
 \end{equation}
 Note that the $(x,y)$--summand is non-zero if and only if both $H(x-1,y)=H(x,y)+1$ and $H(x,y-1)=H(x,y)$. In other words, this happens if the vertex at $(x,y)$ has type $II$ (cf.\ Figure \ref{Fig_weights}). We conclude that \eqref{eq_x32} is bounded from above by
 \begin{equation}
 \label{eq_x33}
  \frac{C'}{L^3} \E \bigl(\text{number of vertices of type }II\text{ inside }[1,LX]\times [1,LY] \bigr).
 \end{equation}
 We proceed to bound this expectation. For that let us first bound the expected number of vertices of types $V$ and $VI$ (corners). Let us denote the latter number by $\mathcal N$. Note that we have $O(L^{1-\delta})$ paths entering into $[1,LX]\times [1,LY]$  from the left or from below. Each path has $O(L^{-1})$ vertices, and at each of these vertices with probability at most $1-b_1$ or $1-b_2$ a corner might occur. We conclude that there are $O(1)$ corners along each path. It follows that $\E \mathcal N=O(L^{1-\delta})$ and $\E \mathcal N^2=O(L^{2-2\delta})$. Next note that each vertex of type $II$ must belong to a column (vertical line of fixed $x$--coordinate) in which either a path enters into the quadrant from below or there is a corner in this column. For the same reason, each vertex of type $II$ must belong to a row with similar properties. Since the number of both such rows and columns is $O(L^{1-\delta})$, we conclude that the number of vertices of type $II$ is $O( L^{1-\delta}\cdot L^{1-\delta})$. Plugging into \eqref{eq_x33} we get
 $$
    \frac{C'}{L^3} O( L^{1-\delta}\cdot L^{1-\delta})= O(L^{-1-2\delta}),
 $$
 which is of lower order than the leading term of \eqref{eq_x29}. The justification of the fact that the remainder terms that were left out when passing from \eqref{eq_x29} to \eqref{eq_x30} is straightforward and we omit it.

 We have computed the asymptotic variance of \eqref{eq_x35} and now proceed to showing the asymptotic Gaussianity. Let us linearly order the integer points inside the rectangle $[1,LX]\times [1,LY]$ as follows: $(1,1)$, $(2,1)$, $(1,2)$, $(3,1)$, $(2,2)$, $(1,3)$, $(4,1)$, $(3,2)$, $(2,3)$, $(1,4)$,\dots, i.e., we sequentially trace the diagonals $x+y=\mathrm{const}$. Theorem \ref{Theorem_4_point} implies that then
 $\Rd(LX,LY;x,y) \xi(x,y)$ is then a martingale difference in $(x,y)$, and we can apply the Martingale Central Limit Theorem, see, e.g., \cite[Section 3]{HH}. There are two conditions to check:
 \begin{enumerate}
  \item The conditional variance, which by Theorem \ref{Theorem_4_point} is given by (the expression below differs from \eqref{eq_x29} by the absence of the expectation)
  \begin{multline}
\label{eq_x36}
 \sum_{x=1}^{LX} \sum_{y=1}^{LY} \Rd(LX,LY;x,y)^2 \\ \times  \Biggl[  \bigr(qb(1-b)+b(1-qb)\bigl)
 (q^{H(x,y)}-q^{H(x-1,y)}) (q^{H(x,y)}-q^{H(x,y-1)})
  \\ + b(1-qb)(1-q) q^{H(x,y)} (q^{H(x,y)}-q^{H(x-1,y)})
   \\ - b(1-b)(1-q) q^{H(x,y)} (q^{H(x,y)}-q^{H(x,y-1)}) \Biggr],
 \end{multline}
should have the same $L\to\infty$ behavior as the unconditional variance \eqref{eq_x29}, in the sense that the ratio tends to $1$ in probability.
 \item The Lindeberg condition should hold, which in our setting reads
 \begin{equation}
 \label{eq_x37}
  \lim_{L\to\infty} \sum_{x=1}^{LX} \sum_{y=1}^{LY} \E \Bigl[ \xi^2(x,y) L^{1+\delta} I_{ \xi^2(x,y) L^{1+\delta} >\eps}]=0,\quad \text{ for each } \eps>0.
 \end{equation}
 \end{enumerate}
 For the first condition note that since we already know the decay of variance in \eqref{eq_x29}, we can infer that $L^{1-\delta} H(Lx,Ly)\to \h(x,y)$ in probability.  Since $H$ is a monotone function in each of its variables, the one--point convergence further implies the convergence to $\h$ as a continuous function of two variables in the supremum norm. Then the same argument as for \eqref{eq_x29} goes through and we obtain the same asymptotics \eqref{eq_x34} for \eqref{eq_x36} as for \eqref{eq_x29}.

 For the Lindeberg condition note that by its definition \eqref{eq_4_point_relation}, $|\xi(x,y)|$ is uniformly bounded by $C/L$ for a deterministic constant $C$.
 Thus, the indicator $I_{ \xi^2(x,y) L^{1+\delta} >\eps}$ becomes empty as $L\to\infty$, and the expression \eqref{eq_x37} vanishes for large $L$.

 The asymptotic Gaussianity follows, and we have thus shown the following convergence in finite--dimensional distributions:
 \begin{equation}
\label{eq_x38}
  \lim_{L\to\infty} L^{\frac{1+\delta}{2}}\left[q^{H(LX,LY)}-\E q^{H(LX,LY)}\right]=(\beta_1-\beta_2) \phi(X,Y).
\end{equation}
 It remains to deduce the same convergence for centered and rescaled $H(LX,LY)$. For that we write
 \begin{multline}
 \label{eq_x39}
  q^{H(LX,LY)}=q^{\E H(LX,LY)} q^{H(LX,LY)-\E H(LX,LY)}\\ =q^{\E H(LX,LY)}  \sum_{n=0}^\infty
  \frac{ \bigl[\ln(\q) (H(LX,LY)-\E H(LX,LY))\bigr]^n}{n! L^{n}}.
 \end{multline}
 Since $\ln(\q) H(LX,LY)/L$ is bounded by a deterministic constant, the series in \eqref{eq_x39} is uniformly convergent, and $q^{H(LX,LY)}-\E q^{H(LX,LY)}$ is the centered version of the same series:
 \begin{multline}
 \label{eq_x40}
 q^{\E H(LX,LY)}   \sum_{n=1}^\infty \Biggl(
  \frac{ \bigl[\ln(\q) (H(LX,LY)-\E H(LX,LY))\bigr]^n}{n!L^n}  \\ -\frac{\E \bigl[\ln(\q) (H(LX,LY)-\E H(LX,LY))\bigr]^n}{n!L^n} \Biggr).
 \end{multline}
 As $L\to\infty$, the prefactor $q^{\E H(LX,LY)}$ tends to $1$, the first term in the series is
 $$
  \frac{\ln(\q)}{L} (H(LX,LY)-\E H(LX,LY)),
 $$
 and the following terms are of lower orders. Since $\ln(\q)=\beta_1-\beta_2$, \eqref{eq_x38} now implies
 $$
   \lim_{L\to\infty} L^{\frac{1+\delta}{2}} \frac{\beta_1-\beta_2}{L} (H(LX,LY)-\E H(LX,LY))=(\beta_1-\beta_2)\phi(X,Y),
 $$
 and the proof of Theorem \ref{Theorem_LLN_CLT_low} is complete.

\section{Appendix B: Theorem \ref{Conjecture_general_CLT} through a variational principle and contour integrals}

\label{Section_CLT_through_var}

In this section we provide an alternative arguments towards the validity of Theorem \ref{Conjecture_general_CLT}. This is not a rigorous proof,  only heuristics.

This approach to Theorem \ref{Conjecture_general_CLT} was inspired by
\cite[Appendix]{BD}. In a sense, we develop (non-rigorously) a version of the local variational
principle for the stochastic six-vertex model in the limit regime \eqref{eq_limit_regime}. It would
be interesting to see whether this variational principle can be applied to other situations. For
the computations we rely on contour integral formulas of \cite{Ag}.

\bigskip

We start by considering another integrable case of boundary conditions for the stochastic six--vertex
model that generalizes domain wall boundary conditions of Section \ref{Section_domain_wall}.

At each point of the $y$--axis we flip an independent coin. It comes heads with probability $p_1$,
and in such a case we place a path entering from the left at this point. Otherwise, there is no path.
Similarly, for each point of the $x$ axis we flip a coin which comes heads with probability $p_2$
to create paths entering from the bottom. \cite{Ag} develops proves a multiple contour integral formula for the joint
moments of $q^{H}$ in this situation, generalizing the $\alpha=0$ case of Theorem
\ref{Theorem_observable}.  The formulas are quite similar and only differ by simple rational factors.

In particular, \cite[(3.13), (3.19)]{Ag} yields
\begin{multline}
\label{eq_expectation_Bernoulli} \E q^{n \cdot H(x,y)}= \bigl(\rho_1^{-1}\rho_2
\s^{-1}q^{-n};q\bigr)_n \frac{q^{n(n-1)/2}}{(2\pi \ii)^n} \oint \dots \oint
\prod_{1\le i<j\le n}
 \frac{z_i-z_j}{z_i-qz_j}\\ \times \prod_{i=1}^n \left[ \left( \frac{1+ q^{-1} \frac{1-b}{1-qb}  z_i}{1+ \frac{1-b}{1-qb} z_i}   \right)^{x-1}
 \left( \frac{1+z_i}{1+q^{-1}z_i} \right)^{y} \frac{1}{\bigl(1-q^{-1} \rho_1^{-1}z_i\bigr)\,\bigl(z_i-\rho_2 \frac{1-qb}{1-b}\bigr)} d z_i \right],
\end{multline}
where $n\ge 1$, $\rho_i=\frac{p_i}{1-p_i}$, and the contours have two parts: the first ones are \emph{nested} around
$\{\frac{1-qb}{1-b}\rho_2\}$, and the second ones all coincide with a tiny circle around $-q$. The contours avoid
singularities at $-\frac{1-qb}{1-b}$ and at $\rho_1 q$. In \cite{Ag} the formula \eqref{eq_expectation_Bernoulli} is proven in the case $\rho_1^{-1}\rho_2
\s^{-1}q^{-n}<1$; for other values of parameters, one needs to make an analytic continuation in $\rho_1$, $\rho_2$ of both sides in \eqref{eq_expectation_Bernoulli}.

The following statement is a simple
corollary of \eqref{eq_expectation_Bernoulli}, extending Theorem \ref{Theorem_LLN} and matching the
computations of Section \ref{Section_consistency}.
\begin{proposition} \label{Proposition_LLN_Bernoulli}  In the regime \eqref{eq_limit_regime}, with the Bernoulli boundary conditions as described above,
$\frac{1}{L}H(Lx,Ly)$ converges to $\h(x,y)$ given by
 \begin{multline}
 \label{eq_LLN_Bernoulli}
    \q^{\mathbf h(x,y)}
    \\ =\frac{1}{2\pi \ii} \oint_{-1} \exp\left(
      \ln(\q)\left(- x \frac{\s z}{1+ \s z }   + y\frac{z}{1+z} \right)\right)
   \left(\frac{1}{\rho_1-z} +\frac{1}{z-\rho_2
  \s^{-1}}\right)
       d z \\ +
       \exp\left(
      \ln(\q)\left(- x \frac{\rho_2 }{1+ \rho_2 }   + y\frac{\rho_2 \s^{-1}}{1+\rho_2 \s^{-1}} \right)\right),
 \end{multline}
 with positively oriented integration contour that encircles only the singularity at $z=-1$.
\end{proposition}
\begin{remark}
 When $\rho_1=\rho_2 \s^{-1}$, the distribution of the system in translationally invariant, see \cite{Ag}. This matches \eqref{eq_LLN_Bernoulli} turning into $ \q^{\mathbf h(x,y)}=
      \q^{- x p_2+ yp_1}$.
\end{remark}

An important quantity for us is the second mixed derivative of \eqref{eq_LLN_Bernoulli} at $0$:
\begin{equation}
 M^{\eps}(x,y):=\q^{\mathbf h(\eps x,\eps y)}-\q^{\mathbf h(\eps x,0)}-\q^{\mathbf h(0,\eps y)}+\q^{\mathbf h(0,0)}.
\end{equation}
Direct computation shows that, as $\eps\to 0$,
\begin{equation}
\label{eq mean_compute}
 M^{\eps}(x,y)=\eps^2 xy \ln^2(\q) \frac{p_1\s-p_2}{1-\s}+o(\eps^2)=\eps^2 xy (\beta_2-\beta_1) (p_1\beta_1-p_2\beta_2)+o(\eps^2).
\end{equation}

The computation \eqref{eq_expectation_Bernoulli} admits an extension to joint $q$--moments for
several points $(x,y)$, that lie on the same vertical or same horizontal line, similarly to
Theorem \ref{Theorem_observable}. We can even reach the collections of points on more general
monotone paths:
\begin{equation}
\label{eq_monotone_section}
(x_1,y_1),\, (x_2,y_2),\dots, (x_k,y_k):\quad
x_1\ge x_2\ge\dots\ge x_k,\quad y_1\le
y_2\le \dots\le y_k;
\end{equation}
for the domain wall boundary conditions this was done in \cite{BBW}, and here the situation is analogous.

It is very plausible that arguing similarly to the proof of CLT in Section \ref{Section_domain_wall}, one can reach the following statement.
\begin{claim} \label{Claim_CLT_Bernoulli_multi}
 For the stochastic six-vertex model with Bernoulli boundary conditions as described above,
  as $L\to\infty$ in the regime \eqref{eq_limit_regime},
  $L^{1/2} \bigl(q^{H(Lx,Ly)}-\E q^{H(Lx,Ly)}\bigr)$ converges to a Gaussian random variable (jointly over
monotone sections \eqref{eq_monotone_section}) with variance given for $x_1\ge x_2$, $y_1\le y_2$ by
\begin{multline}
\label{eq_covariance_Bernoulli_1} \lim_{L\to\infty} L  \biggl(\E(q^{H(Lx_1,Ly)}
q^{H(Lx_2,Ly)}- \E q^{H(Lx_1,Ly_1)}\E q^{H(Lx_2,Ly_2)}) \biggr)\\=
 \frac{\ln(\q)}{(2\pi\ii)^2}  \oint \oint
 \frac{z_1 \rho_1-z_2\rho_2 \s^{-1}}{(z_1-z_2)(\rho_1-\rho_2
\s^{-1})} \\ \times \prod_{i=1}^2 \left[ \exp\left(
      \ln(\q)\left(- x_i \frac{\s z_i}{1+ \s z_i }   + y_i\frac{z_i}{1+z_i} \right)\right)
   \left(\frac{1}{\rho_1-z_i} +\frac{1}{z_i-\rho_2
  \s^{-1}}\right)
       d z_i\right],
\end{multline}
where the integration goes in positive direction around the singularities at $-1$ and at $\rho_2
\s^{-1}$, and $z_1$ is inside $z_2$.
\end{claim}
\begin{remark}
 The right--hand side  of \eqref{eq_covariance_Bernoulli_1} depends on $\rho_1$, $\rho_2$ in an analytic way; in order to continue through the line $\rho_1=\rho_2 \s^{-1}$, one should split $z_1$ and $z_2$ integrals into two parts: enclosing $-1$ and enclosing $\rho_2 \s^{-1}$. The latter part can then be explicitly computed.
\end{remark}

Let $\tilde h(x,y)$ denote the limiting Gaussian field of Claim
\ref{Claim_CLT_Bernoulli_multi}. We are interested in the following mixed difference:
\begin{equation}
\label{eq_second_difference}
 D^\eps(x,y):=\tilde h(\eps x,\eps y)+\tilde h(0,0) - \tilde h(\eps x,0) -\tilde h(0,\eps y).
\end{equation}
Note that $\tilde h(0,0)=0$, but we still add it to the formula in order to emphasize the
structure. Claim \ref{Claim_CLT_Bernoulli_multi} implies that $D^\eps(x,y)$ is
Gaussian, and we would like to find its variance as $\eps\to 0$. We compute
\begin{multline}
\label{eq_x15}
 \mathrm{Var}(D^\eps(x,y))=
 \mathrm{Cov}(\tilde h(\eps x,\eps y),\tilde h(\eps x,\eps y))\\ +
  \mathrm{Cov}(\tilde h(\eps x,0),\tilde h(\eps x,0))+
   \mathrm{Cov}(\tilde h(0,\eps y),\tilde h(0,\eps y))
 \\ -2 \mathrm{Cov}(\tilde h(\eps x,\eps y),\tilde h(\eps x,0))
 -2 \mathrm{Cov}(\tilde h(\eps x,\eps y),\tilde h(0,\eps y))
 +2 \mathrm{Cov}(\tilde h(\eps x,0),\tilde h(0,\eps y)),
\end{multline}
where the last term vanishes, as the boundary values are independent. We use the expression of
Claim \ref{Claim_CLT_Bernoulli_multi} for each term of \eqref{eq_x15}, expand the
exponentials in series in $\eps$, and compute the integrals as residues. Simplifying the result and expressing it in terms of $p_1$, $p_2$ we
get
\begin{multline}
\label{eq_Var_compute}
  \mathrm{Var}\bigl[D^\eps(x,y)\bigr]=-\eps^2 x y \ln^3(\q)\, \frac{-p_1 p_2 (\s+1)+p_1\s+p_2}{1-\s}+o(\eps^2)\\ = \eps^2 x y (\beta_2-\beta_1)^2
  \bigl(-p_1 p_2 (\beta_1+\beta_2)+p_1\beta_1+p_2\beta_2 \bigr)+o(\eps^2).
\end{multline}
Note that the individual terms in the definition of $D^{\eps}(x,y)$ have much greater variance. For
instance, $\mathrm{Var} \tilde h(\eps x,0)= \eps x p_2(1-p_2)$ due to the conventional CLT for sums of
independent Bernoulli random variables. However, mixed difference leads to cancelations, and
\eqref{eq_Var_compute} has variance of order $\eps^{2}$ rather than $\eps$.

\begin{proof}[Heuristic proof of Theorem \ref{Conjecture_general_CLT}]

 Fix small $\eps>0$ and consider the values of the height function $H$ at points $(\eps i,\eps
 j)$, $i,j=1,2,\dots$ inside a fixed
 $[0,A]\times [0,B]$ rectangle.

 We would like to compute the conditional distribution of $q^{H(\eps L (i+1),\eps L (j+1))}$ given
 $q^{H(\eps L i,\eps L j)}$, $q^{H(\eps L (i+1),\eps L j)}$, $q^{H(\eps L i,\eps L (j+1))}$.

At this moment we will make a non-rigorous step, approximating the system in an $\eps L \times \eps L$ square by the system with Bernoulli boundary conditions as in Proposition \ref{Proposition_LLN_Bernoulli}, Claim \ref{Claim_CLT_Bernoulli_multi} in a similarly sized square. Therefore, we say that when $\eps$ is small and $L$ is large, the horizontal lines crossing the vertical segment between points
 $(\eps L i, \eps L j)$ and $(\eps L i, \eps L (j+1))$ become Bernoulli--distributed with parameter
 $$
  p_1\approx \frac{H(\eps L i, \eps L (j+1)) - H(\eps L i, \eps L j)}{\eps L}.
 $$
 The vertical lines crossing the horizontal segment between points $(\eps Li, \eps Lj)$ and $(\eps
 L(i+1),\eps L(j))$ also become Bernoulli--distributed with parameter
 $$
  p_2\approx \frac{H(\eps L i, \eps L j)) - H(\eps L (i+1), \eps L j)}{\eps L}.
 $$

 At this point we can use Claim \ref{Claim_CLT_Bernoulli_multi},
 which will give us the conditional distribution as a Gaussian law. Shortening the notations as $h_{ij}=H(\eps L i,\eps Lj)$, we
 write
 \begin{multline}
 \label{eq_Law_factor}
  \mathrm{Prob}\Bigl(q^{h_{i+1,j+1}} \mid q^{h_{i,j}}, q^{h_{i+1,j}}, q^{h_{i,j+1}}\Bigr)
   \\ \approx \frac{1}{\sqrt{2\pi \eps^2 L \,V[p_2,p_1]}} \\ \times
   \exp\left(-\frac{\left(q^{h_{i+1,j+1}}- q^{h_{i+1,j}}-q^{h_{i,j+1}}+
  q^{h_{i,j}}-L\eps^2 M(p_1,p_2) \right)^2}{2 \eps^2 L\, V[p_1,p_2]}
  \right),
 \end{multline}
 where $\eps^2 M(p_1,p_2)$ is $\q^\h$ multiplied by the leading $\eps\to 0$ term of the expression \eqref{eq mean_compute} with $x=y=1$, and $\eps^2 V[p_1,p_2]$ is $\q^{2\h}$ multiplied by the leading $\eps\to 0$ term of the expression \eqref{eq_Var_compute} with $x=y=1$. The multiplication by $\q^\h$ and $\q^{2\h}$ appears because of the height function at the origin was zero in Proposition \ref{Proposition_LLN_Bernoulli} and Claim \ref{Claim_CLT_Bernoulli_multi}, while we need the value $h_{ij}$ here.

At this point we can multiply \eqref{eq_Law_factor} over all $i,j$ to get the joint
law of $h_{i,j}$, $i,j=1,2,\dots$. Implicitly we use the Markovian structure
of the stochastic six--vertex model here.

Now let us analyze various parts of \eqref{eq_Law_factor}. Recall that as
$L\to\infty$, $q^{H(Lx,Ly)}$ approximates a smooth profile $\q^\h(x,y)$ plus $\frac{1}{\sqrt{L}}$ multiplied
by the fluctuation field $\phi(x,y)$ as in Theorem \ref{Conjecture_general_CLT}. Then we have
$$
 p_1\approx \frac{\partial}{\partial y} \frac{1}{L} H(Lx,Ly)\approx \frac{\q^{\h}_y+L^{-1/2} \phi_y}{\ln(\q) \q^\h},
$$
$$
p_2\approx - \frac{\partial}{\partial x} \frac{1}{L} H(Lx,Ly)\approx-\frac{\q^{\h}_x+L^{-1/2} \phi_x}{\ln(\q) \q^\h}.
$$
$$
q^{h_{i+1,j+1}}- q^{h_{i+1,j}}-q^{h_{i,j+1}}+
  q^{h_{i,j}}\approx \q^\h_{xy} \eps^2 L+ \phi_{xy}(\eps i,\eps j) \eps^2  L^{1/2}.
$$

Therefore, plugging in the expression for $M[p_1,p_2]$,
the joint law of all $h_{i,j}$ can be approximated as
\begin{multline}
\label{eq_Law_joint}
 \prod_{i,j}  \left(2\pi \eps^2 L \,V\left[\frac{\q^{\h}_y}{\ln(\q) \q^\h},-\frac{\q^{\h}_x}{\ln(\q) \q^\h}\right]\right)^{-1/2}
 \\ \times \exp\left( - L \eps^2 \frac{(\q^\h_{xy} +\beta_1 \q^\h_y+\beta_2 \q^\h_x+ L^{-1/2}(\phi_{xy}+\beta_1 \q^\h_y+\beta_2\q^\h_x)^2}{2 V\left[\frac{\q^{\h}_y}{\ln(\q) \q^\h},-\frac{\q^{\h}_x}{\ln(\q) \q^\h}\right]}\right),
\end{multline}
where in $(i,j)$th term all functions are evaluated at the point $(x,y)=(\eps i,\eps j)$.

Theorem \ref{Theorem_LLN_general} says that $\q^\h_{xy} +\beta_1 \q^\h_y+\beta_2 \q^\h_x$ in \eqref{eq_Law_joint} vanishes.\footnote{Alternatively, one can use the leading exponential part of \eqref{eq_Law_joint} to \emph{show} that $\q^\h_{xy} +\beta_1 \q^\h_y+\beta_2 \q^\h_x=0$.}
Plugging in the expression for $V[\cdot,\cdot]$, we further approximate the joint law of all $h_{i,j}$ by
\begin{multline}
\label{eq_Law_joint_2}
 \prod_{i,j} \frac{1}{\sqrt{2\pi \eps^2 L \,\bigl(\q^{\h}_y \q^{\h}_x  (\beta_1+\beta_2)+\q^{\h}_x\q^\h\beta_2(\beta_2-\beta_1)-\q^{\h}_y \q^\h \beta_1(\beta_2-\beta_1) \bigr)}}
 \\ \times \exp\left( - \eps^2 \frac{(\phi_{xy}+\beta_1 \q^\h_y+\beta_2\q^\h_x)^2}{2 \bigl(\q^{\h}_y \q^{\h}_x  (\beta_1+\beta_2)+\q^{\h}_x\q^\h\beta_2(\beta_2-\beta_1)-\q^{\h}_y \q^\h \beta_1(\beta_2-\beta_1) \bigr)}\right).
\end{multline}
Note that informally the second line in \eqref{eq_Law_joint_2} approximates as $\eps\to 0$ the exponential of a double integral, which shows that the scalings are chosen in the correct way. On the other hand, it matches Theorem \ref{Conjecture_general_CLT}. Indeed, the numerator in the exponential is the left--hand side of \eqref{eq_SDE_betas}, and the denominator is the same as the (squared) coefficient in the right--hand side. The noise in \eqref{eq_SDE_betas} is Gaussian, as is density in \eqref{eq_Law_joint_2}. Finally, the noise is white (uncorrelated), and \eqref{eq_Law_joint_2} has the product structure over points of the plane manifesting the independence.
\end{proof}

\end{document}